\documentclass[10pt]{article}
\usepackage[T1]{fontenc}
\usepackage[utf8]{inputenc}
\usepackage{amsmath,amsthm,graphicx,mathtools,multicol,array}

\input{insbox}
\usepackage{ifthen}
\usepackage{float}
\usepackage{subcaption}
\usepackage{amssymb}
\usepackage{setspace}
\usepackage{tikz}
\usepackage{relsize}
\usetikzlibrary{patterns,patterns.meta,arrows.meta,calc,3d,perspective,fadings,shadows,cd}
\usepackage[a4paper, width=170mm, top=30mm, bottom=30mm]{geometry}
\usepackage[colorlinks, allcolors=black]{hyperref}
\usepackage{hyperref}

\theoremstyle{plain}
	\newtheorem{thm}{Theorem}[section]
	\newtheorem{cor}[thm]{Corollary}
	\newtheorem{lem}[thm]{Lemma}
	\newtheorem{prop}[thm]{Proposition}
	
	\newtheorem*{prop:ubi_surj}{Proposition~\ref{prop:ubi_surj}}
	\newtheorem*{thm:equiv_pol_cov}{Corollary~\ref{thm:equiv_pol_cov}}
	\newtheorem*{cor:exact_seq_h2}{Corollary~\ref{cor:exact_seq_h2}}
	\newtheorem*{cor:exact_seq_pi1}{Corollary~\ref{cor:exact_seq_pi1}}
	\newtheorem*{cor:phase}{Corollary~\ref{cor:phase}}

\theoremstyle{definition}
	\newtheorem{dfn}[thm]{Definition}
	\newtheorem{rem}[thm]{Remark}
	\newtheorem{ex}[thm]{Example}
	
\newcommand{\bigslant}[2]{{\raisebox{.2em}{$#1$}\left/\raisebox{-.2em}{$#2$}\right.}}
\newcommand{\smallslant}[2]{{\raisebox{.1em}{$#1$}\left/\raisebox{-.1em}{$#2$}\right.}}
\newcommand{\C}{\mathbb{C}}
\newcommand{\R}{\mathbb{R}}
\newcommand{\Z}{\mathbb{Z}}
\newcommand{\F}{\mathbb{F}}
\newcommand{\Proj}{\mathbb{P}}

\newcommand{\Hom}{\textnormal{Hom}}
\newcommand{\id}{\textnormal{id}}
\newcommand{\df}{\textnormal{d}}
\renewcommand{\mod}{\textnormal{mod}\,}

\newcommand{\Sph}{\textnormal{S}}
\newcommand{\Br}{\textnormal{Br}}
\newcommand{\PolC}{\textnormal{Pol.Cov.}}
\newcommand{\E}{\textnormal{E}}
\newcommand{\M}{\textnormal{M}}
\newcommand{\T}{\Theta}
\newcommand{\Ext}{\textnormal{Ext}}
\newcommand{\Aut}{\textnormal{Aut}}
\newcommand{\colim}{\textnormal{colim}}
\newcommand{\pr}{\textnormal{pr}}
\newcommand{\Cell}{\textnormal{Cell}\,}
\newcommand{\gl}{{\textnormal{g}}}
\newcommand{\G}[1]{\mathbb{G}_{\textnormal{m},#1}}
\newcommand{\hp}{\hat{\pi}_0}
\newcommand{\st}{\textnormal{St}}
\newcommand{\rev}[1]{{#1}'}
\newcommand{\simul}{simultaneously locally path-connected}

\title{
{\bf \textsc{Polar Coordinates \\ and \\ Fundamental Group}}\\ -------------- }
\author{Jules Chenal\footnote{julesche\{at\}math.uio.no} \\
\normalsize{Universitetet i Oslo}}
\date{March 16, 2026}
\begin{document}

\maketitle

\begin{abstract}
	In this article, we investigate the relationship between the fundamental group of a space and its continuous transformations. To be more precise, we show that if a continuous action of a Lie group on a space admits a simply connected cross-section, then we can build the universal covering of the space using an extension of the Lie group by a discrete group.
\end{abstract}

\renewcommand{\abstractname}{Acknowledgements}
\begin{abstract}
	This research was funded, in part, by l’Agence Nationale de la Recherche (ANR), project ANR-22-CE40-0014. The author was also supported by the Trond Mohn Foundation project “\,Algebraic and Topological Cycles in Complex and Tropical Geometries\,”. For the purpose of open access, the author has applied a CC-BY public copyright licence to any Author Accepted Manuscript (AAM) version arising from this submission. The author is grateful to Achim Krause for his insight on homotopy.
\end{abstract}

\tableofcontents

\newpage

\section*{Introduction}
How are a space's symmetries reflected in its fundamental group? A possible answer is the action of its homeomorphisms on its fundamental groupoid. This can be somewhat unsatisfactory, for these objects are inherently large and contain excessive information. In this article, we restrict ourselves to the most topologically significant symmetries. We consider spaces endowed with a particularly well-suited action of a Lie group. Those actions allow for the reconstruction of the space from the data of the group, the orbit space, and the isotropy. In essence, these equivariant spaces are far-reaching generalisations of the couple $(\Sph^1;\R^2)$, where the circle acts as planar rotations. Using analogues of the classical polar coordinates enables us to characterise the fundamental group of these \emph{polar spaces} through mostly cohomological techniques. In particular, the reconstruction property links the topology of the space and the group cohomology of the isotropy with values in the fundamental group of the acting group. Before delving further into the specifics, it is worth noting that many examples of such polar spaces, including the case of the plane, originate from toric geometry. More precisely, if $X$ is a real toric variety under the action of $\G{\R}^n$, then both $(\{\pm1\}^n;X(\R))$ and $((\Sph^1)^n;X(\C))$ are polar spaces. Our approach provides a unified framework to the study of the fundamental groups of $X(\R)$ and $X(\C)$. One can derive both \cite[Theorem~9.1 and Proposition~9.3]{danilov_geometry_1978} and \cite[Corollary~4.5]{davis_convex_1991} from our results.

\paragraph{Polar Coordinates}
In classical polar coordinates a point in the plane is represented by its modulus $q$, a non-negative real number, and an argument $g$, an element of the unit circle. The circle acts on $\R^2$ as the group of planar rotations and the norm $\R^2\rightarrow \R_+$ realises the quotient of the plane by the action. This quotient map is split. A section is given by the map $x:q\in\R_+\mapsto (q;0)\in\R^2$. The idea of polar coordinates can then be understood as the following statement:
\begin{center}
\textit{The continuous surjection $\phi_x: (g;q)\in \Sph^1\times\R_+ \mapsto  g\cdot x(q)\in \R^2$ is a quotient map.}
\end{center}
It means that, despite not being a homeomorphism, this map determines the topology of the plane. The over-determinacy of these coordinates is recorded in the isotropy groups of the points $x(q)$ as $q$ varies in $\R_+$. Of course  only the origin has non-trivial isotropy. This statement can be extended to the action of $\Sph^1$ on $\Sph^2$ or even on a cylinder. The principle remains the same, only the quotient space or the isotropy will change. This leads us to the definition of \emph{polar spaces} and \emph{polar coordinates}, cf. Definition~\ref{dfn:polar_space}. Modulo some technical assumptions, the situation is as follows. We are given a connected space $X$ endowed with a continuous action of a Lie group $G$ whose quotient $Q$ is simply connected and whose quotient map $X\rightarrow Q$ admits a continuous section $x:Q\rightarrow X$. The polar coordinates of $(G;X)$ associated to $x$ are defined just as in the case of the plane:
\begin{equation*}
	\begin{array}{rcl}
		\phi_x:G\times Q&\longrightarrow & X \\
		(g;q) &\longmapsto & g\cdot x(q),
	\end{array}
\end{equation*}
and are assumed to be a quotient map. We do not require $G$ to be connected or compact in general. An example of a polar space with finite group $G$ is given by $\Z/2$ acting on the real projective line by multiplication by $-1$. A topologically equivalent picture is the unit circle endowed with the reflexion about the horizontal axis. The quotient space is a line segment and a section is realised by the closure $[0;\infty]$ of $\R_+$ in $\Proj^1(\R)$. The only points with non-trivial isotropy are the two fixed points $0$ and $\infty$.
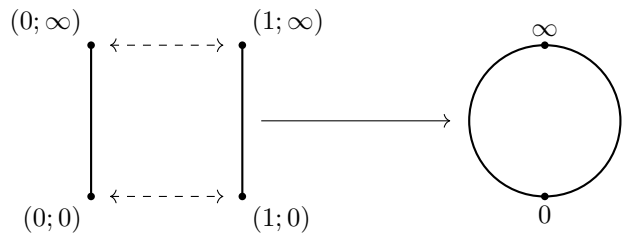
\begin{figure}[H]
	\centering
	\begin{tikzpicture}
		\draw[thick] (0,0) -- (0,2);
		\draw[thick] (2,0) -- (2,2);
		\fill (0,0) circle (.05) node[below left]{$(0;0)$};
		\fill (0,2) circle (.05) node[above left]{$(0;\infty)$};
		\fill (2,0) circle (.05) node[below right]{$(1;0)$};
		\fill (2,2) circle (.05) node[above right]{$(1;\infty)$};
		\draw[<->,dashed] (0.25,0) -- (1.75,0);
		\draw[<->,dashed] (0.25,2) -- (1.75,2);
		\draw[thick] (6,1) circle (1);
		\fill (6,0) circle (.05);
		\fill (6,2) circle (.05);
		\draw[->] (2.25,1) -- (4.75,1);
		\draw (6,0) node[below]{$0$};
		\draw (6,2) node[above]{$\infty$};
	\end{tikzpicture}
	\caption{The Polar Coordinates of $(\Z/2;\Proj^1(\R))$.}
	\label{fig:pol_coor_line}
\end{figure}

\paragraph{Polar Coverings and Fundamental Group} Our goal is to exhibit the relationship between the acting group $G$ and the fundamental group of $X$. Instead of seeking coverings of $X$, one may consider extensions of $G$ by discrete groups and try to build a covering of $X$ using the polar coordinates $\phi_x$. Such extensions of $G$ will be referred to as \emph{discrete extensions}, cf. Definition~\ref{dfn:funct_Br}. Note, however, that an extension alone is not enough. One also needs to choose lifts of the isotropy groups in a “continuous manner”. To illustrate this idea, we will construct the two double coverings of the real projective line. Consider the first factor projection $\pr_1:(\Z/2)^2\rightarrow \Z/2$. The trivial group has only one lift but the whole group $\Z/2$ has two, namely $\Z/2(1;0)$ and $\Z/2(1;1)$. We chose $H_0$ and $H_\infty$ to be either of these two lifts. We have four possible choices in total but only two of them are meaningfully different. Our choice allows us to form the following space:
\begin{equation*}
	X':=\bigslant{(\Z/2)^2\times[0;\infty]}{\sim},
\end{equation*} 
where $\sim$ identifies $(v;0)$ with $(v+h;0)$ for all $v$ in $(\Z/2)^2$ and all $h$ in $H_0$, and $(v;\infty)$ with $(v+h;\infty)$ for all $v$ in $(\Z/2)^2$ and all $h$ in $H_\infty$. This space is naturally endowed with an action of $(\Z/2)^2$ whose quotient is $[0;\infty]$ and whose quotient map is split. The map $\phi_x\circ(\pr_1\times\id_{[0;\infty]}):(\Z/2)^2\times [0;\infty]\rightarrow \Proj^1(\R)$ induces a double covering map $p:X'\rightarrow \Proj^1(\R)$. This map is equivariant in the sense that $p(v\cdot x')$ equals $\pr_1(v)\cdot p(x')$. This covering is connected if and only if $H_0$ is different from $H_\infty$ and Figure~\ref{fig:2cov_circle} depicts these two possibilities.
\begin{figure}[H]
	\centering
	\begin{subfigure}[t]{.5\textwidth}
		\centering
		\begin{tikzpicture}
			\draw[thick] (0,0,0) -- (0,2,0);
			\draw[thick] (2,0,0) -- (2,2,0);
			\fill (0,0,0) circle (.05) node[below left]{$(0;0;0)$};
			\fill (0,2,0) circle (.05) node[above left]{$(0;0;\infty)$};
			\fill (2,0,0) circle (.05) node[below]{$(1;0;0)$};
			\fill (2,2,0) circle (.05) node[above]{$(1;0;\infty)$};
			\draw[thick] (0,0,-2) -- (0,2,-2);
			\draw[thick] (2,0,-2) -- (2,2,-2);
			\fill (0,0,-2) circle (.05) node[below]{$(0;1;0)$};
			\fill (0,2,-2) circle (.05) node[above left]{$(0;1;\infty)$};
			\fill (2,0,-2) circle (.05) node[below]{$(1;1;0)$};
			\fill (2,2,-2) circle (.05) node[above]{$(1;1;\infty)$};
			\draw[<->,dashed] (0.25,0,0) -- (1.75,0,0);
			\draw[<->,dashed] (0.25,2,0) -- (1.75,2,0);
			\draw[<->,dashed] (0.25,0,-2) -- (1.75,0,-2);
			\draw[<->,dashed] (0.25,2,-2) -- (1.75,2,-2);
			\draw[thick] (5,1,-2) circle (1);
			\fill[white] (5.97,.79,0) circle (.07);
			\fill[white] (4.8,1.98,0) circle (.07);
			\draw[thick] (5,1,0) circle (1);
			\fill (5,0,0) circle (.05);
			\fill (5,2,0) circle (.05);
			\fill (5,0,-2) circle (.05);
			\fill (5,2,-2) circle (.05);
			\draw[->] (2.5,1,-1) -- (3.5,1,-1);
		\end{tikzpicture}
		\caption{The Polar Coordinates of $((\Z/2)^2;\Proj^1(\R)\coprod \Proj^1(\R))$.}
	\end{subfigure}
	\hfill
	\begin{subfigure}[t]{.47\textwidth}
		\centering
		\begin{tikzpicture}
			\draw[thick] (0,0,0) -- (0,2,0);
			\draw[thick] (2,0,0) -- (2,2,0);
			\fill (0,0,0) circle (.05);
			\fill (0,2,0) circle (.05);
			\fill (2,0,0) circle (.05);
			\fill (2,2,0) circle (.05);
			\draw[thick] (0,0,-2) -- (0,2,-2);
			\draw[thick] (2,0,-2) -- (2,2,-2);
			\fill (0,0,-2) circle (.05);
			\fill (0,2,-2) circle (.05);
			\fill (2,0,-2) circle (.05);
			\fill (2,2,-2) circle (.05);
			
			\draw[<->,dashed] (0.25,0,0) -- (1.75,0,0);
			\draw[<->,dashed] (0.25,0,-2) -- (1.75,0,-2);
			\draw[<->,dashed] (0.15,2,-0.15) .. controls (.5,2,-.5) and (1.5,2,-1.5) .. (1.85,2,-1.85);
			\draw[<->,dashed] (2,2.15,0) .. controls (1.5,3,-.5) and (.5,2.5,-1.5) .. (0,2.15,-2);
			\draw[thick] (5,2,-2) .. controls (3.6,1.9,-2) and (3.6,.1,-2) .. (5,0,-2);
			\draw[thick] (5,0,-2) .. controls (7,0,-2) and (7,2,0) .. (5,2,0);
			\fill[white] (4.74,1.96,0) circle (.07);
			\fill[white] (6.27,.53,-1) circle (.07);
			\fill[white] (6.5,1,-1) circle (.07);
			\draw[thick] (5,0,0) .. controls (7,0,0) and (7,2,-2) .. (5,2,-2);
			\draw[thick] (5,2,0) .. controls (3.6,1.9,0) and (3.6,0.1,0) .. (5,0,0);
			\fill (5,0,0) circle (.05);
			\fill (5,2,0) circle (.05);
			\fill (5,0,-2) circle (.05);
			\fill (5,2,-2) circle (.05);
			\draw[->] (2.5,1,-1) -- (3.25,1,-1);
		\end{tikzpicture}
		\caption{The Polar Coordinates of $((\Z/2)^2;\Proj^1(\R))$.}
	\end{subfigure}
	\caption{The Two Double Coverings of $\Proj^1(\R)$.}
	\label{fig:2cov_circle}
\end{figure}
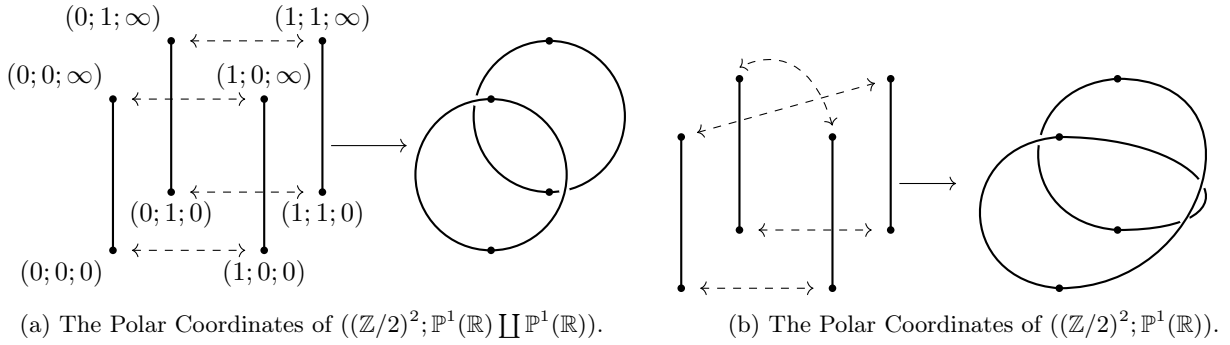
We define a \emph{polar covering}, cf. Definition~\ref{dfn:polcov}, of a polar space $(G;X)$ to be a connected Galois covering $p_{X'}:X'\rightarrow X$ that admits an action of a discrete extension $p_{G'}:G'\rightarrow G$ that “lifts” the action of $G$ on $X$, i.e. $p_{X'}(g\cdot x)=p_{G'}(g)\cdot p_{X'}(x)$. In this case, the group of deck transformations of $p_{X'}$ is the kernel of $p_{G'}$. We show that the universal covering of $X$ is necessarily a polar covering, cf. Proposition~\ref{prop:univ=polar}. Therefore, the fundamental group of $X$ can be found by identifying the discrete extension of $G$ leading to the universal covering. We will show that every polar covering of a polar space can be constructed as the two double covers of the real projective line provided one can choose appropriately the lifts of the isotropy of the orbits of $X$. 

\paragraph{Isotropy} As illustrated by our example of $(\Z/2;\Proj^1(\R))$, the isotropy of the action plays a central role in the story. The isotropy groups of all the orbits of the action of $G$ over $X$ must be considered all at once with their variation as we change orbits. For this reason, we treat them as a single object attached to a section $x:Q\rightarrow X$ of the quotient map. Given such a section $x$, we define its isotropy to be the space:
\begin{equation*}
	H:=\{(g;q)\in G\times Q~|~g\cdot x(q)=x(q)\},
\end{equation*}
together with its projection onto $Q$. Every fibre of this projection is a group. The object $H\rightarrow Q$ is even a group in the category of topological spaces over $Q$: the category of \emph{$Q$-spaces}. We refer to these as \emph{groups over $Q$} or {\emph{$Q$-groups}}. Since these objects are internal groups in the category of spaces over $Q$ it makes sense to consider morphisms of $Q$-groups and actions of $Q$-groups over $Q$-spaces. Another example of $Q$-groups is given by $\pr_Q:G\times Q \rightarrow Q$. We denote it by $G_Q$. It contains $H$ and we denote this inclusion by $u$. Therefore, $G_Q$ is endowed with a right action of $H$ and $\phi_x$ realises\footnote{This is a requirement of Definition~\ref{dfn:polar_space}.} the quotient $G_Q/H$. In that regard, the couple $(G;X)$ can be thought of as a homogeneous $Q$-space $(G_Q;X)$. We can now be more precise about the polar covering building process. If we consider a discrete extension $G'$ of $G$ endowed with a $Q$-group morphism $u':H\rightarrow G'_Q$ that lifts $u$, the quotient $G'_Q/u(H)$ is a covering of $X$. To be polar, it has to be connected. Hence, we need to understand when such a quotient is connected. This leads us to the notion of \emph{ubiquity}: we say that a morphism of $Q$-groups $f:H\rightarrow H'$ is ubiquitous when the quotient $H'/H$ is a connected topological space. We can derive a criterion for ubiquity using a construction that assign a discrete group $\hp(H)$ to every $Q$-group $H$. It plays the role of the group of connected components of a usual topological group. It may be thought of as a sort of colimit of the family of groups of connected components of the various fibres of $H$. Concretely, $\hp:(Q-\textnormal{Groups})\rightarrow (\textnormal{Discrete Groups})$ is left adjoint to the functor $(-)_Q:(\textnormal{Discrete Groups})\rightarrow (Q-\textnormal{Groups})$ sending $K$ to $K_Q$, the $Q$-group $\pr_Q:K\times Q\rightarrow Q$. 

\begin{prop:ubi_surj}
	Let $G$ be a Lie group, $H\rightarrow Q$ be a $Q$-group, and $f:H\rightarrow G_Q$ be a morphism of $Q$-groups. The morphism $f$ is ubiquitous if and only if the morphism $f_*:\hat{\pi}_0(H)\rightarrow \pi_0(G)$ is onto. 
\end{prop:ubi_surj}
We define the cohomology of $Q$-groups to foster obstruction classes, cf. Definition~\ref{def:Qgroup_cohomo}. The most interesting situation is the following: we consider a $Q$-group $G$ and a discrete abelian group $M$ such that $M_Q$ is endowed with a $Q$-group action of $G$ by $Q$-group automorphisms. We show that, for $G$ fixed, the category of such objects $M$ is equivalent to the category of $\hp(G)$-modules. Hence, this category is Abelian. If $G$ satisfies a condition of local connectedness the cohomology functors $(M\mapsto H^n_Q(G;M_Q))_{n\geq 0}$ form a $\delta$-functor. Using the unit $\eta$ of the adjunction, these are endowed with natural split injective morphisms $\eta_G^*:H^n(\hp(G);M)\rightarrow H^n_Q(G;M_Q)$. They are always isomorphisms in degrees $0$ and $1$, and we provide a spectral sequence that computes the cohomology of $G$ with coefficients in $M$ using classical objects, cf. Proposition~\ref{prop:defct_spectral_sequence}. In particular, we can derive a short exact sequence in degree 2, the most important group for our purpose.     
\begin{cor:exact_seq_h2}
	Let $G$ be a $Q$-group, we have a natural exact sequence for every $\hp(G)$-module $M$:
	\begin{equation*}
		0 \longrightarrow H^2(\hp(G);M) \overset{\eta_G^*}{\longrightarrow} H^2_Q(G;M_Q) \longrightarrow \Hom_{\Z[\hp(G)]}(\Delta_2(G);M) \longrightarrow 0.
	\end{equation*}
\end{cor:exact_seq_h2}
The group $\Delta_2(G)$ is part of the \emph{defect} of $G$, cf. Definition~\ref{dfn:defct}. The defect is the homology of what should be the bar resolution of $G$. For a $Q$-group $G$ without defect, $\eta^*_G$ is an isomorphism in every degree. 

\paragraph{Equivalence of Categories} The proper way to exhibit the link between the covering building process from discrete extensions and polar coverings is through categorical equivalence. For $(G;X)$ a \emph{regular} polar space endowed with a \emph{regular} section $x$ of the quotient map, we establish an equivalence between the category $\textnormal{Pol.Cov}(G;X;x)$ of polar coverings of $X$ endowed with a lift of the section $x$ and the category $\E(G;X;x)$ of discrete extensions of $G$ that admit a ubiquitous lift of the isotropy of $x$. See Definitions~\ref{dfn:regular}, \ref{dfn:PolCovGXx}, and \ref{dfn:EGXx}.

\begin{thm:equiv_pol_cov}[Equivalence of Categories]
	Let $(G;X)$ be a regular polar space with quotient space $Q$ and $x:Q\rightarrow X$ be a regular section of the quotient map. There is an equivalence of categories:
	\begin{equation*}
		\Br^*:\PolC(G;X;x) \longleftrightarrow \textnormal{E}(G;X;x):\textnormal{C}.
	\end{equation*} 
\end{thm:equiv_pol_cov}

\paragraph{In Search of the Fundamental Group} In light of Theorem~\ref{thm:equiv_pol_cov}, we seek the discrete extension associated to the universal covering of $X$. If $p_E:E\rightarrow G$ is a discrete extension of $G$, the homotopy long exact sequence yields the following five terms exact sequence\footnote{Groups are naturally pointed at the identity and group morphisms are necessarily maps of pointed spaces. As a consequence, when we write $\pi_k(G)$ for some group $G$ we implicitly mean $\pi_k(G;1)$.}:
\begin{equation}
	0\rightarrow \pi_1(E)\rightarrow \pi_1(G) \rightarrow \ker(p_E) \rightarrow \pi_0(E) \rightarrow \pi_0(G) \rightarrow 1.
\end{equation}
As a (non-necessarily connected) covering, $p_E:E\rightarrow G$ is fully characterised by the two group morphisms: $(p_E)_*:\pi_0(E) \rightarrow \pi_0(G)$ and $(p_E)_*:\pi_1(E)\rightarrow \pi_1(G)$. We respectively call them the \emph{magnitude} and the \emph{phase} of the extension. Following \cite{taylor_covering_1954}, isomorphisms classes of discrete extensions $p_{E'}:E'\rightarrow G$ with the same magnitude and phase as $E$ form a $H^2(\pi_0(E);\pi_1(G)/\pi_1(E))$-principal homogeneous space. Furthermore, given a magnitude $\mu$ and a phase $\theta$, there is an obstruction class $\chi(\mu;\theta)$ in $H^3(\mu;\pi_1(G)/\theta)$ that vanishes if and only if there exists a discrete extension of $G$ of the given magnitude and phase. The characterisation of ubiquity determines the magnitude of the universal covering of $X$. 

\begin{cor:exact_seq_pi1}
	Let $(G;X)$ be a polar space with quotient space $Q$, and $H$ be the isotropy of a regular section. The fundamental group of $X$ is a central extension of the kernel of the morphism $\hat{\pi}_0(H)\rightarrow \pi_0(G)$ by a quotient of $\pi_1(G)$ by a $\pi_0(G)$-submodule $\theta_\textnormal{uni}$. In particular, we have an exact sequence:
	\begin{equation*}
		0 \longrightarrow \theta_\textnormal{uni} \longrightarrow \pi_1(G) \longrightarrow \pi_1(X) \longrightarrow \hat{\pi}_0(H)\rightarrow \pi_0(G) \longrightarrow 1.
	\end{equation*}
\end{cor:exact_seq_pi1}

To find the phase, we use two lift obstruction classes. One is of topological nature and lives in the singular cohomology of the isotropy $H$. The other one is algebraic and is a class in $H^2_Q(H;(\pi_1(G)/\theta)_Q)$.
 They allow us to find a topological lower bound $\theta_\textnormal{top}$ for $\theta_\textnormal{uni}$ and to define a closure operator on a subcategory $\Theta^*(G;X)$ of phases of $X$ from which we derive the following corollary: 

\begin{cor:phase}
	The phase $\theta_\textnormal{uni}$ of the universal covering of $X$ is given by the following formula:
	\begin{equation*}
		\theta_\textnormal{uni}=\bigcap_{\theta\in \Theta^*(G;X)}\overline{\theta}.
	\end{equation*}
	In particular, if $\chi(\hat{\pi}_0(H);\theta_\textnormal{top})$ vanishes, then $\theta_\textnormal{uni}$ is the effective closure of $\theta_\textnormal{top}$.
\end{cor:phase}

It is worth adding that the obstruction classes can also be used to find the class of $\pi_1(X)$ as a central extension of $\kappa$, the kernel of $\hat{\pi}_0(H)\rightarrow \pi_0(G)$, by  $\pi_1(G)/\theta_\textnormal{uni}$  in $H^2(\kappa;\pi_1(G)/\theta_\textnormal{uni})$.

\vspace{5pt}

To conclude, we may point out that the most simple examples of polar spaces are connected homogeneous spaces and compare Corollary~\ref{cor:exact_seq_pi1} with the homotopy long exact sequence of the fibration $H\rightarrow G \rightarrow G/H$:
\begin{equation*}
	\pi_1(H)\longrightarrow \pi_1(G) \longrightarrow \pi_1(G/H) \longrightarrow \pi_0(H) \longrightarrow \pi_0(G) \longrightarrow 1.
\end{equation*}
To strengthen the comparison, we can rewrite Corollary~\ref{cor:exact_seq_pi1} as the following exact sequence:
\begin{equation*}
	0 \longrightarrow \theta_\textnormal{uni} \longrightarrow \pi_1(G_Q) \longrightarrow \pi_1(G_Q/H) \longrightarrow \hp(H)\rightarrow \hp(G_Q) \longrightarrow 1.
\end{equation*}
A satisfying achievement would be to define a functor $\hat{\pi}_1$ on the category of $Q$-groups for which $\hat{\pi}_1(G_Q)$ is naturally isomorphic to $\pi_1(G)$ and such that the image of the morphism $\hat{\pi}_1(H)\rightarrow \hat{\pi}_1(G_Q)$ is precisely $\theta_\textnormal{uni}$. It would finish the analogy and raise the question: Does it holds for more general homogeneous $Q$-spaces? i.e. for the homogeneous spaces where the numerator is not necessarily of the form $G_Q$.

\section{Spaces and Groups over Another}
Let $Q$ be an Hausdorff, connected, and locally path-connected space.

\begin{dfn}
	A \emph{space over} $Q$ or \emph{$Q$-space} is a Hausdorff, locally path-connected topological space $X$ together with a continuous map $X\rightarrow Q$. We will not introduce a special notation for the latter map but every time we need to consider the image of a point $x$ of $X$ by such morphism we will denote it by $||x||$. For every point $q$ of $Q$ we denote by $X_q$ the fibre of $X\rightarrow Q$ over $q$. A \emph{morphism over} $Q$ or \emph{morphism of $Q$-spaces}, $f:X\rightarrow Y$ is a continuous map from $X$ to $Y$ such that the following diagram commutes:
	\begin{equation*}
		\begin{tikzcd}
			X \ar[rr,"f"] \ar[rd] && Y \ar[ld]\\ & Q &
		\end{tikzcd}
	\end{equation*}
	Moreover, a \emph{group over $Q$} or a \emph{$Q$-group} is a group object in the category of spaces over $Q$. We have a functor $X\mapsto X_Q:=X\times Q$ that sends locally path-connected Hausdorff topological spaces (resp. locally path-connected Hausdorff topological groups) to the category of $Q$-spaces (resp. $Q$-groups).
\end{dfn}

\begin{ex}\label{ex:1}
	Let $Z\rightarrow D^2$ denote the subgroup of $\Z\times D^2$ made of the points $(k;t)$ such that $|t|<1$ implies $k=0$. It is a $D^2$-subgroup of $\Z\times D^2$.
\end{ex}

\begin{dfn}
	Let $X\rightarrow Q$ be a $Q$-space and $k$ be a non-negative integer. We denote\footnote{N.b. It is consistent with the notation $Y_Q$ for a topological space $Y$ for $(Y_Q)^k_Q$ is naturally homeomorphic to $(Y^k)_Q$.} by $X^k_Q$ the $k$-fold fibre product of $X$ with itself over $Q$:
	\begin{equation*}
		X^k_Q\coloneqq \underset{k \textnormal{ times}}{\underbrace{X \underset{Q}{\times} \cdots \underset{Q}{\times} X}}.
	\end{equation*}
	We say that $X\rightarrow Q$ is \emph{\simul} if $X^k_Q$ is locally path-connected for every non-negative integer $k$.
\end{dfn}

\begin{dfn}
	Let $X\rightarrow Q$ and $G\rightarrow Q$ be respectively a space and a group over $Q$. We denote the group law of $G$ by $m$ and its identity by $1:Q\rightarrow G$. An action of $G$ on $X$ is a morphism of $Q$-spaces:
	\begin{equation*}
		a:G\times_Q X \rightarrow X,
	\end{equation*}
	such that the following square commutes:
	\begin{equation*}
		\begin{tikzcd}
			G\times_Q G\times_Q X \ar[rr,"\id_G\times a"] \ar[d,"m\times\id_X" left] && G\times_Q X \ar[d,"a"]\\
			G\times_Q X \ar[rr,"a" below] && X
		\end{tikzcd}
	\end{equation*}
	and such that the composition $a\circ (1\times\id_X): X=Q\times_Q X \rightarrow X$ is the identity.
\end{dfn}

As usual an action yields an equivalence relation on the space for which one can form the quotient.

\begin{dfn}[Quotients]\label{dfn:quotient}
	Let $X\rightarrow Q$ be a space over $Q$ endowed with an action of a $Q$-group $G\rightarrow Q$. We define the quotient space $X/ G$ to be the quotient of $X$ by the following equivalence relation:
	\begin{equation}\label{rel:quotient}
		x\sim y \Leftrightarrow \left\{\begin{array}{l} ||x||=||y|| \\ \exists g\in G_{||x||},\;y=g\cdot x \end{array}\right.
	\end{equation}
	This topological space is endowed with a continuous map to $Q$. Since the quotient of a locally path connected space is locally path connected, $X/G$ is a $Q$-space when it is Hausdorff. 
\end{dfn}

\subsection{Ubiquity}
This notion formalises the idea that an action on a space allows to visit every connected component of the space by only moving through the transformation group. 

\begin{dfn}
	We say that an action of $G\rightarrow Q$ over $X\rightarrow Q$ is \emph{ubiquitous} if the quotient $X/ G$ is connected. Moreover, we say that a morphism $f:H\rightarrow G$ of $Q$-groups is \emph{ubiquitous} if the induced action of $H$ on $G$ is ubiquitous. 
\end{dfn}

When $Q$ is a point there is an easy characterisation of ubiquity. Indeed, a morphism of locally path-connected Hausdorff groups $f:H\rightarrow G$ is ubiquitous if and only if the induced morphism of discrete groups $f_*:\pi_0(H)\rightarrow \pi_0(G)$ is surjective. We will derive a similar criterion in special cases of $Q$-groups. To do so, we need to define an analogue of the functor $\pi_0$  for $Q$-groups. We can note that $\pi_0$ is the left adjoint of the inclusion of discrete groups into the category of locally path-connected groups. This observation provides the adequate generalisation. 

\begin{prop}\label{prop:adjunction}
	The functor $D\mapsto D_Q$ restricted to discrete groups admits a left adjoint denoted by $\hat{\pi}_0$. 
\end{prop}

\begin{proof}
	Let $G\rightarrow Q$ be a $Q$-group and $D$ be a discrete group. A map of $Q$-spaces $f:G\rightarrow D_Q$ is fully determined by its composition with the projection onto $D$. Since $D$ is discrete, the latter induces a map $f_*:\pi_0(G)\rightarrow D$. Reciprocally, any map $f_*:\pi_0(G)\rightarrow D$ yields a continuous map of $Q$-spaces $f:G\rightarrow D_Q$ for $Q$-groups are assumed to be locally connected. One might ask: under what condition (on $f_*$) is $f$ a $Q$-group morphism ? We can note that $\pi_0(G)$ is some sort of “\,partial multivalued group\,”. Let $m$ denote the group law of $G$. For any $c,d$ in $\pi_0(G)$ we denote by $c\cdot d$ the set of connected components of $G$ reached by $m(c\times_Q d)$. If $c\times_Q d$ is empty then so is $c\cdot d$. We will show that $f$ is a morphism of $Q$-groups if and only if $f_*$ satisfies the following property:
	\begin{equation}\label{eq:multivalued_group_morph}
		\forall c_3\in c_1\cdot c_2,\;f_*(c_3)=f_*(c_1)\cdot f_*(c_2).
	\end{equation}
	Let us assume that $f$ is a $Q$-group morphism. We consider two components $c_1$ and $c_2$ such that $c_1\times_Qc_2$ is non-empty, and choose a connected component $c_3$ in $c_1\cdot c_2$. By assumption, there exists $q$ in $Q$, and $g_i$ in $G_q\cap c_i$ for all $i$ in $\{1,2,3\}$, such that $g_3$ equals $g_1\cdot g_2$. Since $f$ is a $Q$-group morphism, $f(g_3)$ equals $f(g_1)\cdot f(g_2)$. Hence, $f_*(c_3)$ equals $f_*(c_1)\cdot f_*(c_2)$, and (\ref{eq:multivalued_group_morph}) is satisfied. Reciprocally, let us assume that $f_*$ satisfies (\ref{eq:multivalued_group_morph}). We consider $g_1$ and $g_2$ two elements of a fibre $G_q$. We denote by $c_1$, $c_2$ and $c_3$ the respective connected components of $g_1$, $g_2$, and $g_1\cdot g_2$. Then, we have:
	\begin{equation*}
		f(g_1\cdot g_2)= (f_*(c_3);q) = (f_*(c_1)\cdot f_*(c_2) ; q) = f(g_1)\cdot f(g_2). 
	\end{equation*}
	And $f$ is a morphism of $Q$-groups. We can now turn $\pi_0(G)$ into an actual group with the following presentation:
	\begin{equation}\label{eq:pi_0_hat}
		\hat{\pi}_0(G):= \langle \pi_0(G)\;|\; c_1c_2c_3^{-1},\;\forall c_3\in c_1\cdot c_2\rangle.
	\end{equation}
	There is a natural map $\eta':\pi_0(G)\rightarrow \hat{\pi}_0(G)$ that satisfies (\ref{eq:multivalued_group_morph}). Given the presentation (\ref{eq:pi_0_hat}), we see that any $f_*:\pi_0(G)\rightarrow D$ satisfying (\ref{eq:multivalued_group_morph}) is of the form $g\circ\eta'$ for a unique group morphism $g:\hat{\pi}_0(G)\rightarrow D$. We can finish this proof by emphasising that the construction (\ref{eq:pi_0_hat}) is functorial and that $\eta:g\mapsto(\eta'[g];||g||)$, where $[g]$ denotes the connected component of $g$, is the unit of the adjunction.   
\end{proof}

Given the last proof, we find the following direct consequence.

\begin{prop}
	If $G$ is a Lie group, then $\hat{\pi}_0(G_Q)$ is naturally isomorphic to the group $\pi_0(G)$. 
\end{prop}

\begin{dfn}
	We denote by $\eta_G:\textnormal{Id}_{Q\textnormal{-group}}\rightarrow (\hat{\pi}_0)_Q$ the unit of the adjonction of Proposition~\ref{prop:adjunction}. Let $G$ be a $Q$-group. In order to limit the number of notations, we will also denote by $\eta_G(g)$ the projection of the image of $g$ by the unit of the adjunction onto $\hp(G)$. (The actual image of $g$ by the adjunction morphism is then $(\eta_G(g);||g||)$ in $\hp(G)_Q$.) 
\end{dfn}

\begin{prop}\label{prop:ubi_surj}
	Let $G$ be a Lie group, $H\rightarrow Q$ be a $Q$-group, and $f:H\rightarrow G_Q$ be a morphism of $Q$-groups. The morphism $f$ is ubiquitous if and only if the morphism $f_*:\hat{\pi}_0(H)\rightarrow \pi_0(G)$ is onto. 
\end{prop}

\begin{proof}
	Let us denote by $X$ the quotient space $G_Q/ H$. We will begin by proving the following property:
	\begin{center}\textit{
		Let $(g;q),(g';q')$ be two points of $G_Q$, if there exists $c$ in $\pi_0(H)$ such that $[g]$ equals $[g']\cdot f_*(c)$ in $\pi_0(G)$, then the images of $(g;q)$ and $(g';q')$ in $X$ belong to the same connected component. 
	}\end{center}
	First, we note that $[g]$ being equal to $[g']\cdot f_*(c)$ implies that we can find $g_0$ and $g_0'$ in $G^\circ$, the connected component of the identity, as well as $(h;p)$ in the connected component of $G_Q$ spanned by $f(c)$ such that $(gg_0;p)$ equals $(g'g_0'h;p)$. We also consider a point $(h';p')$ in $c$. Then, we can consider five continuous paths in $G_Q$:\begin{multicols}{2}\begin{enumerate}
		\item $\gamma_1$ from $(g;q)$ to $(g;p)$ for we assumed $Q$ connected;
		\item $\gamma_2$ from $(g;p)$ to $(gg_0;p)=(g'g'_0h;p)$;
		\item $\gamma_3$ from $(g'g'_0h;p)$ to $(g'g'_0f(h');p')$. It is the image by the multiplication by $g'g'_0$ of a path from $(h;p)$ to $(f(h');p')$;\columnbreak
		\item $\gamma_4$ from $(g'g'_0f(h');p')$ to $(g'f(h');p')$;
		\item $\gamma_5$ from $(g';p')$ to $(g';q')$.
	\end{enumerate}
	\end{multicols}
	Since $(g'f(h');p')$ and $(g';p')$ have the same image in $X$, one can form the concatenation of the images of these paths in the quotient. It yields a path from the image of $(g;q)$ to the image of $(g';q')$. Therefore, by transitivity, one finds that $X$ is path connected when $f_*:\hat{\pi}_0(H)\rightarrow \pi_0(G)$ is onto. For the converse statement, we can note that the composition $G_Q\rightarrow \pi_0(G) \rightarrow \pi_0(G)/f_*\hat{\pi}_0(H)$ is onto and invariant by the action of $H$. Thus, we have a continuous surjective map $X\rightarrow \pi_0(G)/f_*\hat{\pi}_0(H)$. If $X$ is path connected, then $\pi_0(G)/f_*\hat{\pi}_0(H)$ has to be reduced to a point for it is discrete. 
\end{proof}

\subsection{Gluings}
There is a well behaved type of spaces and group over $Q$ that we will consider later for example purposes. Let us consider a regular cellular structure\footnote{A CW complex is \emph{regular} when every open cell has a characteristic map that extends to an embedding of the closed ball.} $\Cell Q$ on $Q$. We denote by the same symbol the small category associated to the underlying ordered set of cells of $\Cell Q$.

\begin{dfn}\label{dfn:gluing}
	A \emph{cosheaf of spaces} (resp. \emph{groups}, resp. \emph{cellular complexes}) is a contravariant functor $X$ from $\Cell Q$ to the category of topological spaces (resp. Lie groups, resp. cellular complexes). A morphism of such objects is a natural transformation. To every such cosheaf one can associate a space (resp. groups, resp. cellular complexes) over $Q$ by considering the following gluing procedure:
	\begin{equation*}
	(X)_\gl:=\colim\left(\coprod_{e_1\leq e_2\in \Cell Q} X(e_2)\times \bar{e}_1 \overset{u}{\underset{d}{\rightrightarrows }}\coprod_{e\in\Cell Q}X(e)\times \bar{e}\right)\rightarrow Q,
	\end{equation*}
	where the map $u$ sends $X(e_2)\times \bar{e}_1$ in $X(e_2)\times \bar{e}_2$ by the identity on the first factor and inclusion on the second factor, and $d$ sends $X(e_2)\times \bar{e}_1$ in $X(e_1)\times \bar{e}_1$ by the transition map of $X$ on the first factor and the identity on the second factor. This procedure yields a functor from cosheaves on $\Cell Q$ to objects over $Q$.
\end{dfn}

\begin{ex}
	If one endows $[0;1]$ with its minimal cellular structure $\{\{0\};[0;1];\{1\}\}$ a cosheaf of groups is nothing but a choice of three groups $G_0$, $G_1$, and $G_{[0;1]}$ together with two morphisms $G_{[0;1]}\rightarrow G_0$ and $G_{[0;1]}\rightarrow G_1$.
\end{ex}

\begin{ex}
	Any flat vector bundle or system of local coefficients over a cellular space is isomorphic to the gluing of cosheaf of groups;
\end{ex}

\begin{ex}
	Drawing inspiration on \cite[Example 8]{blazquez-sanz_group_2022}, we consider the $[1;2]$-group $G$ endowing $\R^3\times [1;2]\rightarrow  [1;2]$ with the following $[1;2]$-group law:
		\begin{equation*}
			(x_1;y_1;z_1;t)\cdot (x_2;y_2;z_2;t)\coloneqq (x_1+e^{z_1}x_2;y_1+e^{tz_1}y_2;z_1+z_2;t).
		\end{equation*}
		This group cannot be a gluing for no two fibres are isomorphic. The Lie algebra $\mathfrak{g}_t$ of $G_t$ has three generators $X$, $Y$, and $Z$. The two first commute, $[Z;X]$ is $X$, and $[Z;Y]$ is $tY$. Any morphism of Lie algebras from $\mathfrak{g}_{t_1}$ to $\mathfrak{g}_{t_2}$ is necessarily given by a matrix $M$ (in the respective bases given by $X,Y,Z$) satisfying the following description:
		\begin{equation*}
			M=\left( \begin{array}{ccc} a & b & u\\ c & d & v \\ 0 & 0 & k \end{array}\right) \quad \textnormal{with} \quad \left( \begin{array}{cc} k & 0 \\ 0 & kt_2\end{array}\right) \left(\begin{array}{cc} a & b \\ c & d \end{array}\right)= \left(\begin{array}{cc} a & b \\ c & d \end{array}\right) \left(\begin{array}{cc} 1 & 0 \\ 0 & t_1 \end{array}\right).
		\end{equation*}
		If it is to be an isomorphism, the polynomials $(x-k)(x-kt_2)$ and $(x-1)(x-t_1)$ must be equal. Hence $t_2$ has to equal $t_1$ or its inverse. Given our restrictions, they must be equal.
\end{ex}

\begin{prop}\label{prop:Hausdorff_LPC_gluing}
	Let $Q$ be a cellular space and $X$ be a cosheaf of topological spaces on a regular cellular structure $\Cell Q$ of $Q$. If $X(e)$ is Hausdorff and locally path-connected for every cell $e$, then $(X)_\gl\rightarrow Q$ is a simultaneously locally path-connected $Q$-space. 
\end{prop}

\begin{proof}
	By definition, we have a quotient projection:
	\begin{equation*}
		\pi:\coprod_{e\in \Cell Q}X(e)\times \bar{e} \rightarrow (X)_\gl,
	\end{equation*}
	Hence, if $X(e)$ is locally path-connected for all cells $e$, so is the gluing for quotients of locally path-connected spaces are locally path-connected. It can be proved as \cite[Chapitre~I, \S 11, \textnumero 6, Proposition~12]{bourbaki_topologie_2007-1} using \cite[Theorem 25.3]{munkres_topology_2014} in place of \cite[Chapitre~I, \S 11, \textnumero 6, Proposition~11]{bourbaki_topologie_2007-1}. Moreover, for every non-negative integer $k$, $((X)_\gl)^k_Q$ is homeomorphic to $(X^k)_\gl$, so the gluing is even simultaneously locally path connected over $Q$. Only remains to show that $(X)_\gl$ is Hausdorff when $X(e)$ is Hausdorff for all cells $e$. Since $Q$ is Hausdorff, every couple of points with distinct images under $(X)_\gl\rightarrow Q$ are separated by disjoint open neighbourhoods. Let us be given two distinct points in the same fibre of $(X)_\gl\rightarrow Q$. Let $q$ be there common image in $Q$ and $e$ be the unique cell containing $q$. The restriction of $\pi$ to $X(e)\times e$ is injective. Hence, our initial points are of the form $\pi(x;q)$ and $\pi(y;q)$ for uniquely determined and distincts $x,y$ in $X(e)$. By assumption, $x$ and $y$ respectively lie in open sets $U$ and $V$ of $X(e)$ whose intersection is empty. For every cell $e'$ containing $e$, we denote by $\st(e;\bar{e}')$ the open star\footnote{It is the reunion of every open cell $e''$ satisfying $e\leq e''\leq e'$.} of $e$ in $\bar{e}'$ and by $U(e')$ (resp. $V(e')$) the inverse image of $U$ (resp. $V$) by the transition map ${X(e')\rightarrow X(e)}$. The sets:
	\begin{equation*}
		\coprod_{e'\geq e}U(e')\times \st(e;\bar{e}') \quad \textnormal{and} \quad \coprod_{e'\geq e}V(e')\times \st(e;\bar{e}'),
	\end{equation*}
	are two disjoint open sets of $\coprod_{e}X(e)\times \bar{e}$. Moreover, they are saturated for the kernel relation\footnote{A set $A$ is saturated for the kernel relation of $\pi$ if $A$ equals $\pi^{-1}\pi(A)$.} of $\pi$. Thus, their respective images under $\pi$ yield disjoint open neighbourhoods of our initial points. Therefore, $(X)_\gl$ is Hausdorff and the proposition follows. 
\end{proof}

\begin{prop}\label{prop:commutativity_quotients_gluings}
	Let $Q$ be a topological space endowed with a regular cellular structure $\Cell Q$ and $G$ be a cosheaf of Lie groups over $\Cell Q$ acting continuously on a cosheaf of spaces $X$. The $Q$-space $(X/G)_\gl$ is isomorphic to $(X)_\gl/(G)_\gl$.
\end{prop}

\begin{proof}
	The functoriality of the gluing implies that the morphism of cosheaves of spaces $p:X\rightarrow X/G$, yields a continuous surjective map $(p)_\gl:(X)_\gl\rightarrow (X/G)_\gl$. Moreover, as a direct consequence of Definitions~\ref{dfn:quotient} and \ref{dfn:gluing}, one finds that two points in $(X)_\gl$ have the same image under $(p)_\gl$ if and only if they are equivalent for the relation (\ref{rel:quotient}) induced on $(X)_\gl$ by the action of $(G)_\gl$. Hence, the proposition will follow once we showed that $(p)_\gl$ is a quotient map. We have the following commutative square:
	\begin{equation*}
		\begin{tikzcd}
			\displaystyle\coprod_{e\in \Cell P} X(e)\times \bar{e} \ar[d,"h:=\coprod_e p_e\times\id_{\bar{e}}" left] \ar[r,"f"]& (X)_\gl \ar[d,"(p)_\gl"]\\
			\displaystyle\coprod_{e\in \Cell P} (X/G)(e)\times \bar{e} \ar[r,"g" below]& (X/G)_\gl
		\end{tikzcd}
	\end{equation*}
	The maps $f$ and $g$ are induced from the gluing construction and are quotient maps. Since the maps $p_e$ are quotient maps for all cells $e$, the map $h$ is also a quotient map. Thus, every subset $A$ of $(X/G)_\gl$ whose inverse image by $(p)_\gl$ is open, satisfies $h^{-1}(g^{-1}(A))$ open. Since $g$, and $h$ are quotient maps, we find $g^{-1}(A)$ open and then $A$ open. Therefore, $(p)_\gl$ is a quotient map and the proposition holds. 
\end{proof}

\begin{prop}\label{prop:regular_subgroup}
	Let $G$ be a Lie group and $Q$ be a topological space endowed with a regular cellular structure $\Cell Q$. Let $H$ be a cosheaf of closed subgroups of $G$ on $\Cell Q$. Let $q_0$ be a point in $Q$ and $e_0$ be the open cell that contains $q_0$. There is an $(H)_\gl$-equivariant open neighbourhood $U$ of $H(e_0)\times\{q_0\}$ in $G_Q$ and a retraction by deformations $h:[0;1]\times U\rightarrow U$ of $U$ onto $H_q\times\{q\}$ such that:\begin{enumerate}
		\item The image of $U\cap (H)_\gl$ by $h_t$ is contained in itself at every time $t$ in $[0;1]$;
		\item If $y$ and $z$ are two points of $U$ lying above the same point of $Q$, the paths $t\mapsto||h_t(y)||$ and $t\mapsto||h_t(z)||$ are identical.
	\end{enumerate}
\end{prop}

\begin{proof}
	Up to replacing $\Cell Q$ by its barycentric subdivision \cite[Theorem~1.7]{lundell_topology_1969}, we may assume $\Cell Q$ to be a simplicial complex. Since $H(e_0)$ is a closed subgroup of $G$ it acts properly on $G$, cf. \cite[Chapitre~III, \S 4, \textnumero 1, Exemple 1]{bourbaki_topologie_2007-1}. Thus, we can find an $H(e_0)$-equivariant tubular neighbourhood $T$ of $H(e_0)$ using a slice, cf. \cite[Theorem~B.24]{guillemin_moment_2002}. We denote by $k$ a retraction by deformations of $T$ onto $H(e_0)$. We define $U$ to be $T\times S$ where $S$ is the open star of $e_0$ in $\Cell Q$. We also consider the retraction by deformations $l$ of $S$ onto $\{q_0\}$ defined by the formula $tq_0+(1-t)q$. It is well defined for we assumed $\Cell Q$ to be a simplicial complex. We note that, for all $q$ in $S$, $l_t(q)$ only visits faces of the cell that contains $q$ as $t$ varies. Then, by construction, $h_t:(g;q)\mapsto (k_t(g);l_t(q))$ defines a retraction by deformation of $U$ onto $H(e_0)\times \{q_0\}$. Let $(g;q)$ be a point of $U\cap (H)_\gl$. It means that if $e$ is the cell that contains $q$, then $g$ belongs to $H(e)$. Since $H$ is a cosheaf of subgroups of $G$, $H(e)$ is a subgroup of $H(e_0)$. Thus, $h_t(g;q)$ is simply $(g;l_t(q))$. The assumption on $l$ implies $h_t(U\cap (H)_\gl)$ is contained in $U\cap (H)_\gl$. Moreover, the same remark shows that $U$ is $(H)_\gl$-equivariant for $T$ is assumed to be $H(e_0)$-equivariant. The second property is immediate for $||h_t(g;q)||$ is, by definition, $l_t(q)$.
\end{proof}

\begin{rem}
	The gluing construction is some kind of homotopy colimit of the cosheaf.
\end{rem}

\begin{prop}\label{prop:hp_colim_cellular}
	Let $H$ be a cosheaf of Lie groups on a regular cellular structure $\Cell\,Q$ of $Q$. There is a natural isomorphism:
	\begin{equation*}
		\underset{e\in\Cell Q}{\colim} \pi_0\big(H(e)\big) \cong \hat{\pi}_0\big((H)_\gl\big). 
	\end{equation*}
\end{prop}

\begin{proof}
	Let $D$ be a discrete group and $f:(H)_\gl\rightarrow D_Q$ be a morphism of $Q$-groups. We can consider the following composition of continuous map:
	\begin{equation*}
		\begin{tikzcd}
			\displaystyle\coprod_{e\in \Cell Q} H(e)\times \bar{e} \ar[r,"g"]& (H)_\gl \ar[r,"f"]& D_Q \ar[r,"\pr_D"]& D.
		\end{tikzcd}
	\end{equation*}
	It is locally constant for $D$ is discrete and the source is locally connected. It defines a maps $f_e:\pi_0(H(e))\rightarrow D$ for every cell $e$. These maps are group morphism. Indeed, if $q$ is a point of the relative interior of $e$, and $h_1,h_2$ belong to $H(e)$, the product $g(h_1;q)\cdot g(h_2;q)$ equals $g(h_1h_2;q)$. Hence:
	\begin{equation*}
		\big(f_e\big([h_1]\cdot[h_2]\big);q\big)=fg(h_1h_2;q)=fg(h_1;q)\cdot fg(h_2;q) = \big(f_e[h_1]\cdot f_e[h_2];q\big).
	\end{equation*}
	Furthermore, given the definition of the gluing, we find that the following diagram commutes for every couple of adjacent cells $e_1\leq e_2$:
	\begin{equation*}
		\begin{tikzcd}
			\pi_0\big(H(e_1)\big) \ar[rd,"f_{e_1}"]\\
			& D \\
			\pi_0\big(H(e_2)\big) \ar[uu,"t_*"] \ar[ru,"~\;f_{e_2}" below]
		\end{tikzcd}
	\end{equation*}
	where $t:H(e_2)\rightarrow H(e_1)$ denotes the transition morphism of $H$.  Thus, $f$ uniquely yields a cosheaf morphism $\pi_0(H)\rightarrow D$. Reciprocally, every such cosheaf morphism yields a morphism of $Q$-groups $(H)_\gl\rightarrow D_Q $. Now we can conclude by remarking that the set of cosheaf morphisms $\pi_0(H)\rightarrow D$ is naturally isomorphic to the set of morphisms $\colim_{e\in \Cell Q} \pi_0(H(e)) \rightarrow D$ by the universal property of the colimit.
\end{proof}

\begin{prop}
	Let $\Cell Q$ be a cellular structure on $Q$. The gluing functor from the category of cosheaves of discrete spaces (resp. discrete groups) on $\Cell Q$ to the category of $Q$-spaces (resp. $Q$-groups) is fully faithful.
\end{prop}

\begin{proof}
	The faithfulness is very easy to derive given the gluing construction, hence we will only show the fullness. Let $X$ and $Y$ be two cosheaves of discrete spaces on $\Cell Q$ and $f:(X)_\gl\rightarrow (Y)_\gl$ be a morphism of $Q$-space. Let $e$ be an open cell of $Q$. Since $f$ commutes with the ground projection, it maps $X(e)\times e$ continuously on $Y(e)\times e$ as a map of $e$-spaces. Thus, $e$ being connected and $X(e)$ being dicrete, the restriction of $f$ to $X(e)\times e$ has the form $(x;q)\mapsto (f_e(x);q)$ for a uniquely defined map $f_e:X(e)\rightarrow Y(e)$. Note that, by construction, the inclusion of $X(e)\times e$ in $(X)_\gl$ extends by continuity as a map $u_e:X(e)\times \bar{e}\rightarrow (X)_\gl$ and likewise as a map $v_e$ for $Y$. Moreover, by density, we must have the following identity:
	\begin{equation*}
		f\circ u_e=v_e\circ(f_e\times\id_{\bar{e}}).
	\end{equation*}
	Therefore, continuity imposes that we have the following commutative diagram for every pair of adjacent cells $e_1\leq e_2$:
	\begin{equation*}
		\begin{tikzcd}
			X(e_1)\times\bar{e}_1 \ar[ddd,"f_{e_1}\times\id" swap]  \ar[rrrrd, bend left, "u_{e_1}"] \\
			& X(e_2)\times\bar{e}_1 \ar[d,"f_{e_2}\times\id"] \ar[rr,"\id\times(\bar{e}_1\subset \bar{e}_2)"] \ar[lu,"t^X_{e_1,e_2}\times \id" swap] && X(e_2)\times\bar{e}_2 \ar[d,"f_{e_2}\times\id"] \ar[r,"u_{e_2}"] & (X)_\gl \ar[d, "f"] \\
			 & Y(e_2) \times\bar{e}_1 \ar[rr,"\id\times(\bar{e}_1\subset \bar{e}_2)" swap] \ar[ld,"t^Y_{e_1,e_2}\times \id"] && Y(e_2)\times\bar{e}_2 \ar[r,"v_{e_2}" swap] & (Y)_\gl \\
			 Y(e_1)\times\bar{e}_1  \ar[rrrru, bend right, "v_{e_1}" swap] 
		\end{tikzcd}
	\end{equation*}
	where $t$ denotes the transition map of the respective cosheaves. Hence, the collection $(f_e)_{e\in \Cell Q}$ forms a morphism of cosheaves $X\rightarrow Y$ yielding $f$. If both $X$ and $Y$ are cosheaves of groups and $f$ is a $Q$-group morphism, we verify without difficulty that each $f_e$ is a group morphism so that the proposition follows for group objects.  
\end{proof}

\subsection{Cohomology}
We introduce the $Q$-groups cohomology by following the classical construction. These groups, especially in small degrees, will foster useful obstruction classes. In addition of giving the definition, we provide elementary results that compare these groups with the classical group cohomology for the types of coefficients that will be of most interest in this article.  

\begin{dfn}
	Let $G\rightarrow Q$ be a group. A \emph{$G$-module} is an Abelian $Q$-group $M\rightarrow Q$ endowed with a continuous action $a:G\times_Q M \rightarrow M$ such that the following diagram commutes:
	\begin{equation*}
		\begin{tikzcd}
			G\times_Q M\times_Q M \ar[rr,"\id_G\times m_M"] \ar[d,"a^2"] && G\times_Q M \ar[d,"a"] \\
			M\times_Q M \ar[rr,"m_M" swap] && M
		\end{tikzcd}
	\end{equation*}
	where $m_M$ is the group law of $M$, and $a^2$ is the diagonal action of $G$ on $M\times_Q M$ induced by $a$.
\end{dfn}

\begin{dfn}[Group Cohomology]\label{def:Qgroup_cohomo}
	Let $G\rightarrow Q$ be a group and $M\rightarrow Q$ be a $G$-module. For every integer $k\geq 0$, a \emph{$k$-cochain} $\alpha$ on $G$ with values in $M$, is a continuous map of $Q$-spaces $\alpha:G^k_Q\rightarrow M$. They form a group under pointwise addition. We denote it by $C^k_Q(G;M)$. Just as in classical group cohomology, these groups are endowed with a differential $\df:C^k_Q(G;M)\rightarrow C^{k+1}_Q(G;M)$ defined as follows:
	\begin{equation*}
		\df\alpha(g_1;...;g_{k+1}):=g_1\cdot\alpha(g_2;...;g_{k+1})+\sum_{i=1}^k(-1)^i\alpha(g_1;...;g_{i-1};g_ig_{i+1};g_{i+2};...;g_{k+1}) + (-1)^{k+1}\alpha(g_1;...;g_k).
	\end{equation*}
	The cohomology of this complex is denoted by $H^k_Q(G;M)$.
\end{dfn}

\paragraph{Discrete Modules}
The category of $G$-module without further assumptions is not Abelian in general. This is already the case when $G$ is the trivial group over the point. Hence, we restrict our attention to a better behaved subcategory of $G$-modules.

\begin{dfn}
	Let $G$ be a $Q$-group. A $G$-module is called \emph{discrete} if its underlying $Q$-group of the form $M_Q$ for some discrete Abelian group $M$.
\end{dfn} 

\begin{prop}
	Let $G$ be a group over $Q$. The category of discrete $G$-modules is equivalent to the category of discrete $\hat{\pi}_0(G)$-modules.
\end{prop}

\begin{proof}
	Let $M_Q$ be a discrete $G$-module. We denote the group of automorphisms of $M$ by $\Aut(M)$. The following map is continuous by construction:
	\begin{equation*}
		\begin{array}{rcl}
			b:G\times M & \longrightarrow & M\\
			(g;m) & \longmapsto & \pr_M\big(g\cdot(m;||g||)\big).
		\end{array}
	\end{equation*}
	Following \cite[Chapitre~X, \S 3, \textnumero 4, Théorème 3]{bourbaki_topologie_2007}, the map:
	\begin{equation*}
		\begin{array}{rcl}
			G & \longrightarrow & \Aut(M)\\
			g & \longmapsto & [m\mapsto b(g;m)].
		\end{array}
	\end{equation*}
	is continuous for the compact-open topology of $\Aut(M)$ for $M$ is locally compact. Since $M$ is discrete, $\Aut(M)$ endowed with the compact-open topology is discrete. Hence, the following map $f$ is a morphism of $Q$-groups:
	\begin{equation*}
		\begin{array}{rcl}
			f:G & \longrightarrow & \Aut(M)_Q\\
			g & \longmapsto & \big([m\mapsto b(g;m)];||g||\big).
		\end{array}
	\end{equation*}
	Reciprocally, every such morphism $f$ endows the discrete Abelian $Q$-group $M_Q$, with a structure of $G$-module. Following Proposition~\ref{prop:adjunction}, $f$ yields a unique group morphism $f_*:\hat{\pi}_0(G)\rightarrow \Aut(M)$, and this procedure endows $M$ with the structure of discrete $\hat{\pi}_0(G)$-module. A routine check ensures that it defines a functor:
	\begin{equation*}
		\begin{array}{rcl}
			\hat{\pi}_0:\big(\textnormal{Discrete } G\textnormal{-Modules}\big) & \longrightarrow & \big(\textnormal{Discrete } \hat{\pi}_0(G)\textnormal{-Modules}\big)\\[5pt] 
			\big(M_Q;f:G\rightarrow \Aut(M)_Q\big) & \longmapsto & \big(M;f_*:\hat{\pi}_0(G)\rightarrow \Aut(M)\big).
		\end{array}
	\end{equation*}
	Moreover, the unit of the adjunction $\eta_G:G\rightarrow \hat{\pi}_0(G)_Q$ defines a functor in the other direction:
	\begin{equation*}
		\begin{array}{rcl}
			(-)_Q:  \big(\textnormal{Discrete } \hat{\pi}_0(G)\textnormal{-Modules}\big)&\longrightarrow & \big(\textnormal{Discrete } G\textnormal{-Modules}\big)\\[5pt]
			\big(M;g:\hat{\pi}_0(G)\rightarrow \Aut(M)\big) & \longmapsto & \big(M_Q;g_Q\circ\eta_G:G\rightarrow \Aut(M)_Q \big). 
		\end{array}
	\end{equation*}
	The adjunction, cf. Proposition~\ref{prop:adjunction}, implies that they yield an equivalence. 
\end{proof}

\begin{prop}
	Let $G$ be a {\simul} $Q$-group. The family of cohomology functors $[M\mapsto H^k_Q(G;M_Q)]_{k\geq0}$ has a canonical structure of $\delta$-functor on the category of $\hp(G)$-modules.
\end{prop}

\begin{proof}
	Let us consider an exact sequence of $\hp(G)$-modules $0 \longrightarrow K \overset{f}{\longrightarrow} L \overset{g}{\longrightarrow} M \longrightarrow 0$.	It yields an exact sequence of cochain complexes:
	\begin{equation*}
		0 \longrightarrow C^*_Q(G;K_Q) \overset{(f_Q)_*}{\longrightarrow} C^*_Q(G;L_Q) \overset{(g_Q)_*}{\longrightarrow} C^*_Q(G;M_Q).
	\end{equation*}
	Since for every non-negative integer $n$, the $n$-fold fibre product of $G$ with itself over $Q$ is locally connected, $(g_Q)_*$ is surjective. Thus, we can consider the usual cohomological exact sequence:
	\begin{equation*}
		\cdots \longrightarrow H^n_Q(G;K_Q) \overset{(f_Q)_*}{\longrightarrow} H^n_Q(G;L_Q) \overset{(g_Q)_*}{\longrightarrow} H^n_Q(G;M_Q) \overset{\df}{\longrightarrow} H^{n+1}_Q(G;K_Q) \longrightarrow \cdots
	\end{equation*}
	and its differentials endow the cohomology functors with the structure of a $\delta$-functor.
\end{proof}

\begin{dfn}
	Let $G$ be a $Q$-group and $M$ be a $\hp(G)$-module. We have a natural transformation:
	\begin{equation*}
		\eta^*_G:\big[M\mapsto C^*(\hp(G);M) \big] \rightarrow \big[M\mapsto C^*_Q(\hp(G)_Q;M_Q) \big] \rightarrow \big[M\mapsto C^*_Q(G;M_Q)\big].
	\end{equation*}
	Where the first transformation is the obvious natural isomorphism, and the second one is the pullback induced by the adjunction unit $\eta_G:G\rightarrow \hp(G)_Q$. We also denote the morphism it induces in cohomology by the same symbol. It is a morphism of $\delta$-functors.
\end{dfn}

\begin{lem}\label{lem:iso_h0}
	The morphism $\eta^*_G:H^0(\hat{\pi}_0(G);M)\rightarrow H^0_Q(G;M_Q)$ is an isomorphism.
\end{lem}

\begin{proof}
	By construction, $\eta_G^*:C^0(\hp(G);M) \rightarrow C^0_Q(G;M_Q)$ is an isomorphism of $\hp(G)$-modules. Recall that the first cochain group is simply $M$. By direct application of the definition, we find that an element $m$ of $M$ is a cocycle in $C^0_Q(G;M_Q)$ if and only if $\eta_G(g)\cdot m$ equals $m$ for every $g$ in $G$. However, by definition, the elements $\eta_G(g)$ for all $g$ in $G$ span the group $\hp(G)$. Hence, $m$ is a cocycle in $C^0_Q(G;M_Q)$ if and only if it is $\hp(G)$-invariant, i.e. a cocycle in $C^0(\hp(G);M)$.
\end{proof}

\begin{cor}\label{cor:split_mono}
	If $G$ is a {\simul} $Q$-group, the cohomological natural transformation of $\delta$-functors $\eta^*_G$ is a naturally split monomorphism and there is only one such natural splitting. 
\end{cor}

\begin{proof}
	Lemma~\ref{lem:iso_h0} implies that we can consider the inverse $\rho_G$ of $\eta^*_G:H^0(\hat{\pi}_0(G);M)\rightarrow H^0_Q(G;M_Q)$. This is a natural transformation. The universal character of derived functors implies that we can uniquely extend $\rho_G$ into a morphism of $\delta$-functors:
	\begin{equation*}
		\rho_G:\big[M\mapsto H^*_Q(G;M_Q)\big] \rightarrow \big[M\mapsto H^*(\hp(G);M) \big],
	\end{equation*}
	cf. \cite[Theorem 2.4.7 and \S 2.5.1]{weibel_introduction_1994}. We note that $\rho_G\circ\eta_G^*$ is a morphism of $\delta$-functors yielding the identity in degree $0$. But, by universality again, there can only be one such transformation: the identity. Hence, we found our unique natural retraction.
\end{proof}

\begin{rem}\label{rem:cohomology_constant_groups}
	If the $Q$-group $G$ is of the form $H_Q$ for some Lie group $H$, then $\eta_G^*$ is a natural isomorphism. This is even true at the cochains level for connected components of the $k$-fold fibre product of $G$ over $Q$ is in bijection with $\pi_0(H)^k$ which is the same as $\hp(G)^k$.
\end{rem}

\begin{lem}\label{lem:iso_h1}
	The morphism $\eta^*_G:H^1(\hat{\pi}_0(G);M)\rightarrow H^1_Q(G;M_Q)$ is an isomorphism. 
\end{lem}

\begin{proof}
	Let $M$ be a discrete $\hat{\pi}_0(G)$-module. We will consider its holomorph $\textnormal{Hol}(M)$ given as the semi-direct product of $\Aut(M)$ and $M$ with obvious action. We denote by $\df$ the projection of $\textnormal{Hol}(M)$ onto $\Aut(M)$, and by $\sigma$ a section of $\df$. Let $a$ denote the action of $\hat{\pi}_0(G)$ onto $M$.  Let us consider the set $B$ of $Q$-group morphisms $b:G\rightarrow \textnormal{Hol}(M)_Q$ (resp. $\hat{B}$ of group morphisms $\hat{b}:\hat{\pi}_0(G)\rightarrow \textnormal{Hol}(M)$) that make the following diagram commute:
	\begin{equation*}
		\begin{tikzcd}
			&& 0 \ar[d] && 0 \ar[d]\\
			&& M_Q \ar[d] && M \ar[d]\\
			&& \textnormal{Hol}(M)_Q \ar[d,"\df_Q"] & \textnormal{resp.} & \textnormal{Hol}(M) \ar[d,"\df" swap]\\
			G \ar[rr,"a_Q\circ \eta_G" swap] \ar[rru,"b"]&& \Aut(M)_Q \ar[d] && \Aut(M) \ar[d] && \hat{\pi}_0(G) \ar[ll,"a"] \ar[llu,"\hat{b}" swap] \\
			&& 1 && 1
		\end{tikzcd}
	\end{equation*}
	We have the following commutative diagram:
	\begin{equation*}
		\begin{tikzcd}
			\hat{\beta} \ar[d,mapsto]& Z^1\big(\hat{\pi}_0(G);M \big) \ar[d,"\simeq" swap] \ar[r,"\eta^*_G"]& Z^1_Q(G;M_Q) \ar[d,"\simeq"] & \beta \ar[d,mapsto]\\ 
			{[g\mapsto \hat{\beta}(g)\cdot \sigma a(g)]} & \hat{B} \ar[r] & B & {[g\mapsto \beta(g)\cdot(\sigma a)_Q\circ\eta_G(g)]} \\
			& \hat{b} \ar[r,mapsto] & \hat{b}_Q\circ\eta_G
		\end{tikzcd}
	\end{equation*}
	The fact that the vertical arrows are bijective is a very classical matter of group cohomology. Since $\textnormal{Hol}(M)$ is a discrete group, the defining universal property of $\hat{\pi}_0$ and some very straightforward verifications imply that the bottom arrow is a bijection. In particular, $\eta_G^*$ is surjective on cocycles thus \emph{a fortiori} surjective in cohomology. Since it is injective as well, cf. Corollary~\ref{cor:split_mono}, it is an isomorphism.
\end{proof}

\paragraph{Defect} In this last paragraph, we introduce the \emph{defect} of a {\simul} $Q$-group $G$. It is, by definition, the homology of what should be the bar resolution. We say should for this complex will not be acyclic in general. It will measure how, given a $\hp(G)$-module $M$, the cohomology of $G$ with coefficients in $M$ diverges from the cohomology of $\hp(G)$ with coefficients in $M$.  

\begin{dfn}\label{dfn:defct}
	Let $G\rightarrow Q$ be {\simul} $Q$-group. We define the chain complex of free $\Z[\hp(G)]$-modules:
	\begin{equation*}
		\partial : H_0\big(G^{k+1}_Q;\Z[\hat{\pi}_0(G)]\big)  \longrightarrow  H_0\big(G^k_Q;\Z[\hat{\pi}_0(G)]\big),
	\end{equation*}
	where $H_0$ indicates singular homology. The $k$\textsuperscript{th} group has a $\Z[\hp(G)]$-basis made of elements of the form $[g_1;...;g_k]$ which denotes the connected component of the element $(g_1;...;g_k)$ of $G^k_Q$. The boundary operator is defined by the following formula: 
	\begin{equation*}
		\partial [g_0;\cdots;g_k] := \eta_G(g_0)[g_1;...;g_k] + \sum_{i=1}^{k}(-1)^i[g_0;\cdots;g_{i-1}g_i;\cdots;g_k] + (-1)^{k+1}[g_0;\cdots;g_{k-1}].
	\end{equation*}
	We denote its homology by $\Delta_k(G)$ and call it the \emph{defect} of $G$. This is a $\hat{\pi}_0(G)$-module. It is functorial in $G$.
\end{dfn}

\begin{lem}\label{lem:Delta_0}
	Let $G\rightarrow Q$ be {\simul} $Q$-group. $\Delta_0(G)$ is isomorphic to the $\hp(G)$-module $\Z$ with trivial action. 
\end{lem}

\begin{proof}
	Following the definition, $\Delta_0(G)$ is the quotient of $\Z[\hp(G)]$ by the left ideal $( \eta_G(g)-1:g\in G)$. Since $\{\eta_G(g):g\in G\}$ generates $\hp(G)$, the identity:
	\begin{equation*}
		\left(\prod_{i=1}^ng_i\right) - 1 = \sum_{k=1}^n \prod_{i=1}^{n-k}g_i(g_{n+1-k}-1),
	\end{equation*}
	is enough to conclude that $( \eta_G(g)-1:g\in G)$ is the augmentation ideal. Hence, the lemma follows. 
\end{proof}

\begin{prop}\label{prop:iso_defect_cohomo}
	Let $G$ be a {\simul} $Q$-group and $M$ be a $\hp(G)$-module. We have a natural isomorphism of cochain complexes:
	\begin{equation*}
		C^k_Q(G;M_Q) \overset{\cong}{\longrightarrow} \Hom_{\Z[\hp(G)]}\big( H_0\big(G^k_Q;\Z[\hp]\big) ; M \big) .
	\end{equation*}
\end{prop}

\begin{proof}
	For conciseness we denote $\hp(G)$ by $\hp$. Let $k$ be a non-negative integer, we consider the natural group morphism:
	\begin{equation*}
		\varphi: C^k_Q(G;M_Q) \longrightarrow \Hom_{\Z[\hp]}\big( H_0\big(G^k_Q;\Z[\hp]\big) ; M \big),
	\end{equation*}
	that sends a cochain $\alpha$ of $C^k_Q(G;M_Q)$ to the $\Z[\hp]$-linear morphism defined by the following property: 
	\begin{equation*}
		(\varphi\alpha)[g_1;...;g_k] = \pr_M\big(\alpha(g_1;...;g_k)\big).
	\end{equation*}
	It is well defined and bijective for cochains are continuous, $M$ is discrete, and $G^k_Q$ is locally path connected. Moreover, it is a cochain complex morphism as indicated by the following computation:
	\begin{align*}
		(\partial^*\varphi\alpha)[g_1;...;g_{k+1}]&=(\varphi\alpha)\Big( \eta_G(g_1)[g_2;...;g_{k+1}] + \sum_{i=1}^{k+1}(-1)^i[g_1;...;g_{i-1}g_i;...;g_{k+1}]+ (-1)^{k+2}[g_1;...;g_k]\Big)\\
		&= \eta_G(g_1)\pr_M\big(\alpha(g_2;...;g_{k+1})\big) + \sum_{i=1}^{k+1}(-1)^i\pr_M\big(\alpha(g_1;...;g_{i-1}g_i;...;g_{k+1})\big) \\& \hspace{8.25cm} + (-1)^{k+2}\pr_M\big(\alpha(g_1;...;g_k)\big) \\
		&= \pr_M\Big((\eta_G(g_1);||g_1||)\cdot\alpha(g_2;...;g_{k+1}) + \sum_{i=1}^{k+1}(-1)^i\alpha(g_1;...;g_{i-1}g_i;...;g_{k+1}) \\& \hspace{9cm} + (-1)^{k+2}\alpha(g_1;...;g_k)\Big) \\
		&= \pr_M\big( \df\alpha(g_1;...;g_{k+1})\big).  \\
	\end{align*}
	And we have a natural isomorphism of cochain complexes. 
\end{proof}

\begin{prop}\label{prop:tool_defect_cellular}
	Let $G$ be a cosheaf of Lie groups over a cellular structure $\Cell Q$ of $Q$. We have an isomorphism of chain complexes:
		\begin{equation*}
			H_0\big( (G^k)_\gl;\Z[\hp((G)_\gl)]\big) \overset{\cong}{\longrightarrow} \underset{e\in \Cell Q}{\colim}\, C_k(\pi_0(G(e));\Z[\hp((G)_\gl)]\big),
		\end{equation*}
	where the right hand side denotes the colimit of the bar complexes computing the group homology of the discrete groups $\pi_0(G)$ with coefficients in $\Z[\hp((G)_\gl)]$.
\end{prop}

\begin{proof}
	This follows from the fact that $\pi_0((G^k)_\gl)$ is the set colimit of the diagram of sets induced by the cosheaf $[e\mapsto \pi_0(G(e))^k]$. We have the following chain of natural isomorphisms:
	\begin{align*}
			H_0\big( (G^k)_\gl;\Z[\hp((G)_\gl)]\big) &\cong \underset{e\in \Cell Q}{\colim}\, H_0\big( G(e)^k;\Z[\hp((G)_\gl)]\big) \\
			&\cong \underset{e\in \Cell Q}{\colim}\, \Z[\pi_0(G(e))^k]\otimes_\Z  \Z[\hp((G)_\gl)] \\
			&\cong \underset{e\in \Cell Q}{\colim}\, \Z[\pi_0(G(e))^{k+1}]\otimes_{\Z[\pi_0(G(e))]}  \Z[\hp((G)_\gl)] \\
			&\cong \underset{e\in \Cell Q}{\colim}\, C_k(\pi_0(G(e));\Z[\hp((G)_\gl)]).
		\end{align*}
	Furthermore, we can easily see that the last isomorphism is obtained as the colimit of the following family of morphisms:
	\begin{equation*}
		\begin{array}{rcl}
			C_k\big(\pi_0(G(e));\Z[\hp((G)_\gl)]\big) & \longrightarrow & H_0\big( (G^k)_\gl;\Z[\hp((G)_\gl)]\big) \\
			(g_1,...,g_k) & \longmapsto & [(g_1;b_e);...;(g_k;b_e)],
		\end{array}
	\end{equation*}
	where $b_e$ denotes the “\,barycentre\,” of the cell $e$. Those morphisms are chain complexes morphisms, this is straightforwardly checked. Therefore, we have the desired natural chain complex isomorphism.  
\end{proof}

\begin{prop}\label{prop:defct_spectral_sequence}
	For every $Q$-group $G$ and every $\hp(G)$-module $M$ there is a functorial spectral sequence such that:
	\begin{equation*}
		E^{p,q}_2(G;M)=\Ext^p_{\Z[\hat{\pi}_0(G)]}(\Delta_q(G);M) \Rightarrow H^{p+q}_Q(G;M_Q).
	\end{equation*}
\end{prop}

\begin{proof}
	In this proof we will denote $\hp(G)$ by $\hp$. Let us consider now an injective resolution $0\rightarrow M\rightarrow I$ of $M$ in the category of $\Z[\hp]$-modules. It allows us to form the bicomplex $(C^{p,q},\partial^*;\df)$ defined as follows:
	\begin{equation}\label{eq:dfn_bicomplex}
		C^{p,q}:=\Hom_{\Z[\hp]}\big( H_0\big(G^{p}_{Q};\Z[\hp]\big) ; I^q \big) \quad \textnormal{with} \quad \partial^*:C^{p,q}\rightarrow C^{p+1,q} \quad \textnormal{and} \quad \df:C^{p,q}\rightarrow C^{p,q+1}.
	\end{equation}
	The total differential $D$ is defined on $C^{p,q}$ by $\partial^*+(-1)^p\df$. We denote the cohomology of the total complex by $H^n$ for all non-negative integers $n$. As usual, there are two spectral sequences that compute the total cohomology. The vertical spectral sequence: $F_r^{p,q} \Rightarrow H^{p+q}$, and the horizontal spectral sequence\footnote{N.b. the classical index conventions are well suited to the horizontal spectral sequence i.e. $F^{p,q}_0$ is $C^{p,q}$. For the vertical spectral sequence, these conventions switch the indices, i.e. $E^{p,q}_0$ is $C^{q,p}$.}: $E_r^{p,q} \Rightarrow H^{p+q}$. We can use the vertical spectral sequence to show that the total cohomology is isomorphic to the cohomology of $G$ with values in $M_Q$. Indeed, $F_1^{p,q}$ is the $q$\textsuperscript{th} cohomology group of the complex:
	\begin{equation*}
		\cdots \overset{\df}{\longrightarrow} \Hom_{\Z[\hp]}\big( H_0\big(G^{p}_{Q};\Z[\hp]\big) ; I^q \big) \overset{\df}{\longrightarrow} \Hom_{\Z[\hp]}\big( H_0\big(G^{p}_{Q};\Z[\hp]\big) ; I^{q+1} \big) \overset{\df}{\longrightarrow} \cdots
	\end{equation*}
	Since $H_0(G^{p}_{Q};\Z[\hp])$ is free, it is a projective $\Z[\hp]$-module. Hence, $I$ being a resolution of $M$, we find that $F^{p,q}_r$ collapses on the first page:
	\begin{equation*}
		F_1^{p,q}=\left\{\begin{array}{cc}
			\Hom_{\Z[\hp]}\big( H_0\big(G^{p}_{Q};\Z[\hp]\big) ; M \big) & q=0 \\
			0 & q>0.
			\end{array}\right.
	\end{equation*}
	Thus, using Proposition~\ref{prop:iso_defect_cohomo}, we deduce that $H^n$ is isomorphic to $H^n_Q(G;M_Q)$. Now, let us show that $E_1^{p,q}$ is isomorphic to $\Hom_{\Z[\hp]}( \Delta_q(G);I^p)$. To that end, we denote by $C_q$ the group of chains $H_0\big(G^q_Q;\Z[\hp]\big)$, by $Z_q$ the subgroups of cycles, and by $B_q$ the subgroup of boundaries. We have the following isomorphisms:
	\begin{align*}
		\{\alpha\in C^{q,p}\;|\; \partial^*\alpha=0\} &= \{\alpha\in C^{q,p}\;|\; \alpha(b)=0, \forall b\in B_q\} \quad\cong \Hom_{\Z[\hp]}(C_q/B_q;I^p) \\
		\{\partial^*\alpha : \alpha\in C^{q-1,p}\} &= \{\alpha\in C^{q,p}\;|\; \alpha(z)=0, \forall z\in Z_q\} \quad\cong \Hom_{\Z[\hp]}(C_q/Z_q;I^p)
	\end{align*}
	The first one is straightforward, the second is a consequence of $I^p$ being injective. Indeed, if $\beta$ belongs to $\Hom_{\Z[\hp]}(C_q/Z_q;I^p)$, then we have the following commutative diagram:
	\begin{equation*}
		\begin{tikzcd}
			C_q/Z_q \ar[d,"\beta" swap] \ar[r,hookrightarrow,"\partial"] & C_{q-1} \ar[dl,dashed,"\exists\alpha"] \\
			I^p
		\end{tikzcd}
	\end{equation*}
	so that $\beta$ equals $\partial^*\alpha$. Moreover, the snake lemma applied to the diagram:
	\begin{equation*}
		\begin{tikzcd}
			0 \ar[r]& B_q \ar[r] \ar[d]& Z_q \ar[r] \ar[d]& \Delta_q(G) \ar[r] \ar[d]& 0 \\
			0 \ar[r]& C_q \ar[r]& C_q \ar[r]& 0 \ar[r]& 0
		\end{tikzcd}
	\end{equation*}
	provides the following exact sequence:
	\begin{equation}\label{seq:defect_as_kernel}
		0 \longrightarrow \Delta_q(G) \longrightarrow C_q/B_q \longrightarrow C_q/Z_q \longrightarrow 0.
	\end{equation}
	Thus, applying the exact functor $\Hom_{\Z[\hp]}(-;I^p)$ to (\ref{seq:defect_as_kernel}), yields the desired isomorphism between $E_1^{p,q}$ and $\Hom_{\Z[\hp]}(\Delta_q(G);I^p)$. From here, it is straightforward that $E_2^{p,q}$ is given by $\Ext^p_{\Z[\hp]}(\Delta_q(G);M)$.
\end{proof}

\begin{cor}\label{cor:exact_seq_h2}
	Let $G$ be a $Q$-group, we have a natural exact sequence for every $\hp(G)$-module $M$:
	\begin{equation*}
		0 \longrightarrow H^2(\hp(G);M) \overset{\eta_G^*}{\longrightarrow} H^2_Q(G;M_Q) \longrightarrow \Hom_{\Z[\hp(G)]}(\Delta_2(G);M) \longrightarrow 0.
	\end{equation*}
\end{cor}

\begin{proof}
	To derive the statement we will compare the spectral sequences $E^{p,q}_r(G;M)$ and $E^{p,q}_r(\hp(G)_Q;M)$. First, we note that the adjunction unit yields a chain complex morphism:
	\begin{equation*}
		(\eta_G)_*:H_0\big(G^k_Q;\Z[\hp(G)]\big) \rightarrow H_0\big(\hp(G)^k_Q;\Z[\hp(G)]\big),
	\end{equation*}
	whose image under the functor $\Hom_{\Z[\hp(G)]}(-;M)$ is precisely $\eta^*_G$. Then, it extends to a bicomplex morphism and thus to a morphism of spectral sequences, cf. (\ref{eq:dfn_bicomplex}). Now, the spectral sequence $E^{p,q}_r(\hp(G)_Q;M)$ degenerates on the second page for the defect of $\hp(H)_Q$ vanishes in positive degrees. Indeed, the chain complex $H_0\big(\hp(G)^*_Q;\Z[\hp(G)]\big)$ is the bar resolution of $\Z$ for $\hp(G)$. Thus, we find again the classical statement:
	\begin{equation*}
		H^p(\hp(G);M)=E_2^{p,0}(\hp(G)_Q;M)=\Ext^p_{\Z[\hp(G)]}(\Z;M).
	\end{equation*}
	The exact sequence of low degree, cf. \cite[Theorem 5.12]{cartan_homological_1956}, yields the following commutative diagram:
	\begin{equation*}
		\begin{tikzcd}
			0 \ar[d] &&& 0 \ar[d] \\
			\Ext^1_{\Z[\hp]}(\Z;M) \ar[d] \ar[rrr,"\textnormal{isomorphism}","\textnormal{by Lemma~\ref{lem:Delta_0}}" swap] &&& \Ext^1_{\Z[\hp]}(\Delta_0(G);M) \ar[d] \\
			H^1(\hp;M) \ar[d] \ar[rrr,"\eta^*_G","\textnormal{isomorphism by Lemma~\ref{lem:iso_h1}}" swap] &&& H^1_Q(G;M_Q) \ar[d] \\
			0 \ar[d] \ar[rrr] &&& \Hom_{\Z[\hp]}(\Delta_1(G);M) \ar[d] \\
			\Ext^2_{\Z[\hp]}(\Z;M) \ar[d,"\cong" swap] \ar[rrr,"\textnormal{isomorphism}","\textnormal{by Lemma~\ref{lem:Delta_0}}" swap] &&& \Ext^2_{\Z[\hp]}(\Delta_0(G);M) \ar[d] \\
			H^2(\hp;M) \ar[rrr,"\textnormal{injective}","\textnormal{by Corollary~\ref{cor:split_mono}}" swap] &&& H^2_Q(G;M_Q)
		\end{tikzcd}
	\end{equation*}
	Hence, we deduce that $\Hom_{\Z[\hp]}(\Delta_1(G);M)$ vanishes. This is true for all $M$, thus $\Delta_1(G)$ must be trivial. It follows that $E^{p,1}_2(G;M)$ vanishes for all integers $p$. Thus, we can apply \cite[Theorem~5.12]{cartan_homological_1956} a second time with $n$ equal to $2$ to find the following commutative diagram:
	\begin{equation*}
		\begin{tikzcd}
			0 \ar[d] &&& 0 \ar[d] \\
			\Ext^2_{\Z[\hp]}(\Z;M) \ar[d] \ar[rrr,"\textnormal{isomorphism}","\textnormal{by Lemma~\ref{lem:Delta_0}}" swap] &&& \Ext^2_{\Z[\hp]}(\Delta_0(G);M) \ar[d] \\
			H^2(\hp;M) \ar[d] \ar[rrr,"\textnormal{injective}","\textnormal{by Corollary~\ref{cor:split_mono}}" swap] &&& H^2_Q(G;M_Q) \ar[d] \\
			0 \ar[d] \ar[rrr] &&& \Hom_{\Z[\hp]}(\Delta_2(G);M) \ar[d,"d_3"] \\
			\Ext^3_{\Z[\hp]}(\Z;M) \ar[d,"\cong" swap] \ar[rrr,"\textnormal{isomorphism}","\textnormal{by Lemma~\ref{lem:Delta_0}}" swap] &&& \Ext^3_{\Z[\hp]}(\Delta_0(G);M) \ar[d] \\
			H^3(\hp;M) \ar[rrr,"\textnormal{injective}","\textnormal{by Corollary~\ref{cor:split_mono}}" swap] &&& H^3_Q(G;M_Q)
		\end{tikzcd}
	\end{equation*}
	Hence, the morphism $d_3$ is trivial and we find the desired short exact sequence.
\end{proof}

\begin{rem}\label{rem:vanishing_delta1}
	The proof of Corollary~\ref{cor:exact_seq_h2} yields the vanishing of $\Delta_1(G)$ for any $Q$-groups $G$.
\end{rem}

\begin{ex}[The Quotient]
	Let us consider a $[0;1]$-group $G$ obtained as the gluing of a cosheaf of groups $\Gamma$ on the usual cell structure $\{\{0\};\{1\};[0;1]\}$. It is defined by $\Gamma(\{0\})$ being trivial, $\Gamma(\{1\})$ being $\Z/4$, $\Gamma([0;1])$ being $\Z/2$, and the transition map $\Gamma([0;1])\rightarrow \Gamma(\{1\})$ being the multiplication by $2$, cf. Figure~\ref{fig:gluing_quot}. This group has $\hp(G)$ equal to $\Z/2$, cf. Proposition~\ref{prop:hp_colim_cellular}. Its defect is non-trivial. Let us show that $\Delta_2(G)$ is $\Z/2$.
	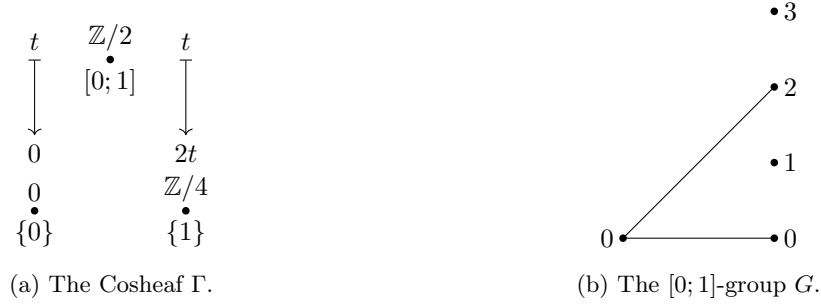
\begin{figure}[H]
		\centering
		\begin{subfigure}[t]{.45\textwidth}
			\centering
			\begin{tikzpicture}
				\fill (0,0) circle(.05) node[below]{$\{0\}$} node[above]{$0$};
				\fill (2,0) circle(.05) node[below]{$\{1\}$} node[above]{$\Z/4$};
				\fill (1,2) circle(.05) node[below]{$[0;1]$} node[above]{$\Z/2$};	
				\draw[|->] (2,2) node[above]{$t$} -- (2,1) node[below]{$2t$};
				\draw[|->] (0,2) node[above]{$t$} -- (0,1) node[below]{$0$};
			\end{tikzpicture}
			\caption{The Cosheaf $\Gamma$.}
		\end{subfigure}
		\begin{subfigure}[t]{.45\textwidth}
			\centering
			\begin{tikzpicture}
				\draw (2,2) -- (0,0) -- (2,0);
				\fill (0,0) circle(.05) node[left]{$0$};
				\fill (2,0) circle(.05) node[right]{$0$};
				\fill (2,1) circle(.05) node[right]{$1$};
				\fill (2,2) circle(.05) node[right]{$2$};
				\fill (2,3) circle(.05) node[right]{$3$};	
			\end{tikzpicture}
			\caption{The $[0;1]$-group $G$.}
		\end{subfigure}
		\caption{The Gluing $G$ of $\Gamma$.}
		\label{fig:gluing_quot}
	\end{figure}
	\noindent Following Proposition~\ref{prop:tool_defect_cellular}, we have the identity:
	\begin{equation*}
		H_0\big(G^k_Q;\Z[\hp(G)]\big)=\underset{e\in \Cell\, [0;1]}{\colim} C_k\big(\Gamma(e);\Z[\hp(G)]\big).
	\end{equation*}
	It can be rephrased as a short exact sequence of chain complexes:
	\begin{equation}\label{seq:defect_quot}
		0\rightarrow C_k\big(\Z/2;\Z[\hp(G)]\big) \rightarrow C_k\big(0;\Z[\hp(G)]\big) \oplus C_k\big(\Z/4;\Z[\hp(G)]\big) \rightarrow H_0\big(G_Q^k;\Z[\hp(G)]\big)\rightarrow 0.
	\end{equation}
	In the kernel complex, $\Z/2$ acts trivially on $\Z[\hp(G)]$. Hence, its homology is simply the homology of $\Z/2$ with coefficients in $\Z$ (with trivial action) tensorised by $\Z[\hp(G)]$. In the middle complex, we find the homology of the trivial group. This is the coefficient group in degree $0$ and trivial in positive degrees. Finally, we also have the homology of $\Z/4$ with coefficients in $\Z[\hp(G)]$. The group $\hp(G)$ is the quotient of $\Z/4$ by its $2$-torsion subgroup. Hence, $\Z[\hp(G)]$ is induced on $\Z/4$ by the $(\Z/4[2])$-module $\Z$ with trivial action, i.e. we have the following formula:
	\begin{equation*}
		\Z[\hp(G)]\cong \Z[\Z\underset{\Z\big[\Z/4[2]\big]}{/4]\otimes \Z}.
	\end{equation*}
	By Shapiro's Lemma, cf.~\cite[Shapiro's Lemma 6.3.2]{weibel_introduction_1994}, the homology $H_k(\Z/4;\Z[\hp(G)])$ is isomorphic to the homology $H_k(\Z/2;\Z)$. We summarised these results on Table~\ref{tab:group_homo_quot}.
	\begingroup
	\renewcommand{\arraystretch}{1.5}
	\begin{table}[H]
		\centering
		\begin{tabular}{r||l l l}
			$k$ & $H_k\big(\Z/2;\Z[\hp(G)]\big)$ & $H_k\big(0;\Z[\hp(G)]\big)$ & $H_k\big(\Z/4;\Z[\hp(G)]\big)$ \\ \hline \hline
			$0$ & $\Z[\hp(G)]$ & $\Z[\hp(G)]$ & $\Z$ \\ \hline
			$2l+1$ & $\Z/2[\hp(G)]$ & $0$ & $\Z/2$ \\ \hline
			$2l+2$ & $0$ & $0$ & $0$ 
		\end{tabular}
		\caption{Homology of the Complexes of the Sequence (\ref{seq:defect_quot}).}
		\label{tab:group_homo_quot}
	\end{table}
	\endgroup 
	\noindent And we find the following homological exact sequence:
	\begin{equation*}
		\begin{tikzcd}
			  \Z[\hp(G)] \ar[r] & \Z[\hp(G)]\oplus \Z \ar[r] \ar[d, phantom, ""{coordinate, name=Z}]& \Delta_0(G) \ar[r] & 0 \\
			  \Z/2[\hp(G)] \ar[r] & \Z/2 \ar[r] \ar[d, phantom, ""{coordinate, name=Z'}] & \Delta_1(G) \arrow[ull, rounded corners, to path={ -- ([xshift=2ex]\tikztostart.east)|- (Z) [near end]\tikztonodes -| ([xshift=-2ex]\tikztotarget.west) -- (\tikztotarget)}] \\
			  0 \ar[r] & 0 \ar[r] & \Delta_2(G)\arrow[ull, rounded corners, to path={ -- ([xshift=2ex]\tikztostart.east)|- (Z') [near end]\tikztonodes -| ([xshift=-2ex]\tikztotarget.west) -- (\tikztotarget)}]
		\end{tikzcd}
	\end{equation*}
	Since we know that $\Delta_1(G)$ is trivial, cf. Remark~\ref{rem:vanishing_delta1}, we find that $\Delta_2(G)$ is isomorphic, as an Abelian group, to $\Z/2$. We note that it also determines its structure of $\hp(G)$-module for $\Z/2$ has a trivial automorphism group. 
\end{ex}

\section{Polar Spaces and Polar Coordinates}

\begin{dfn}\label{dfn:polar_space}
	A \emph{polar space} is a couple $(G;X)$ made of a Lie group acting continuously and faithfully on a Hausdorff, connected, locally path-connected, and semi-locally simply connected space $X$ in such a way that:
	\begin{enumerate}
	\item The quotient $Q:=X/G$ is simply connected;
	\item The quotient map $||\cdot||:X\rightarrow Q$ admits a continuous section $x$;
	\item The map:
		\begin{equation*}
			\begin{array}{rcl}
				\phi_x:G\times Q & \longrightarrow & X \\
				(g;q) & \longmapsto & g\cdot x(q),
			\end{array}
		\end{equation*}
		is a quotient map\footnote{Recall that if $X$ and $Y$ are two topological spaces, a continuous map $f:X\rightarrow Y$ is said to be a \emph{quotient map} if it is surjective and a subset $A$ of $Y$ is open if and only if $f^{-1}(A)$ is open. In this case, it yields a homeomorphism between $Y$ and the quotient of $X$ by the equivalence relation induced by $f$.}. We call the map $\phi_x$ the polar coordinates of $(G;X)$ associated to the section $x$;
	\end{enumerate}
	The \emph{isotropy} of the section $x$ designates the closed subspace:
		\begin{equation*}
			H=\{(g;q)\in G\times Q\;|\; g\cdot x(q)=x(q)\},
		\end{equation*}
		This is a $Q$-subgroup of $G_Q$.
\end{dfn}

\begin{prop}\label{prop:orbit_quotient}
	Let $(G;X)$ be a polar space with quotient space $Q$, $x$ be a section of the quotient map, and $H$ be the isotropy of $x$. For all $q$ in $Q$, the orbit of $x(q)$ is homeomorphic to $G/H_q$. 
\end{prop}

\begin{proof}
	The inverse image by $\phi_x$ of the orbit of $x(q)$ is $G\times\{q\}$. This is a closed set that is saturated for the relation that defines the quotient. Thus, by \cite[Chapitre~I, \S 5, \textnumero 6, Corollaire~1]{bourbaki_topologie_2007-1}, $\phi_x$ yields a homeomorphism between $G/H_q$ and the orbit of $x(q)$. 
\end{proof}
	
\begin{dfn}\label{dfn:regular}
	Let $(G;X)$ be a polar space with quotient $Q$, $x:Q\rightarrow X$ be a section of the quotient map, and $H$ be the isotropy of $x$. We say that $x$ is \emph{regular} if for every $q$ in $Q$, there is an open neighbourhood $U$ of $x(q)$ and a retraction by deformations $h:[0;1]\times \phi_x^{-1}(U)\rightarrow \phi_x^{-1}(U)$ of $\phi_x^{-1}(U)$ onto $H_q\times\{q\}$ such that:\begin{enumerate}
		\item The image of $\phi_x^{-1}(U)\cap H$ by $h_t$ is contained in $\phi_x^{-1}(U)\cap H$ at every time $t$ in $[0;1]$;
		\item If $y$ and $z$ are two points of $\phi_x^{-1}(U)$ lying above the same point of $Q$, the paths $t\mapsto||h_t(y)||$ and $t\mapsto||h_t(z)||$ are identical.
	\end{enumerate}
	A polar space that admits a regular section will be called \emph{regular}.
\end{dfn}

\begin{rem}
	The quotient space of a regular polar space is necessarily locally contractible. 
\end{rem}

\begin{prop}
	The isotropy of a regular section of a polar space is {\simul}.
\end{prop}

\begin{proof}
	Let $(G;X)$ be a regular polar space with quotient $Q$, $H$ be the isotropy of a regular section, and $k$ be a positive integer. The formula:
	\begin{equation*} 
		h^k_t(g_1;...;g_k)\coloneqq (h_t(g_1);...;h_t(g_k)),
	\end{equation*}
	derived from a local retraction by deformations of Definition~\ref{dfn:regular}, provides local retractions by deformations of $H^k_Q$ onto its fibres. But its fibres are closed subgroups of a Lie group, hence locally contractible. Therefore, $H^k_Q$ is locally contractible for all positive integers $k$. Since $Q$ is also locally contractible, $H\rightarrow Q$ is {\simul}.
\end{proof}

\begin{prop}
	Let $Q$ be a simply connected space endowed with a cellular structure $\Cell Q$ and $G$ be a Lie group. If $H$ is a cosheaf of closed subgroups of $G$ then $(G;G_Q/(H)_\gl)$ is a regular polar space. 
\end{prop}

\begin{proof}
	This is a direct consequence of Propositions~\ref{prop:Hausdorff_LPC_gluing}, \ref{prop:commutativity_quotients_gluings}, \ref{prop:regular_subgroup}, and the fact that the quotient of a Lie group by a closed subgroup is a Hausdorff space, cf. \cite[Theorem~21.17]{lee_introduction_2012}. 
\end{proof}

\begin{ex}
	Let $n$ be a positive integer. We consider $\Sph^n$ the unit sphere of $\R^{n+1}$ endowed with the restriction of the action of $\textnormal{SO}(n)$ on the $n$ first coordinates. The quotient space is realised by the projection on the last coordinate line. It is the line segment $[-1;1]$. This quotient map admits a section $x:q\mapsto(\sqrt{1-q^2};0;\cdots;0;q)$. If we consider $\textnormal{SO}(n-1)$ included in $\textnormal{SO}(n)$ via its action on $\{0\}\times\R^{n-1}$ then the isotropy of $x$ is given by the union of $\textnormal{SO}(n)\times\{-1;1\}$ and $\textnormal{SO}(n-1)\times[-1;1]$. The polar coordinates associated to $x$ is a quotient map for we have the following commutative diagram:
	\begin{equation*}
		\begin{tikzcd}
			\textnormal{SO}(n)\times [-1;1]\ar[rr,"\phi_x"] \ar[dr,"(g;q)\mapsto(g\,\mod\textnormal{SO}(n-1);q)" swap] & & \Sph^n \\
			& \Sph^{n-1}\times [-1;1] \ar[ru,"(v;q)\mapsto(\sqrt{1-q^2} \cdot v;q)" swap]&
		\end{tikzcd}
	\end{equation*}
	and the two factors of $\phi_x$ are straightforwardly quotient maps. We also find that the isotropy of $x$ is cellular, hence $x$ is regular.
\end{ex}

\begin{ex}\label{ex:tor_var1}
	Let $X$ be a real toric variety under the action of $\G{\R}^n$. We can consider its real locus $X(\R)$ and its complex locus $X(\C)$. These spaces are endowed with a continuous action of $(\R^\times)^n$ and $(\C^\times)^n$ respectively. We restrict these actions to their respective maximal compact subgroup $\{\pm1\}^n$ and $(\Sph^1)^n$. Following \cite[\S4.1]{fulton_introduction_1993}, we can also consider the non-negative locus $X_+$ of $X$ contained in $X(\R)$. Let $Y$ be a completion of $X$. The toric orbit filtration of $Y$ endows $Y_+$ with a regular cellular structure with a unique maximal cell of dimension $n$ corresponding to the principal orbit of $Y$. Hence, $Y_+$ is a cellular ball and $X_+$ is obtained from $Y_+$ by removing the open cells corresponding to the toric orbits we had to add to $X$ to make $Y$. Thus $X_+$ is contractible. Moreover, both $X(\R)$ and $X(\C)$ are endowed with an absolute value, i.e. a retraction onto $X_+$. This retraction realises the quotient of $X(\C)$ by $(\Sph^1)^n$ (resp. $X(\R)$ by $\{\pm1\}^n$) so that we have a section of the quotient map, cf. \cite[\S4.1, Proposition]{fulton_introduction_1993}. If the variety is projective, any momentum map will induce a homeomorphism between $X_+$ and the associated momentum polytope. In both cases we have a polar space. The easiest way to check that the polar coordinates are a quotient map is by considering the completion $Y$. Let $K$ be either $\R$ or $\C$ and $G(K)$ denote the maximal compact subgroup of $(K^\times)^n$. The polar coordinates $G(K)\times Y_+\rightarrow Y(K)$ is a closed surjective map for $G(K)$ and $Y_+$ are compact. Hence, it is a quotient map. Now $G(K)\times X_+$ is a locally closed subspace of $G(K)\times Y_+$ that is saturated for the kernel relation of $G(K)\times Y_+\rightarrow Y(K)$ whose restriction to $G(K)\times X_+$ is the polar coordinates of $(G(K);X(K))$. Thus, by \cite[Chapitre~I, \S 5, \textnumero 6, Proposition~10]{bourbaki_topologie_2007-1}, the polar coordinates of $(G(K);X(K))$ is a quotient map. This is mentioned without a proof in \cite[\S4.1, (3) p.80]{fulton_introduction_1993}. Since the isotropy of $(G(K);Y(K))$ is cellular it is a regular polar space. Furthermore, having only removed open cells from $Y_+$ to construct $X_+$, we easily deduce that $(G(K);X(K))$ is a polar space as well. In the real case the isotropy groups of the orbits are always $\Z/2$-modules while in the complex case they are always products of circles. We depicted the isotropy of the complex and real projective planes in Figure~\ref{fig:iso_P2} (in this case the quotient is a triangle).
\end{ex}

\begin{figure}[H]
	\centering
	\begin{subfigure}[t]{.45\textwidth}
		\centering
		\begin{tikzpicture}
			\draw[thick] (0,0) node[below left]{$\{\pm1\}^2$} -- (2,0) node[below right]{$\{\pm1\}^2$} -- (1,1.73) node[above]{$\{\pm1\}^2$} -- cycle;
			\draw (1,0) node[below]{$\langle(1;-1)\rangle$};
			\draw (1.73,.86) node[right]{$\langle(-1;-1)\rangle$};
			\draw (.27,.86) node[left]{$\langle(-1;1)\rangle$};
			\draw (1,.73) node{$1$};
		\end{tikzpicture}
		\caption{The Isotropy for the Real Projective Plane.}
		\label{fig:isoP2}
	\end{subfigure}
	\hfill
	\begin{subfigure}[t]{.45\textwidth}
		\centering
		\begin{tikzpicture}
			\draw[thick] (0,0) node[below left]{$(\Sph^1)^2$} -- (2,0) node[below right]{$(\Sph^1)^2$} -- (1,1.73) node[above]{$\hspace{.2cm}(\Sph^1)^2$} -- cycle;
			\draw (1,0) node[below]{$1\times \Sph^1$};
			\draw (1.73,.86) node[right]{$\Sph^1(1;1)$};
			\draw (.27,.86) node[left]{$\Sph^1\times 1$};
			\draw (1,.73) node{$1$};
		\end{tikzpicture}
		\caption{The Isotropy for the Complex Projective Plane.}
	\end{subfigure}
	\caption{The Isotropy of the Canonical Polar Coordinates of the Real and Complex Projective Planes.}
	\label{fig:iso_P2}
\end{figure}
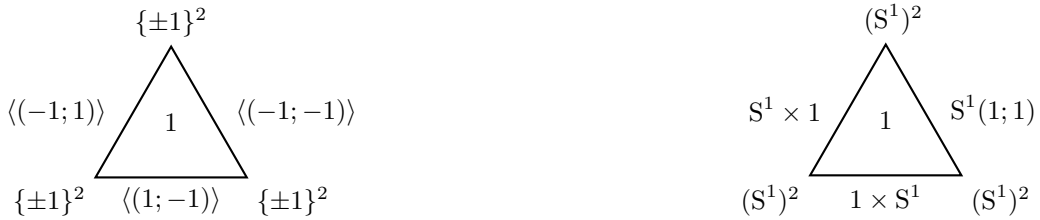

\subsection{Polar Coverings}

Let $G$ be a Lie group acting continuously and faithfully on a connected space $X$.

\paragraph{Action Lifting}
In this paragraph, we describe the classical lifting process for a group action along a covering. When we say that $\rev{X}$ is a covering of $X$, it is implicitly equipped with a covering map $p_{\rev{X}}:\rev{X}\rightarrow X$.

\begin{dfn}\label{dfn:Group_Lift_abstract}
	Let $\rev{X}$ be a connected covering of $X$. We denote by $\rev{G}$ the abstract group made of homeomorphisms $\rev{g}$ of $\rev{X}$ for which there exists a $g$ in $G$ for which $p_{\rev{X}}\circ \rev{g}$ equals $g\circ p_{\rev{X}}$. Such a $g$ is necessarily unique and this endows $\rev{G}$ with a natural group morphism $p_{\rev{G}}$ to $G$.
\end{dfn}

\begin{prop}\label{prop:exact_seq_lift_group} 
	We have the following exact sequence:
	\begin{equation*}
		1 \rightarrow \Aut_X\big(\rev{X}\big) \rightarrow \rev{G} \rightarrow G.
	\end{equation*}
	Moreover, the image of $\rev{G}$ in $G$ is a union of connected components.
\end{prop}

\begin{proof}
	The exactness is a direct consequence of Definition~\ref{dfn:Group_Lift_abstract}. The second assertion follows the fact that the action of a connected group can always be lifted along a covering map, cf. \cite[Theorem~9.3]{bredon_introduction_1972}. 
\end{proof}

The following proposition is a slight variation of \cite[Theorem~9.3]{bredon_introduction_1972}

\begin{prop}\label{prop:topG'}
	There is a unique topology on $\rev{G}$ for which (1) the induced morphism $p_{\rev{G}}$ is continuous and a covering of its image and (2) the induced action $\rev{G}\times \rev{X} \rightarrow \rev{X}$ is continuous.
\end{prop}

From now on, $(G;X)$ is a polar space with quotient space $Q$. 

\begin{dfn}\label{dfn:polcov}
	A covering $\rev{X}$ of $X$ is said to be \emph{polar} if it is connected, Galois, and if $p_{\rev{G}}$ is surjective.
\end{dfn}

We shall remark that G.~Bredon provides a criterion for a covering to be polar, cf. \cite[Theorem~9.3]{bredon_introduction_1972}. Namely, a covering $\rev{X}$ of $X$ is polar if and only if the image by $p_{\rev{X}}$ of the free loop space of $\rev{X}$ is stabilised by the action of $G$ on the free loop space of $X$.  

\begin{dfn}\label{dfn:funct_Br}
	We denote by $\PolC(G;X)$ the full subcategory of coverings of $X$ spanned by its polar coverings. In addition, we denote by $\E(G)$ the category of \emph{discrete extensions} of $G$. Its objects $E$ are Lie groups (also denoted by $E$) endowed with a surjective morphism $p_E:E\rightarrow G$ with discrete kernel. Its morphisms are surjective morphisms of Lie groups $f:E\rightarrow F$ such that $p_F$ equals the composition $f\circ p_E$. Proposition~\ref{prop:topG'} and Definition~\ref{dfn:polcov} yield a functor:
	\begin{equation*}
		\Br:\PolC(G;X)\rightarrow \E(G).
	\end{equation*} 
	We will say that a discrete extension of $G$ is \emph{effective relatively to} $X$ (or simply \emph{effective} when $X$ is clear from the context) if it belongs to $\E(G;X)$, the image category\footnote{N.b. a morphism of extensions belongs to this category if it is induced by a morphism of covering spaces.} of $\Br$. Given a polar covering $\rev{X}$ of $X$, we will often denote by $\rev{G}$ the associated discrete extension of $G$. 
\end{dfn}

\begin{prop}\label{prop:polcov_polspace}
	If $\rev{X}$ is a polar covering of $X$ with associated extension $\rev{G}$, then $(\rev{G};\rev{X})$ is a polar space whose quotient $\rev{X}/\rev{G}$ is homeomorphic to $X/G$.
\end{prop}

\begin{proof}
	Let $Q$ be the quotient $X/G$ and $x:Q\rightarrow X$ be a section of the quotient map and $u:H\rightarrow G_Q$ be the isotropy of $x$. Proposition~\ref{prop:topG'} ensures the action of $\rev{G}$ on $\rev{X}$ is continuous. Moreover, since $\rev{G}$ acts on $\rev{X}$ as a subgroup of homeomorphisms, the action is faithful. Moreover, $\rev{X}$ is necessarily locally path-connected and semi-locally simply connected for it is a covering of $X$. It is connected by assumption. Now let us show that the quotient of $\rev{X}$ by $\rev{G}$ is $Q$. Since $p_{\rev{X}}$ is a Galois covering, it realises the quotient of $\rev{X}$ by the action of $\Aut_X(\rev{X})$. Hence, the quotient of $\rev{X}$ by $\rev{G}$ is homeomorphic to the quotient of $X$ by $G$, i.e. to $Q$. Furthermore, this quotient is realised by the composition of $p_{\rev{X}}$ with the quotient projection $X\rightarrow Q$. Since $Q$ is assumed to be simply connected, we can find a lift $\rev{x}$ of $x$ along $p_{\rev{X}}$. This map is a section of the quotient projection of $\rev{X}$. We have the following commutative square of polar coordinates:
	\begin{equation}\label{eq:com_pol_coord}
		\begin{tikzcd}
			\rev{G}_Q \ar[r,"\phi_{\rev{x}}"] \ar[d,"(p_{\rev{G}})_Q" swap] & \rev{X} \ar[d,"p_{\rev{X}}"]\\
			G_Q \ar[r,"\phi_x" swap] & X
		\end{tikzcd}
	\end{equation}
	The map $\phi_{\rev{x}}$ is surjective for $\rev{X}$ is polar. Indeed, let $P$ be a point of $\rev{X}$ that lies above a point of polar coordinates $(g;q)$ in $X$. Since $p_{\rev{X}}$ is Galois and $\rev{G}$ contains $\Aut_X(\rev{X})$, we can find $\rev{g}$ in $\rev{G}$ above $g$ such that $P$ equals $\phi_{\rev{x}}(\rev{g};q)$. Finally, let us show that $\phi_{\rev{x}}$ is a quotient map. This condition can be checked locally on $\rev{X}$. Let $U$ denote an open set of $X$ over which $p_{\rev{X}}$ is trivial. The restriction of (\ref{eq:com_pol_coord}) to $U$ has the following shape:
	\begin{equation}\label{eq:com_pol_coord2}
		\begin{tikzcd}
			\phi_{\rev{x}}^{-1}\big(\Aut_X\big(\rev{X}\big)\times U\big) \ar[r,"\phi_{\rev{x}}"] \ar[d,"p" swap] & \Aut_X\big(\rev{X}\big)\times U \ar[d,"\pr_U"]\\
			\phi_x^{-1}(U) \ar[r,"\phi_x" swap] & U
		\end{tikzcd}
	\end{equation}
	The open set $\phi_{\rev{x}}^{-1}(\Aut_X(\rev{X})\times U)$ of $\rev{G}_Q$ is stable under the action of $\Aut_X(\rev{X})$ and the map $p$ realises the quotient by said action. Let $C$ be a connected component of $\phi^{-1}_x(U)$. Since $\phi_{\rev{x}}$ is onto, continuous, and $\Aut_X(\rev{X})$-equivariant, the action of $\Aut_X(\rev{X})$ on the connected components of $p^{-1}(C)$ must be free. Hence, $p$ is trivial above every connected components of $\phi^{-1}_x(U)$. Therefore, it must be trivial above the whole of $\phi^{-1}_x(U)$. Thus, in an appropriate choice of trivialisation, (\ref{eq:com_pol_coord2}) can be rewritten as follows: 
	\begin{equation}\label{eq:com_pol_coord3}
		\begin{tikzcd}
			\Aut_X\big(\rev{X}\big)\times \phi_{x}^{-1}(U) \ar[rr,"\id_\Aut\times\phi_x"] \ar[d,"\pr_{\phi_x^{-1}(U)}" swap] && \Aut_X\big(\rev{X}\big)\times U \ar[d,"\pr_U"]\\
			\phi_x^{-1}(U) \ar[rr,"\phi_x" swap] && U
		\end{tikzcd}
	\end{equation}
	Hence, $\phi_{\rev{x}}$ is a quotient map for it is locally given by $\phi_x$.
\end{proof} 

\begin{prop}\label{prop:isotropy_and_regularity}
	Let $\rev{X}$ be a polar covering of $X$ with associated extension $\rev{G}$ of $G$. Let $x:Q\rightarrow X$ be a section of the quotient map, and $\rev{x}$ be a lift of $x$ along $p_{\rev{X}}$. The map $(p_{\rev{G}})_Q$ yields an isomorphism between the isotropy of $\rev{x}$ and the isotropy of $x$. Moreover, if $x$ is regular so is $\rev{x}$.
\end{prop}

\begin{proof}
	Let $H$ be the isotropy of $x$ and $\rev{H}$ be the isotropy of $\rev{x}$. By the commutativity of (\ref{eq:com_pol_coord}), $(p_{\rev{G}})_Q$ sends $\rev{H}$ into $H$. Let $q$ be a point of $Q$. The inverse image by $p_{\rev{X}}$ of the orbit of $x(q)$ is the orbit of $\rev{x}(q)$. Thus, by Proposition~\ref{prop:orbit_quotient}, the restriction of $p_{\rev{X}}$ between these orbits is induced by $p_{\rev{G}}$ as follows:
	\begin{equation}
		\begin{array}{ccc}
			\rev{G}/\rev{H}_q & \longrightarrow & \displaystyle G/H_q \\[5pt] 
			\,[g]\; & \longmapsto & [p_{\rev{G}}(g)].
		\end{array}
	\end{equation}
	Hence, the inverse image of $x(q)$ by $p_{\rev{X}}$ is homeomorphic to $p_{\rev{G}}^{-1}(H_q)/\rev{H}_q$. However, we already know that the map $f:\Aut_X(\rev{X})\rightarrow p_{\rev{G}}^{-1}(H_q)/\rev{H}_q$ induced by the inclusion of $\Aut_X(\rev{X})$ in $p_{\rev{G}}^{-1}(H_q)$ is a bijection. Moreover, we have the following commutative diagram with exact rows: 
	\begin{equation}\label{eq:Isotropy1}
		\begin{tikzcd}
			1 \ar[r] & \Aut_X\big(\rev{X}\big) \ar[r] & (p_{\rev{G}})^{-1}(H_q) \ar[r,"p_{\rev{G}}"] & H_q \ar[r] & 1 \\
			1 \ar[r] & 1 \ar[u] \ar[r] & \rev{H}_q \ar[u,"\subset"] \ar[r,equal] & \rev{H}_q \ar[u,"p_{\rev{G}}" swap] \ar[r] & 1
		\end{tikzcd}
	\end{equation}
And the Snake Lemma yields the following exact sequence of pointed sets:
	\begin{equation}
		1 \longrightarrow \ker(p_{\rev{G}}|_{\rev{H}_q}) \overset{g}{\longrightarrow} \Aut_X\big(\rev{X}\big) \overset{f}{\longrightarrow} (p_{\rev{G}})^{-1}(H_q)/\rev{H}_q \overset{h}{\longrightarrow} H_q/p_{\rev{G}}(\rev{H}_q) \longrightarrow 1.
	\end{equation}
	The map $g$ is an inclusion, thus a group morphism. Hence, the bijectivity of $f$ implies that $p_{\rev{G}}|_{\rev{H}_q}$ is injective. Furthermore, the map $h$ is surjective, and its kernel -- here the inverse image of the distinguished point -- is the image of $f$. Therefore, $h$ is constant for $f$ is onto, and the restriction of $p_{\rev{G}}$ to $\rev{H}_q$ is surjective. Thus, it is a bijection for every $q$. Since the restriction of $(p_{\rev{G}})_Q$ to $\rev{H}$ is a morphism of $Q$-groups that restricts to a bijection over every fibre of $\rev{H}$, it is a continuous bijection between $\rev{H}$ and $H$. The map $(p_{\rev{G}})_Q$ is open for $p_{\rev{G}}$ is a covering. Therefore, so is the map it induces between $\rev{H}$ and $H$. Hence, $(p_{\rev{G}})_Q$ yields a homeomorphism between $\rev{H}$ and $H$. 
	
	\vspace{5pt}
	
	Let us now assume that $x$ is regular and let us be given for $q$ in $Q$ an open neighbourhood $U$ of $x(q)$ and a retraction by deformations $h:[0;1]\times\phi^{-1}_x(U)\rightarrow \phi^{-1}_x(U)$ satisfying the requirements of Definition~\ref{dfn:regular}. The inverse $\rev{u}$ of the isomorphism $\rev{H}\rightarrow H$ induced by $(p_{\rev{G}})_Q$ is a lift of $u$ (the inclusion of $H$ in $G_Q$) along $(p_{\rev{G}})_Q$. Since coverings have the homotopy lifting property, cf. \cite[Proposition~1.30]{hatcher_algebraic_2000}, we can find $\
	f:[0;1]\times \phi_x^{-1}(U)\rightarrow (p_{\rev{G}})_Q^{-1}\phi_x^{-1}(U)$ that completes the following commute diagram:
	 \begin{equation*}
	 	\begin{tikzcd}
			\phi_x^{-1}(U) \ar[r,"\rev{u}\circ h_1"] \ar[d,"(1;\id)" swap] & (p_{\rev{G}})_Q^{-1}\phi_x^{-1}(U) \ar[d,"(p_{\rev{G}})_Q"] \\
			{[0;1]}\times \phi_x^{-1}(U) \ar[r,"h" swap] \ar[ru, dashed, "f"] & \phi_x^{-1}(U) 
		\end{tikzcd}
	 \end{equation*}
	 Then, the map $v:\phi_x^{-1}(U)\rightarrow p_Q^{-1}\phi_x^{-1}(U)$ given by $f_0$ is a section of $(p_{\rev{G}})_Q$ above $\phi_x^{-1}(U)$. This is an extension of $\rev{u}|_{H\cap \phi_x^{-1}(U)}$ to $\phi_x^{-1}(U)$. First, it is an extension of $\rev{u}|_{H_q\times\{q\}}$ to $\phi_x^{-1}(U)$ by construction. Since $\rev{u}|_{H\cap \phi_x^{-1}(U)}$ is a lift map and $H\cap \phi_x^{-1}(U)$ retracts onto $H_q\times\{q\}$, $v|_{H\cap \phi_x^{-1}(U)}$ and $\rev{u}|_{H\cap \phi_x^{-1}(U)}$ must be equal. In addition, we have the following commutative diagram:
	 \begin{equation*}
	 	\begin{tikzcd}
			\phi_x^{-1}(U) \ar[rr,"v"] \ar[dr] && (p_{\rev{G}})_Q^{-1}\phi_x^{-1}(U) \ar[dl] \\
			& {||U||} &
		\end{tikzcd}
	 \end{equation*}
	 so that $v$ is a map of $||U||$-spaces. Moreover, both the source and target of $v$ are endowed with an action of $H\cap \phi_x^{-1}(U)$ (through $\rev{u}$ for the target). We claim that $v$ is equivariant. We have the following commutative diagram:
	 \begin{equation*}
	 	\begin{tikzcd}
			& (p_{\tilde{G}})_Q^{-1}\phi_x^{-1}(U) \ar[d,"(p_{\rev{G}})_Q"] & \left\{\begin{array}{l} a(g;y) :=v(h\cdot y) \\ b(g;y) := u'(g)\cdot v(y) \\ c(g;y):=g\cdot y \end{array}\right. \\
			(H\cap\phi_x^{-1}(U))\underset{||U||}{\times}\phi_x^{-1}(U) \ar[ru,"{a,b}"]  \ar[r,"c" swap] & \phi_x^{-1}(U) & 
		\end{tikzcd}
	 \end{equation*}
	Thus $a$ and $b$ are two lifts of $c$ along $(p_{\rev{G}})_Q$. They coincide on $(H_q\times\{q\})\times_{||U||}( H_q\times \{q\})$. But, by assumptions, the formula $(g;y) \mapsto (h_t(g);h_t(y))$ defines a retraction of $(H\cap\phi_x^{-1}(U))\times_{||U||}\phi_x^{-1}(U)$ onto the fibre product $(H_q\times\{q\})\times_{||U||}( H_q\times \{q\})$. Hence, $a$ equals $b$, and $v$ is equivariant. Therefore, the following commutative diagram:
	 \begin{equation*}
	 	\begin{tikzcd}
			\Aut_X(\rev{X})\times \phi_x^{-1}(U) \ar[rr,"(k;y)\mapsto k\cdot v(y)"] \ar[dr,"\pr" swap ] & & (p_{\rev{G}})_Q^{-1}\phi_x^{-1}(U) \ar[dl,"(p_{\rev{G}})_Q"] \\
			& \phi_x^{-1}(U) & 
		\end{tikzcd}
	 \end{equation*}
	 defines a trivialisation of $(p_{\rev{G}})_Q:(p_{\rev{G}})_Q^{-1}\phi_x^{-1}(U)\rightarrow \phi_x^{-1}(U)$. Moreover, the equivariance of $v$ ensures $v(\phi^{-1}_x(U))$ is saturated under the action of $\rev{H}$. Thus, the image $V$ of $v(\phi^{-1}_x(U))$ by $\phi_{\rev{x}}$ is an open neighbourhood of $\rev{x}(q)$ in $\rev{X}$ whose inverse image $\phi^{-1}_{\rev{x}}(V)$ is precisely $v(\phi^{-1}_x(U))$ by saturation. By construction, the map $v\circ h_t\circ (p_{\rev{G}})_Q$ is a retraction by deformations of $\phi^{-1}_{\rev{x}}(V)$ onto $\rev{H}_q\times\{q\}$ that satisfies the requirements of Definition~\ref{dfn:regular}. Thus, $\rev{x}$ is regular. 
\end{proof}

Then we find the obvious proposition.

\begin{prop}\label{prop:univ=polar}
	The universal covering of $X$ is polar.
\end{prop}

The aim of the following of the article is to apprehend the image of the functor $\Br$. If we can characterise the extension associated to the universal covering of $X$, then Proposition~\ref{prop:exact_seq_lift_group} provides a way of computing the fundamental group of $X$.  

\subsection{Equivalence of Categories}

Despite not being an equivalence of categories, the functor $\Br$ is actually not far from it. Here we describe a procedure that turns it into an equivalence by keeping track of more information. 

\begin{dfn}\label{dfn:PolCovGXx}
	Let $(G;X)$ be a polar space, $Q$ be the quotient, and $x:Q\rightarrow X$ be a section of the projection. We denote by $\textnormal{Pol.Cov.}(G;X;x)$ the category whose objects are polar coverings $\rev{X}$ of $X$ endowed with a lift $\rev{x}:Q\rightarrow \rev{X}$ of $x$. A morphism in this category is a continuous map $f:\rev{X}_1\rightarrow \rev{X}_2$ such that the following diagram commutes:
	\begin{equation*}
		\begin{tikzcd}
			\rev{X}_1 \ar[rr,"f"] \ar[rd,"p_{\rev{X}_1}" swap] & & \rev{X}_2\ar[ld,"p_{\rev{X}_2}"] \\
			& X & \\
			& Q \ar[u,"x"] \ar[ruu,bend right=50, "\;\rev{x}_2" swap] \ar[luu,bend left=50,"\rev{x}_1"] &
		\end{tikzcd}
	\end{equation*} 
\end{dfn}

\begin{dfn}\label{dfn:EGXx}
	Let $(G;X)$ be a polar space, $Q$ be the quotient, $x:Q\rightarrow X$ be a section of the projection, and $u:H\rightarrow G_Q$ denote the inclusion of the isotropy of $x$. We denote by $\textnormal{E}(G;X;x)$ the category whose objects are discrete extensions $\rev{G}$ of $G$ endowed with a ubiquitous lift $\rev{u}:H\rightarrow \rev{G}_Q$ of $u$. A morphism in this category is a morphism of Lie groups $f:\rev{G}_1\rightarrow \rev{G}_2$ such that the following diagram commutes:
	\begin{equation*}
		\begin{tikzcd}
			(\rev{G}_1)_Q \ar[rr,"f_Q"] \ar[rd,"(p_{\rev{G}_1})_Q" swap] & & (\rev{G}_2)_Q\ar[ld,"(p_{\rev{G}_2})_Q"] \\
			& G_Q & \\
			& H \ar[u,"u"] \ar[ruu,bend right=50,"\rev{u}_2", swap] \ar[luu,bend left=50,"\rev{u}_1"] &
		\end{tikzcd}
	\end{equation*}
\end{dfn}

\begin{thm}[Equivalence of Categories] \label{thm:equiv_pol_cov}
	Let $(G;X)$ be a regular polar space with quotient space $Q$, and let ${u:H\rightarrow G_Q}$ be the isotropy of a regular section $x:Q\rightarrow X$ of the quotient map. There is an equivalence of categories:
	\begin{equation*}
		\Br^*:\PolC(G;X;x) \longleftrightarrow \textnormal{E}(G;X;x):\textnormal{C}.
	\end{equation*} 
\end{thm}

\begin{proof}
	The Functor $\Br^*$: Let $\rev{X}$ be a polar covering of $(G;X)$ equipped with a lift $\rev{x}$ of $x$ along $p_{\rev{X}}$. We denote by $\rev{G}$, the discrete extension obtained from $\rev{X}$ via $\Br$. If $v:\rev{H}\rightarrow \rev{G}_Q$ denotes the inclusion of the isotropy of $\rev{x}$, we know from Proposition~\ref{prop:isotropy_and_regularity} that the restriction of $p_Q:\rev{G}_Q\rightarrow G_Q$ to $\rev{H}$ is an isomorphism. Thus, we can endow $H$ with a lift $\rev{u}$ of $u$ along $p_Q$. It is defined by: \begin{equation*}
		u'\coloneqq v\circ p_Q |_{\rev{H}}^{-1}.
	\end{equation*}
	Since, $p_Q |_{\rev{H}}^{-1}$ is onto and $v$ is ubiquitous ($\rev{X}$ is connected by assumption), $\rev{u}$ is also ubiquitous. We have our first functor.
	
	\vspace{5pt}
	
	 The functor $\textnormal{C}$: Conversely, let us consider a discrete extension $\rev{G}$ of $G$ and $\rev{u}$ a lift of $u$ along $p_{\rev{Q}}$. We form the space $\rev{X}$, the quotient of $\rev{G}_Q$ by the image of $\rev{u}$. The projection $p_{\rev{G}}$ and the polar coordinates of $X$ induce a projection $p_{\rev{X}}:\rev{X}\rightarrow X$. Moreover, the composition of $q\in Q \mapsto (1;q)\in \rev{G}_Q$ with the quotient map induces a lift $\rev{x}$ of $x$ along $p_{\rev{X}}$. Since $\rev{u}$ is ubiquitous, $\rev{X}$ is connected. Furthermore, $p_{\rev{X}}$ realises the quotient of $\rev{X}$ by the free and continuous action of the discrete subgroup $\ker(p_{\rev{G}})$. To show that our functor is well defined we need to prove that  $p_{\rev{X}}$ is a covering map, i.e. that $\rev{X}$ is Hausdorff and that the action of $\ker(p_{\rev{G}})$ is properly discontinuous. Since $X$ is Hausdorff and $p_{\rev{X}}$ is continuous, two points of $\rev{X}$ with distinct images under $p_{\rev{X}}$ have necessarily disjoint open neighbourhoods in $\rev{X}$. Now we will show that, for every $x_0$ in $X$, we can build an open set $V$ of $\rev{X}$ that contains only one inverse image of $x_0$ by $p_{\rev{X}}$ and which is disjoint from $k\cdot V$ for every non-trivial $k$ in $\ker(p_{\rev{G}})$. This will finish to prove that $X$ is Hausdorff, show that $\ker(p_{\rev{G}})$ acts properly on $\rev{X}$, and thus establish that $p_{\rev{X}}$ is a covering. 
	 
	 \vspace{5pt}
	 
	 Let $q$ be a point of $Q$. We consider a neighbourhood $U$ of $x(q)$ and a homotopy $h:[0;1]\times \phi_x^{-1}(U) \rightarrow \phi_x^{-1}(U)$ satisfying the requirements of Definition~\ref{dfn:regular}. Using the same reasoning as in the second part of the proof of Proposition~\ref{prop:isotropy_and_regularity}, we can find an $H$-equivariant\footnote{The action on the target is through $\rev{u}$.} extension $v$ of $\rev{u}|_{H\cap\phi^{-1}_x(U)}$ to $\phi^{-1}_x(U)$. This implies that the following commutative diagram:
	 \begin{equation*}
	 	\begin{tikzcd}
			\ker(p_{\rev{G}})\times \phi_x^{-1}(U) \ar[rr,"(k;y)\mapsto k\cdot v(y)"] \ar[dr,"\pr_{\phi_x^{-1}(U)}" swap ] & & (p_{\rev{G}})_Q^{-1}\phi_x^{-1}(U) \ar[dl,"(p_{\rev{G}})_Q"] \\
			& \phi_x^{-1}(U) & 
		\end{tikzcd}
	 \end{equation*}
	 is a trivialisation of $(p_{\rev{G}})_Q:(p_{\rev{G}})_Q^{-1}\phi_x^{-1}(U)\rightarrow \phi_x^{-1}(U)$. Hence, the $k\cdot v(\phi_x^{-1}(U))$'s are disjoint open sets of $\rev{G}_Q$. Moreover, the equivariance of $v$ ensures they are saturated under the action of $H$ via $\rev{u}$. Thus, the image $V$ of $v(\phi^{-1}_x(U))$ by $\phi_{\rev{x}}$ is an open set of $\rev{X}$ that is disjoint from $k\cdot V$ for every non-trivial $k$ in $\ker(p_{\rev{G}})$. It can only contain one inverse image of $x(q)$. It is the point $\rev{x}(q)$. For a more general point $x_0$, necessarily of the form $p_{\rev{G}}(\rev{g})\cdot x(q)$, we can consider the open set $\rev{g}\cdot V$. Thus, $p_{\rev{X}}:\rev{X}\rightarrow X$ is a covering. It is connected for $\rev{u}$ is assumed to be ubiquitous, and Galois by construction ($X$ is the quotient of $\rev{X}$ by $\ker(p_{\rev{G}})$). Hence, it is a polar covering.
	 
	 \vspace{5pt}
	 
	 We have now to establish the equivalence. One the one hand, the natural isomorphism between $\textnormal{C}\Br^*$ and the identity of $\PolC(G;X;x)$ is provided by the polar coordinates. Indeed, given $\rev{X}$ and $\rev{x}$, $\Br^*$ extracts the acting group $\rev{G}$ lifting the action and the lift of the isotropy $\rev{H}$, then $\textnormal{C}$ realises the quotient of $\rev{G}_Q$ by $\rev{H}$. This space is isomorphic to $\rev{X}$ for $(\rev{G};\rev{X})$ is a polar space, cf. Proposition~\ref{prop:polcov_polspace}. Since $\rev{u}$ is the lift of $u$ associated to $\rev{x}$, it is clear, using polar coordinates, that the lift of $x$ induced by $\textnormal{C}$ is $\rev{x}$. On the other hand, the isomorphism between $\Br^*\textnormal{C}$ and the identity of $\E(G;X;x)$ comes from Proposition~\ref{prop:exact_seq_lift_group}. Together with Proposition~\ref{prop:topG'}, they imply the following commutative diagram of Lie groups with exact rows:
	 \begin{equation*}
	 	\begin{tikzcd}
			1 \ar[r] & \Aut_X(\rev{X}) \ar[r] & \Br^*\textnormal{C}(\rev{G};\rev{u}) \ar[r] & G \ar[r] & 1\\
			1 \ar[r] & \ker(p_{\rev{G}}) \ar[u,"a"] \ar[r] & \rev{G} \ar[u,"b"] \ar[r] & G \ar[u,equal] \ar[r] & 1\\
		\end{tikzcd}
	 \end{equation*}
	 We saw in the preceding paragraph, that $a$ is an isomorphism. Hence, $b$ is also necessarily an isomorphism. Now, by construction, the isotropy of the lift $\rev{x}$ of $x$ along $\textnormal{C}(\rev{G};\rev{u})\rightarrow X$ is precisely $\rev{u}(H)$ where $H$ denotes the isotropy of $x$. The inverse of $(p_{\rev{G}})_Q$ restricted to $\rev{u}(H)$ is precisely $\rev{u}$ for $(p_{\rev{G}})_Q\circ\rev{u}$ is $u$, the inclusion of $H$ in $G$. 
\end{proof}

We deduce the following immediate corollary:

\begin{cor}\label{cor:effective_ubi}
	Let $(G;X)$ be a regular polar space with quotient space $Q$ and $H$ be the isotropy of a regular section $x$ of $X\rightarrow Q$. A discrete extension $E$ of $G$ is effective relatively to $X$ if and only if the inclusion of $H$ in $G_Q$ admits a ubiquitous lift along $p_E$.
\end{cor}

\section{Discrete Extensions of Lie Groups}

Let $G$ denote a Lie group. Since Lie groups are naturally pointed at the identity, we will denote by $\pi_1(G)$ the fundamental group of $G$ pointed at $1$. If $E$ is a discrete extension of $G$, then it is a covering. As such, the homotopy long exact sequence, cf. \cite[Theorem 4.41 and Proposition 4.48]{hatcher_algebraic_2000}, associated to the fibre bundle $p_E$ yields the following natural exact sequences of groups:
\begin{equation}\label{seq:fib_lES_disc_ext}
	0 \rightarrow \pi_1(E) \overset{(p_E)_*}{\longrightarrow} \pi_1(G) \overset{\partial}{\longrightarrow} \ker(p_E) \longrightarrow \pi_0(E) \overset{(p_E)_*}{\longrightarrow} \pi_0(G) \rightarrow 1.
\end{equation}
Since we write Lie groups in multiplicative notations unless they are specifically commutative and fundamental groups of Lie groups are Abelian, we adopt the following exponential notation:

\begin{dfn}[Exponential]\label{def:exp}
	Let $E$ be a discrete extension of $G$. If $k$ is an element of $\pi_1(G)$, we denote by $e^k$ its image by the natural morphism $\partial :\pi_1(G)\rightarrow \ker(p_E)$ of the exact sequence (\ref{seq:fib_lES_disc_ext}). If $k$ belongs to the cokernel of $(p_E)_*:\pi_1(E)\rightarrow \pi_1(G)$, we will also denote its image by $e^k$. 
\end{dfn}

The sequence (\ref{seq:fib_lES_disc_ext}) contains two useful “\,parameters\,” to navigate the category $\E(G)$.

\begin{dfn}[Magnitude]
	Let $E$ be a discrete extension of $G$. Its \emph{magnitude} designates the discrete group $\pi_0(E)$ together with the surjective morphism $(p_E)_*:\pi_0(E)\rightarrow \pi_0(G)$. It is an object of $\E(\pi_0(G))$, category that we denote by $\M(G)$. 
\end{dfn}

\begin{dfn}[Phase]
	Let $E$ be a discrete extension of $G$. Its \emph{phase} is the image subgroup of the injective morphism $(p_E)_*:\pi_1(E) \hookrightarrow \pi_1(G)$. This subgroup is stable under the action of $\pi_0(G)$ by conjugation. We denote the set of $\pi_0(G)$-subgroups of $\pi_1(G)$ by $\T(G)$. We see it has a small category where a morphism $\theta_1\rightarrow \theta_2$ corresponds to an inclusion $\theta_1\subset \theta_2$.
\end{dfn}

\subsection{Structure of the Category}\label{subsection:Structure of the Category}

Most of the material of this subsection is already contained in \cite{taylor_covering_1954}. We set notations, recall some useful features, and establish simple functorial properties.

\begin{dfn}
	Let $\mu$ be a magnitude and $\theta$ be a phase. We denote by $[\E(G)]_{\mu,\theta}$ the set\footnote{The definition implies that the category is essentially small. Indeed having fixed the magnitude and the phase, the topology of $E$ is fixed and we are reduced to classify certain Lie group structures on a given manifold.} of isomorphism classes of the following groupoid:\begin{enumerate}
	\item An object $E$ is a Lie group, also denoted by $E$, endowed with two morphisms $p_E:E\rightarrow G$ and $r_E:E\rightarrow \mu$ such that: $p_E$ is a discrete extension, $\theta$ is the image of $(p_E)_*:\pi_1(E)\rightarrow \pi_1(G)$, and $(r_E)_*:\pi_0(E)\rightarrow \mu$ is an isomorphism that makes the following diagram commute:
	\begin{equation*}
		\begin{tikzcd}
			\pi_0(E) \ar[rr,"(r_E)_*"] \ar[rd,"(p_E)_*" swap] & & \mu \ar[ld,"p_\mu"] \\
			& \pi_0(G)  
		\end{tikzcd} 
	\end{equation*}
	An object of this groupoid is called \emph{a discrete extension of $G$ of magnitude $\mu$ and phase $\theta$}.
	\item A morphism $f$ from $E$ to $F$ is a morphism of Lie groups $f:E\rightarrow F$ such that $p_{E}$ is the composition $p_F\circ f$ and $r_E$ is the composition $r_F\circ f$.
	\end{enumerate}
\end{dfn}

\begin{prop}\label{prop:functor_iso_ext}
	The process $(\mu;\theta)\mapsto [E(G)]_{\mu,\theta}$ can be seen as a contravariant functor $[E(G)]$ from the category $\M(G)\times \T(G)^\textnormal{op}$ to the category of sets.
\end{prop}

\begin{proof}
	A morphism $h$ from $(\mu';\theta')$ to $(\mu;\theta)$ in $\M(G)\times \T(G)^\textnormal{op}$ is a pair of morphisms:
	\begin{equation*}
		\left\{\begin{array}{l@{\textnormal{ in }}c}
			f:\mu' \rightarrow \mu & \M(G) \\
			g:\theta' \leftarrow \theta & \T(G),
		\end{array}\right.
	\end{equation*}
	Let $[E]$ be an equivalence class of extensions in $[\E(G)]_{\mu;\theta}$. We can consider the following Lie group:
	\begin{equation*}
		E':=\bigslant{E\underset{\mu}{\times}\mu'}{\big(\theta'/\theta\big)},
	\end{equation*}
	where $\theta'/\theta$ is embedded in $\ker(p_E)$ via the exponential, cf. Definition~\ref{def:exp}. The order in which we proceed does not matter, up to isomorphism, for $\theta'/\theta$ lies in the connected component of the identity and the fibre product operation only affects the group of connected components. The new morphism $r_{E'}$ is simply obtained as $\pr_{\mu'}$. Meanwhile, the new projection $p_{E'}$ is the composition $p_E\circ\pr_E$. Note that in both cases $\theta'/\theta$ lies in the kernel of these morphisms, \emph{a priori} defined on the fibre product, so that they descend to the quotient. By the universal properties of the quotient and the fibre product, the class $[E']$ only depends on the class $[E]$ and not the specific representative. Moreover, these properties also imply that, with these images of morphisms, $(\mu;\theta)\mapsto [\E(G)]_{\mu,\theta}$ is a functor.  
\end{proof}

\begin{dfn}
	For all integers $k\geq 0$, we have a contravariant functor:
	\begin{equation*}
		H^k:\M(G)\times\T(G)^\textnormal{op} \rightarrow \textnormal{Mod}_\Z, 
	\end{equation*}
	that associates $H^k(\mu;\pi_1(G)/\theta)$ to $(\mu;\theta)$. The Abelian group $\pi_1(G)/\theta$ is, by definition, a $\pi_0(G)$-module. We pull it back as a $\mu$-module using $p_\mu$.  
\end{dfn}

\begin{thm}[(6.5) in \cite{taylor_covering_1954}, or Theorem 6.3 in \cite{brown_covering_1994}]\label{thm:taylor}
	For every magnitude and phase $(\mu;\theta)$, there is an obstruction class $\chi(\mu;\theta)$ in $H^3(\mu;\pi_1(G)/\theta)$ that vanishes if and only if $[\E(G)]_{\mu,\theta}$ is non-empty. Moreover, if $f:(\mu_1;\theta_1)\rightarrow (\mu_2;\theta_2)$ is a morphism, then $\chi(\mu_2;\theta_2)$ is the pullback $f^*\chi(\mu_1;\theta_1)$. 
\end{thm}

We wish to say a few words on what this obstruction actually represents for it has a nice and simple interpretation. First, we lay down its construction. Let $G$, $\mu$ and $\theta$ be fixed. Moreover, we consider $p^\circ:G^\circ_\theta\rightarrow G^\circ$ the connected discrete extension associated to $\theta$. Recall it is a central extension whose kernel is the image of the exponential embedding $\pi_1(G)/\theta\rightarrow G^\circ_\theta$, cf. \cite[Proposition~1.11]{bredon_introduction_1972}. Let us fix a section $s:\pi_0(G)\rightarrow G$ of the projection $G\rightarrow \pi_0(G)$. It does not have to be a group morphism but we assume that $s(1)=1$. We write $t(a)$ instead of $s(p_\mu(a))$ for every $a$ in $\mu$. We can form the following covering space of $G$:
\begin{equation*}
	\begin{array}{rcl}
		p:G^\circ_\theta\times \mu & \longrightarrow & G \\
		(g;a) & \longmapsto & p^\circ(g)t(a).
	\end{array}
\end{equation*}
Every discrete extension $E$ of $G$ of magnitude $\mu$ and phase $\theta$ is isomorphic, as a covering space, to $p:G^\circ_\theta\times\mu\rightarrow G$. One can even find (using $r_E$) a covering space isomorphism $\varphi:E\rightarrow G^\circ_\theta\times\mu$ such that $\varphi_*:\pi_0(E)\rightarrow \mu$ and $\varphi|_{E^\circ}:E^\circ\rightarrow G^\circ_\theta\times\{1\}$ are group isomorphisms. Thus, the discrete extension $E$ endows $G^\circ_\theta\times\mu$ with a group structure $m$. Such a map has to satisfy the following properties:
\begin{equation}\label{eq:prop_of_m}
	\underset{\textnormal{the diagram is commutative}}{\underbrace{\begin{tikzcd}[ampersand  replacement = \&]
		(G^\circ_\theta\times\mu)^2 \ar[d,"p" left] \ar[rr,"m"] \& \& G^\circ_\theta\times\mu \ar[d,"p"] \\
		G^2 \ar[rr,"\textnormal{group law of }G"below] \& \& G
	\end{tikzcd}}}
	\hspace{1cm}
	\begin{array}{c}
		\forall g,h\in G_\theta^\circ,\,\forall a\in \mu\\[5pt]
		m\big((g;1);(h;a)\big)=(gh;a)
	\end{array}
	\hspace{1cm}
	\underset{\forall a,b\in \mu}{\underbrace{\begin{tikzcd}
		(G^\circ_\theta\times\{a\})\times(G^\circ_\theta\times\{b\}) \ar[d,"m" left] \\
		(G^\circ_\theta\times\{ab\})
	\end{tikzcd}}}
\end{equation}
A few tricks show that if $m$ satisfies (\ref{eq:prop_of_m}) and is associative, then $p:(G^\circ_\theta\times\mu;m)\rightarrow G$ is a discrete extension of magnitude $\mu$ and phase $\theta$. One can show that if $m$ and $m'$ both satisfy (\ref{eq:prop_of_m}) then there exists a unique $\beta$ in $C^2(\mu;\pi_1(G)/\theta)$ such that:
\begin{equation}\label{eq:different_m}
	m'\big((g;a);(h;b)\big)=e^{\beta(a;b)}m\big((g;a);(h;b)\big),
\end{equation}
for all $g,h$ in $G^\circ_\theta$, and all $a,b$ in $\mu$. We note that it does not matter if we multiply by $e^{\beta(a;b)}$ using $m$ or $m'$ because of the second property of (\ref{eq:prop_of_m}). Reciprocally, every $m'$ defined by the formula (\ref{eq:different_m}) will satisfy (\ref{eq:prop_of_m}). Now, given $m$ satisfying (\ref{eq:prop_of_m}), one may try to measure its faillure to be associative. Since $m$ lifts the group law of $G$, there exists $\gamma$ in $C^3(\mu;\pi_1(G)/\theta)$ such that:
\begin{equation}\label{eq:different_m2}
	m\Big( m\big((g;a);(h;b)\big);(k;c)\Big)=e^{\gamma(a;b;c)}m\Big((g;a); m\big((h;b);(k;c)\big)\Big),
\end{equation}
for all $g,h,k$ in $G^\circ_\theta$, and all $a,b,c$ in $\mu$. If we consider a different $m'$ satisfying (\ref{eq:prop_of_m}) and linked to $m$ via (\ref{eq:different_m}), we end up with the cohomologous cochain $\gamma+\df\beta$. In addition, we can always build an $m$ satisfying (\ref{eq:prop_of_m}) as follows:\begin{enumerate}
\item For all $a,b$ in $\mu$, we choose an element $f(a;b)$ in $G^\circ_\theta$ that lifts the element $t(a)t(b)t(ab)^{-1}$ of $G^\circ$. We may always assume that both $f(1;b)$ and $f(a;1)$ equal $1$;
\item For all $a$ in $\mu$, we denote by $u(a)$ the automorphism of $G^\circ$ that conjugates an element by $t(a)$. Note that $u(a)u(b)u(ab)^{-1}$ is the conjugation by $t(a)t(b)t(ab)^{-1}$. Since $\theta$ is stable under the action of $\pi_0(G)$ over $\pi_1(G)$, we can always uniquely lift $u(a)$ into an automorphism $v(a)$ of $G^\circ_\theta$, i.e. $p v(a)=u(a)p$. We can notice that $v(a)v(b)v(ab)^{-1}$ lifts $u(a)u(b)u(ab)^{-1}$ so it is necessarily the conjugation by $f(a;b)$. Furthermore, this particular property implies that the composition $\rev{v}:\mu\rightarrow \textnormal{Out}(G^\circ_\theta)$ of $v$ with $\Aut(G^\circ_\theta)\rightarrow \textnormal{Out}(G^\circ_\theta)$ is a group morphism; 
\item Now we can define $m$ by the following formula:
	\begin{equation}
		\forall g,h\in G^\circ_\theta,\;\forall a,b\in \mu,\; m\big((g;a);(h;b)\big):= \big(g(v(a)\cdot h)f(a;b);ab\big).
	\end{equation}
\end{enumerate}
Some direct computations show that, for this particular $m$, the cochain $\gamma$ satisfies the following identity:
\begin{equation}\label{eq:gamma}
	\forall a,b,c\in\mu,\;\big(v(a)\cdot f(b;c)\big)f(a;bc)=e^{\gamma(a;b;c)}f(a;b)f(ab;c).
\end{equation}
This equation is the same as \cite[(8.5')]{mac_lane_homology_1963} so that $\gamma$ is closed, cf. \cite[Lemma~8.4]{mac_lane_homology_1963}. Moreover, its class represents the obstruction to the existence of an extension of $\mu$ by $G^\circ_\theta$ with conjugation class $\rev{v}$, cf. \cite[Theorem~8.7]{mac_lane_homology_1963}. Henceforth, its cohomology class $\chi(\mu;\theta)$ is indeed the one defined in \cite[\S 6]{taylor_covering_1954} and \cite{brown_covering_1994}. Moreover, it precisely represents the obstruction to the existence of an associative $m$ that satisfies (\ref{eq:prop_of_m}).

\begin{dfn}
	Let $\mu\in M(G)$ and $\theta\in\Theta(G)$. There is an action of $H^2(\mu;\pi_1(G)/\theta)$ on $[\E(G)]_{\mu,\theta}$. To define it, we can note that any class $\beta$ in $H^2(\mu;\pi_1(G)/\theta)$ can be represented by a cocycle $\beta'$ for which both $\beta'(a;1)$ and $\beta'(1;a)$ equal $1$ for all $a$ in $\mu$, cf. \cite{mac_lane_homology_1963}. Given a class of extensions $[E]$ in $[\E(G)]_{\mu,\theta}$ we define the class $[E]+\beta$ by introducing a new group law $m$ on $E$, meanwhile the maps $p_E$ and $r_E$ remain the same. The group law $m$ is defined by the explicit formula:
	\begin{equation}\label{eq:grp_law_twist}
		m(x;y)\coloneqq e^{\beta'(r_E(x);r_E(y))}xy,
	\end{equation}
	where the multiplication denotes the initial group law of $E$. Straightforward verifications show that $p_E$ and $r_E$ are still group morphisms and that the class $[E]+\beta$ does not depend on the particular representatives $E$ and $\beta'$. If $E$ and $F$ are two discrete extensions of $G$ of magnitude $\mu$ and $\theta$ whose classes satisfy $[E]=[F]+\beta$ we say that $F$ is a \emph{shift} of $E$ by $\beta$. Furthermore, these actions are functorial in that this defines a natural action of the functor $H^2$ on the functor $[\E(G)]$. 
\end{dfn}

\begin{rem}\label{rem:modified_inverse}
	Let $x$ be an element of $E$, a direct computation yields that the inverse of $x$ for the modified group law (\ref{eq:grp_law_twist}), expressed with the group structure of $E$, is $x^{-1}e^{-\beta'(r_E(x);r_E(x)^{-1})}$. 
\end{rem}

\begin{thm}[(6.5)(b) in \cite{taylor_covering_1954}]
	If $[\E(G)]_{\mu,\theta}$ is non-empty, the action of $H^2(\mu;\pi_1(G)/\theta)$ is simply transitive.
\end{thm}

\begin{prop}\label{prop:split_cc}
	Let $G$ be a Lie group. If the exact sequence:
	\begin{equation}\label{seq:connected}
		1 \rightarrow G^\circ \rightarrow G \rightarrow \pi_0(G)\rightarrow 1, 
	\end{equation}
	splits then there exist discrete extensions of every magnitude and phase.
\end{prop}

\begin{proof}
	We begin by noting that the category $\M(G)\times \T(G)^\textnormal{op}$ has a terminal object $(\pi_0(G);0)$. Following \cite[(7.1)]{taylor_covering_1954}, the splitting of (\ref{seq:connected}) implies the vanishing of $\chi(\pi_0(G);0)$ and the functorial properties of $\chi$ make the proposition follow.
\end{proof}

\begin{rem}
	In general, a Lie group with finitely many connected components admits a finite magnitude $\mu$ that splits (\ref{seq:connected}), i.e. $G\times_{\pi_0(G)}\mu$ is a split extension of $\mu$ by $G^\circ$. Then, Proposition~\ref{prop:split_cc} implies that the obstruction $\chi(\mu;0)$ vanishes. Let $\kappa$ be the kernel of $p_\mu$. Any extension $E$ in $[E(G)]_{\mu,0}$ is an extension of $G\times_{\pi_0(G)}\mu$ by $\pi_1(G)$. Moreover, the kernel $\kappa$ is naturally embedded in $G\times_{\pi_0(G)}\mu$ as the kernel of the projection $G\times_{\pi_0(G)}\mu\rightarrow G$. Therefore, $E$ yields a central extension of $\kappa$ by $\pi_1(G)$, namely the pullback of the projection $E\rightarrow G\times_{\pi_0(G)}\mu$ by the inclusion $\kappa\rightarrow  G\times_{\pi_0(G)}\mu$, cf. Proposition~\ref{prop:kernel_discrete_ext}. Then, the vanishing of $\chi(\pi_0(G);0)$ is equivalent to the splitting of one of these extensions of $\kappa$. Since $\pi_1(G)$ is Abelian, every such extension of $\kappa$ corresponds uniquely to a cohomology class in $H^2(\kappa;\pi_1(G))$ (with trivial action). Twisting $E$ by a class $\beta$ in $H^2(\mu;\pi_1(G))$ (with action induced by $\mu\rightarrow \pi_0(G)$), twists the extension of $\kappa$ by the class $\beta|_\kappa$ in $H^2(\kappa;\pi_1(G))$. This seems to mean that the classes of such extensions of $\kappa$ are all transgressive for the Lyndon-Hochschild-Serre spectral sequence and that the obstruction $\chi(\pi_0(G);0)$ is the image by the transgression of any of those classes, cf. \cite[Chapter III, \S4, Theorem~2]{hochschild_cohomology_1953}. Example~\ref{ex:non-split_ext_S1_Z/2} depicts a Lie group for which the terminal obstruction is non-trivial. 
\end{rem}

\begin{ex}\label{ex:non-split_ext_S1_Z/2}
	The exponential exact sequence $0\rightarrow 2\pi\Z\rightarrow \R \rightarrow \Sph^1 \rightarrow 0$ is equivariant for the action of $\Z/2$ by inversion. One can use it to show that there is exactly two extensions of $\Z/2$ by $\Sph^1$ for which $\Z/2$ acts by inversion on $\Sph^1$. One splits and the other does not. Let $G$ denote the non-split one. If $\sigma$ denotes the inversion on $\Sph^1$ (or equivalently the complex conjugation) the group law of $G$ on $\Sph^1\times\Z/2$ is given by the following formula:
	\begin{equation*}
		(\xi;\varepsilon)\cdot(\zeta;\eta):=(\xi\sigma^\varepsilon(\zeta)(-1)^{\varepsilon\eta};\varepsilon+\eta).
	\end{equation*}
	Let $u$ denote the element $(1;1)$. It has order four, its iterated powers are the following:
	\begin{equation*}
		u^2=(-1;0) \quad u^3=(-1;1) \quad u^4=(1;0)=1.
	\end{equation*}
	The subgroup $\Z/4$ spanned by $u$ surjects itself onto $\Z/2$, so that the magnitude $\Z/4\rightarrow \Z/2$ splits the sequence~(\ref{seq:connected}). The group $G\times_{\Z/2}\Z/4$ is isomorphic to the semi-direct product $\Sph^1\rtimes\Z/4$ where $\Z/4$ acts on $\Sph^1$ via the inversion. The projection $\Sph^1\rtimes\Z/4\rightarrow G$ is then given by the following formula:
	\begin{equation*}
		(\xi;\varepsilon)\longmapsto (\xi;0)\cdot u^\varepsilon\in G.
	\end{equation*}
	As a consequence, its kernel is the subgroup of order two $\{((-1)^\varepsilon;2\varepsilon) : \varepsilon\in \Z/2\}$. Furthermore, $\Z/4$ acts on $\pi_1(G)\cong \Z$ by multiplication by $-1$ througth $\Z/2$ so one can compute that the group $H^2(\Z/4;\pi_1(G))$ is trivial. It implies that there is only one discrete extension $E$ of $G$ with magnitude $\Z/4$ and phase $0$. The extension $E$ is the semi-direct product $\R\rtimes\Z/4$. Its group law is given by the following formula:
	\begin{equation*}
		(t;\varepsilon)\cdot(s;\eta):=(t+(-1)^\varepsilon s;\varepsilon+\eta).
	\end{equation*}
	The pullback of the kernel is given by the subgroup $\{(t;\varepsilon)\;|\;\exists \eta,\;e^{2i\pi t}=(-1)^\eta \textnormal{ and }\varepsilon=2\eta\}$. It is isomorphic to $\Z$ via the following formula:
	\begin{equation*}
		k\longmapsto \big(\tfrac{k}{2};2k\,(\mod 4)\big).
	\end{equation*}
	Since the extension $\Z\rightarrow \Z/2$ is not split, $\chi(\Z/2;0)$ is not trivial. 
\end{ex}

\begin{prop}\label{prop:kernel_discrete_ext}
	Let $E$ be a discrete extension of $G$ of magnitude $\mu$ and phase $\theta$. The kernel of $p_E$ is a central extension of $\kappa$, the kernel of $\mu\rightarrow\pi_0(G)$, by $\pi_1(G)/\theta$. Whenever $\chi(\mu;\theta)$ vanishes, this yields a map $\ker:[\E(G)]_{\mu,\theta} \rightarrow H^2(\kappa;\pi_1(G)/\theta)$ that satisfies:
	\begin{equation}\label{eq:affine_ker}
		\ker\big([E]+\beta\big)=\ker\big([E]\big)+\beta|_\kappa,
	\end{equation}
	for all $\beta$ in $H^2(\mu;\pi_1(G)/\theta)$.
\end{prop}

\begin{proof}
	If there is no extension of $G$ with magnitude $\mu$ and phase $\theta$, then $\ker:[\E(G)]_{\mu,\theta} \rightarrow H^2(\kappa;\pi_1(G)/\theta)$ is the unique map from the empty set to $H^2(\kappa;\pi_1(G)/\theta)$. If on the contrary, there is an extension $E$ of $G$ with magnitude $\mu$ and phase $\theta$. The natural exact sequence (\ref{seq:fib_lES_disc_ext}) yields:
	\begin{equation}\label{seq:kappa_pi1theta}
		0 \longrightarrow \pi_1(G)/\theta \longrightarrow \ker(p_E) \longrightarrow \kappa \longrightarrow 1.
	\end{equation}
	Since $\kappa$ acts on $\pi(G)/\theta$ through $\mu$ itself acting through $p_\mu$, the extension (\ref{seq:kappa_pi1theta}) is central. Therefore, (\ref{seq:fib_lES_disc_ext}) induces a morphism from the groupoid of discrete extensions of $G$ of magnitude $\mu$ and phase $\theta$ to the groupoid of central extensions of $\kappa$ by $\pi_1(G)/\theta$. The map $\ker:[\E(G)]_{\mu,\theta} \rightarrow H^2(\kappa;\pi_1(G)/\theta)$ is then induced between their isomorphisms classes (also called connected components). The formula (\ref{eq:grp_law_twist}) directly induces (\ref{eq:affine_ker}). 
\end{proof}

\subsection{Lift Obstructions for Relatively Connected Extensions}

In this section, we develop obstructions to the lift of a $Q$-group morphism along a \emph{relatively connected} discrete extension. These obstructions will be the fundamental tool to identify the phase of the universal covering of a polar space. We consider a discrete extension $E$ of a Lie group $G$ such that $(p_E)_*:\pi_0(E)\rightarrow \pi_0(G)$ is an isomorphism. We say that $E$ is a \emph{relatively connected} discret extension of $G$. We also consider a simply connected space $Q$, and we ask the following question:
\begin{center}
\textit{Under what condition can a morphism $f:H\rightarrow G_Q$ of $Q$-groups be lifted along $(p_E)_Q$ ?}
\end{center}
Given \S\ref{subsection:Structure of the Category}, the kernel of $p_E$ is necessarily Abelian and isomorphic to $\pi_1(G)/\theta$, where $\theta$ denotes the phase of $E$.

\paragraph{Topological Obstruction} 

The first obstruction is topological. Indeed, before asking if we can lift $f$ as a morphism of $Q$-groups we must ask is we can lift it as a morphism of $Q$-spaces. To lift $f$ as a map of $Q$-spaces, we only need to lift its $G$-component $f_G=\pr_G\circ f$ along $p_E$. Indeed, if $g_G$ is a lift of $f_G$ along $p_E$, then $h\mapsto (g_G(h);||h||)$ is a map of $Q$-spaces that lifts $f$ along $(p_E)_Q$. This first obstruction will be a singular cohomology class of $H$ with coefficients in $\pi_1(G)/\theta$.

\begin{lem}
	Let $G$ be a Lie group, there is a unique natural morphism $\delta_G:H_1(G;\Z)\rightarrow \pi_1(G)$ that satisfies the following commutative triangle for every element $g\in G$:
	\begin{equation}\label{diag:diagonal}
		\begin{tikzcd}
			H_1(G;\Z) \ar[rr,"\delta_G"] && \pi_1(G) \\
			& \pi_1(G;g) \ar[lu,"\textnormal{Hurewicz map}"] \ar[ru,"g^{-1}_*", swap] &
		\end{tikzcd}
	\end{equation}
\end{lem}

\begin{proof}
	First, we notice that $H_1(G;\Z)$ is endowed with the following natural decomposition:
	\begin{equation}\label{diag:decomp}
		H_1(G;\Z)\cong \bigoplus_{C\in \pi_0(G)}H_1(C;\Z).
	\end{equation}
	Let $C$ be a connected component of $G$ and $g$ be an element of $C$. In the decomposition (\ref{diag:decomp}), the diagram (\ref{diag:diagonal}) becomes the following:
	\begin{equation*}
		\begin{tikzcd}
			H_1(C;\Z) \ar[rr,"\delta_G|_{H_1(C;\Z)}"] && \pi_1(G) \\
			& \pi_1(G;g) \ar[lu,"h_g","\cong" swap] \ar[ru,"g^{-1}_*", swap] &
		\end{tikzcd}
	\end{equation*}
	where $h_g$ is the Hurewicz isomorphism (since the fundamental group of a Lie group is Abelian). Thus, satisfying (\ref{diag:diagonal}) for all $g$ in $G$ completely determines $\delta_G$. Hence, if $\delta_G$ exists it has to be unique. For the existence, we note that for $k$ in $G^\circ$, we have the following commutative diagram: 
	\begin{equation*}
		\begin{tikzcd}
			& \pi_1(G;g) \ar[rd,"h_g"] \ar[dd,"k_*"]\\
			\pi_1(G) \ar[ru,"g_*"] \ar[rd,"kg_*" swap]& & H_1(C;\Z) \\
			& \pi_1(G;kg) \ar[ru,"h_{kg}" swap]
		\end{tikzcd}
	\end{equation*}
	since the Hurewicz map is a natural transformation and the left multiplication by $k$ is isotopic the identity of $C$. Hence, the composition $h_g\circ g_*$ and, \emph{a fortiori} its inverse, does not depend on the choice of $g$ in $C$. Summing these inverses yields the desired map $\delta_G$. The naturality is a consequence of the commutativity of every inner square and triangle of the following diagram for every Lie group morphism $f:G\rightarrow H$ and every $g$ in $G$:
	\begin{equation*}
		\begin{tikzcd}
			H_1(G;\Z) \ar[rr,"\delta_G"] \ar[ddd,"f_*" swap] & & \pi_1(G) \ar[ddd,"f_*"] \\
			& \pi_1(G;g) \ar[lu] \ar[ru,"g_*" swap] \ar[d,"f_*" swap] & \\
			& \pi_1\big(H;f(g)\big) \ar[ld] \ar[rd,"f(g)_*"] &\\
			H_1(H;\Z) \ar[rr,"\delta_H" swap] & & \pi_1(H)
		\end{tikzcd}
	\end{equation*}
\end{proof}

\begin{dfn}
	The \emph{topological obstruction} to the lift of $f$ along $E$ is the cohomology class $[f/E]_\textnormal{top}$ in the singular cohomology group $H^1(H;\pi_1/\theta)$ defined as the following morphism\footnote{We use the natural isomorphism of the universal coefficient theorem between $H^1(H;\pi_1/\theta)$ and $\Hom(H_1(H;\Z);\pi_1/\theta)$, cf. \cite[Theorem~3.3]{cartan_homological_1956}.}:
	\begin{equation*}
		\begin{tikzcd}
			{[f/E]}_\textnormal{top}:H_1(H;\Z)\ar[r,"(f_G)_*"]& H_1(G;\Z) \ar[r,"\delta_G"]& \pi_1 \ar[r,"\textnormal{quot.}","\textnormal{proj.}" swap]& \smallslant{\pi_1}{\theta}.
		\end{tikzcd}
	\end{equation*} 
\end{dfn}

\begin{prop}
	The morphism $f$ has a lift along $(p_E)_Q$ in the category of $Q$-spaces if and only if $[f/E]_\textnormal{top}$ vanishes. 
\end{prop}

\begin{proof}
	As we remarked in the beginning of the paragraph, we only need to lift the $G$-component $f_G$ of $f$. Let us choose a point $h_C$ in each connected component $C$ of $H$. We denote the image $f(h_C)$ by $g_C$. The classical lifting criterion, cf. \cite[Proposition~1.33]{hatcher_algebraic_2000}, states that $f_G$ has a lift along $p_E$ if and only if:
	\begin{equation}\label{eq:lift_crit}
		\forall C\in\pi_0(H),\,\exists e_C\in E \textnormal{ with } p_E(e_C)=g_C \textnormal{ such that } (f_G)_*\big(\pi_1(H;h_C)\big)\subset (p_E)_*\big(\pi_1(E;e_C)\big).
	\end{equation}
	Let $g$ be a point of $G$, $e$ be a point of $E$ above $g$, and $M$ be $(p_E)_*(\pi_1(E;e))$. We can note that $g^{-1}_*(M)$ is actually the group $\theta$. Indeed, $g^{-1}_*(p_E)_*$ equals $(p_E)_*e^{-1}_*$ for $p_E$ is a group morphism. Therefore, $g^{-1}_*(M)$ equals $(p_E)_*(\pi_1(E;1))$ i.e. $\theta$. Hence, (\ref{eq:lift_crit}) is equivalent to the following morphism taking its values in $\theta$:
	\begin{equation}\label{eq:freeprodpi1}
		\begin{tikzcd}
			\displaystyle\coprod_{C\in \pi_0(H)}\pi_1(H;h_C) \ar[rr,"\coprod_C (g_C^{-1}f)_*"]&& \pi_1(G).
		\end{tikzcd}
	\end{equation}
	Here $\coprod$ denotes the free product of groups. The property (\ref{diag:diagonal}) of $\delta_G$ implies that we have the following commutative diagram:
	\begin{equation*}
		\begin{tikzcd}
			\displaystyle\coprod_{C\in \pi_0(H)}\pi_1(H;h_C) \ar[d,"\textnormal{Hurewicz map}" swap] \ar[r,"\textnormal{(\ref{eq:freeprodpi1})}"]& \pi_1(G) \\
			H_1(H;\Z) \ar[r,"(f_G)_*" swap]& H_1(G;\Z) \ar[u,"\delta_G" swap]
		\end{tikzcd}
	\end{equation*}
	Thus, $f_G$ has a lift if and only if $[f/E]_\textnormal{top}$ vanishes.
\end{proof}

\begin{ex} Let us consider the following discrete extension $0\rightarrow \Z \rightarrow \R^2 \rightarrow \Sph^1\times\R \rightarrow 0$ and the $D^2$-group $Z$ from Examples~\ref{ex:1}.1. We recall that $Z$ is $\{(k;t)\in\Z\times D^2\;|\;|t|<1\Rightarrow k=0\}$. Now let us consider the following morphism:
		\begin{equation*}
			\begin{array}{rcl}
				f:Z & \longrightarrow &\Sph^1\times\R\times D^2 \\
				(k;t) & \longmapsto & (t^k;k;t).
			\end{array}
		\end{equation*}
		The first homology group of $Z$ with coefficient in $\Z$ is freely generated by the classes of the circles $\{k\}\times\partial D^2$ for all $k$ different from $0$. None of these circles can be lifted in $\R^2$ so the topological obstruction $[f/\R^2]_\textnormal{top}$ is the function whose value is $1$ on the circles $\{k\}\times\partial D^2$ for $k>0$ and $-1$ on the circles $\{k\}\times\partial D^2$ for $k<0$.
\end{ex}

\begin{rem}
	When $H$ is locally simply connected (when it is the isotropy of a regular section of a polar space for instance), there is another way of understanding the topological obstruction. Let $A$ denote either $G$, $E$, or $\pi_1(G)/\theta$, and denote by $\mathcal{C}_A^0$ the sheaf of groups on $H$ given by $U\mapsto \mathcal{C}^0(U;A)$. The assumption on $H$ implies that the following sequence of sheaves is exact:
	\begin{equation*}
		1\longrightarrow \mathcal{C}_{\pi_1(G)/\theta}^0 \longrightarrow \mathcal{C}_E^0 \longrightarrow \mathcal{C}_G^0 \longrightarrow 1.
	\end{equation*}
	We can remark that $\mathcal{C}_{\pi_1(G)/\theta}^0$ is the constant sheaf $\pi_1(G)/\theta$ for the group is discrete. Since it is also Abelian, we have the following exact sequence:
	\begin{equation*}
		1\longrightarrow H^0(H;\pi_1(G)/\theta) \longrightarrow \mathcal{C}^0(H;E) \longrightarrow \mathcal{C}^0(H;G) \overset{\df}{\longrightarrow} H^1(H;\pi_1(G)/\theta).
	\end{equation*}
	In this framework, $[f/E]_\textnormal{top}$ corresponds to $\df (f_G)$, and it can be constructed in \v Cech cohomology using a simply connected open cover of $H$ and local lifts of $f_G$ over every open sets of the cover. 
\end{rem}

\paragraph{Algebraic Obstruction}
When the topological obstruction vanishes, we need to determine if one of the lifts is a $Q$-group morphism. Let $g:H\rightarrow E_Q$ be a lift of $f$ in the category of $Q$-spaces. Any other lift is of the form $e^k_Q\cdot g$ where $k:H\rightarrow (\pi_1(G)/\theta)_Q$ is a morphism of $Q$-spaces, $e_Q$ is induced by the exponential $\pi_1(G)/\theta\rightarrow E$, and the product is the $Q$-group law of $E_Q$. One can consider the following map of $Q$-spaces:
\begin{equation}\label{eq:zeta}
	\begin{array}{rcl}
		\zeta:H\times_Q H &\longrightarrow & E_Q \\
		(h_1,h_2) & \longmapsto & g(h_1)g(h_2)g(h_1h_2)^{-1}.
	\end{array}
\end{equation} 
Since $g$ lifts a group morphism, $\zeta$ actually takes its values in $(\pi_1(G)/\theta)_Q$. We also note that $(\pi_1(G)/\theta)_Q$ is a $G_Q$-module and thus a $H$-module through $f$.
\begin{prop}\label{prop:zeta_cocycle}
	The cochain $\zeta\in C^2_Q(H;(\pi_1(G)/\theta)_Q)$ defined by (\ref{eq:zeta}) is a cocycle. Furthermore, considering all possible other lifts of $g$ describes the full cohomology class of $\zeta$.
\end{prop}

\begin{proof}
	Let us prove that $\zeta$ is a cocycle. This condition needs only to be checked on the fibres of $H\rightarrow Q$. Let $x,y,z$ be three elements of $H$ in the same fibre. By definition, we have:
	\begin{equation*}
		\df\zeta(x;y;z)= x\cdot\zeta(y;z)-\zeta(xy;z)+\zeta(x;yz)-\zeta(x;y).
	\end{equation*}
	To show that it vanishes, we will compute $e_Q^{\df\zeta(x,y,z)}$ where  $e_Q$ denotes the morphism ${(\pi_1(G)/\theta)_Q\rightarrow E_Q}$ induced by the exponential of $E$. Since we see $(\pi_1(G)/\theta)_Q$ as an $H$-module through $f$, we have the following identity:
	\begin{equation}\label{id:exp_conj}
		e_Q^{h\cdot k}=g(h)e_Q^kg(h)^{-1},
	\end{equation}
	for all $h$ in $H$ and all $k$ in $(\pi_1/\theta)_Q$ above the same point of $Q$ even though $g$ is not necessarily a group morphism. Hence, we have the following computation:
	\begin{align*}
			e_Q^{\df\zeta(x;y;z)}&=g(x)e_Q^{\zeta(y;z)}g(x)^{-1}\left(e_Q^{\zeta(xy;z)}\right)^{-1}e_Q^{\zeta(x;yz)}\left(e_Q^{\zeta(x;y)}\right)^{-1}\\[5pt]
			&=g(x)g(y)g(z)\underset{\textnormal{we exchange these terms for they commute}}{\underbrace{\underset{\textnormal{in }\pi_1/\theta\textnormal{ for }(yz)^{-1}x^{-1}xyz=1}{\underbrace{\Big(g(yz)^{-1}g(x)^{-1}g(xyz)\Big)}}\;\underset{\textnormal{in }\pi_1/\theta\textnormal{ for }z^{-1}(xy)^{-1}xyz=1}{\underbrace{\Big(g(z)^{-1}g(xy)^{-1}g(x)g(yz)\Big)}}}}g(xyz)^{-1}g(xy)g(y)^{-1}g(x)^{-1}\\[5pt]
			&=g(x)g(y)\Big(g(z)g(z)^{-1}\Big)g(xy)^{-1}g(x)\Big(g(yz)g(yz)^{-1}\Big)g(x)^{-1}\Big(g(xyz)g(xyz)^{-1}\Big)g(xy)g(y)^{-1}g(x)^{-1}\\[5pt]
			&=\Big(g(x)g(y)g(xy)^{-1}g(x)\Big)\Big(g(x)^{-1}g(xy)g(y)^{-1}g(x)^{-1}\Big)\\[5pt]
			&=\Big(g(x)g(y)g(xy)^{-1}g(x)\Big)\Big(g(x)g(y)g(xy)^{-1}g(x)\Big)^{-1}\\[5pt]
			&=1.
	\end{align*}
	Therefore, $\zeta$ is a cocycle. Now let us consider another lift of $f$. It has the form $kg$ where $k:H\rightarrow (\pi_1/\theta)_Q$ is a morphism of $Q$-spaces, or equivalently a $1$-cochain. Let $\zeta_k$ be the cocycle associated to this lift. We have:
	\begin{equation*}
		\begin{aligned}
			e_Q^{\zeta_k(x;y)}&=\underset{\textnormal{equals }e_Q^{x\cdot k(y)}g(x) \textnormal{ by (\ref{id:exp_conj})}}{e_Q^{k(x)}\underbrace{g(x)e_Q^{k(y)}}g(y)g(xy)^{-1}}e_Q^{-k(xy)}= e_Q^{k(x)+x\cdot k(y)}g(x)g(y)g(xy)^{-1}e_Q^{-k(xy)}\\[5pt]
			&=e_Q^{k(x)+x\cdot k(y)+\zeta(x;y)-k(xy)}=e_Q^{\zeta(x;y)+\df k(x;y)}.
		\end{aligned}
	\end{equation*}
	Thus, the proposition follows.
	\end{proof}

\begin{dfn}
	Let us assume that the topological obstruction $[f/E]_\textnormal{top}$ is trivial. We denote the cohomology class of $\zeta$ of Proposition~\ref{prop:zeta_cocycle} by:
	\begin{equation*}
		[f/E]_\textnormal{alg}\in H_Q^2\big(H;(\pi_1(G)/\theta)_Q\big).
	\end{equation*}
	We call it the \emph{algebraic obstruction} to the lift of $f$ along $E$. 
\end{dfn}

\begin{ex}
	Let $Q$ be a point, $G$ be $\Sph^1$, $E$ be the square $\Sph^1\rightarrow \Sph^1$, and $f$ be the inclusion of $\Z/2$ into $\Sph^1$. The map $g$ that sends $0$ to $1$ and $1$ to $i$ is a topological lift of $f$ along $p_E$. A straightforward computation yields $\zeta(x;y)=xy$ (the product is from the ring structure of $\Z/2$). This is the cocycle that represents the only non-trivial class of $H^2(\Z/2;\Z/2)$. 
\end{ex}

Proposition~\ref{prop:zeta_cocycle} and the definition of the cocycle $\zeta$, cf.~(\ref{eq:zeta}), directly yields the following proposition.

\begin{prop}\label{prop:lifts_vanish_obs}
	The morphism $f$ has a lift along $(p_E)_Q$ if both its topological and algebraic obstructions vanish. 
\end{prop}

Moreover, we easily deduce from the definition that the obstructions are natural in the following sense:

\begin{prop}\label{prop:change_of_phase_obs}
	Let $E$ be a relatively connected discrete extension of $G$ and $E'$ be another discrete extension of $G$ endowed with a morphism of extensions $E\rightarrow E'$. We denote the phase of $E$ by $\theta$ and the phase of $E'$ by $\theta'$. The latter contains the former and we have:\begin{enumerate}
	\item $E'$ is a relatively connected extension of $G$;
	\item $[f/E']_\textnormal{top}$ is the image of $[f/E]_\textnormal{top}$ by the natural morphism $H^1(H;\pi_1(G)/\theta)\rightarrow H^1(H;\pi_1(G)/\theta')$;
	\item If $[f/E]_\textnormal{top}$ vanishes so does $[f/E']_\textnormal{top}$, and $[f/E']_\textnormal{alg}$ is the image of $[f/E]_\textnormal{alg}$ by the natural morphism $H^2_Q(H;(\pi_1(G)/\theta)_Q)\rightarrow H^2_Q(H;(\pi_1(G)/\theta')_Q)$.
	\end{enumerate}
\end{prop}

\paragraph{Shifting the Extension}
Following \S\ref{subsection:Structure of the Category}, a relatively connected discrete extension $E$ of $G$ can be shifted by a class of $H^2(\pi_0(G);\pi_1(G)/\theta)$. In this paragraph we express how these shifts affect the obstructions.

\begin{prop}\label{prop:twist_top_obs}
	If $E$ and $F$ are two relatively connected discrete extensions of $G$ with the same phase, then they yield the same topological obstructions. 
\end{prop}

\begin{proof}
	If two relatively connected discrete extensions of $G$ are shifts of each others then they are isomorphic as coverings of $G$. Therefore, they yield identical topological obstructions.  
\end{proof}

Proposition~\ref{prop:twist_top_obs} implies that if the topological obstruction of a given extension vanishes then so are the topological obstructions of all its shifts. We note that the pullback:
	\begin{equation*}
		\eta^*_{G_Q}:H^2(\pi_0(G);\pi_1(G)/\theta) \longrightarrow H^2_Q(G_Q;(\pi_1(G)/\theta)_Q),
	\end{equation*}
is an isomorphism, cf. Remark~\ref{rem:cohomology_constant_groups}.

\begin{prop}\label{prop:twist_alg_obs}
	Let $E$ be a relatively connected discrete extension of $G$ whose topological obstruction $[f/E]_\textnormal{top}$ vanishes. If $F$ is a shift of $E$ by $\beta$ in $H^2(\pi_0(G);\pi_1(G)/\theta)$, then $[f/F]_\textnormal{top}$ vanishes and:
	\begin{equation*}
		[f/F]_\textnormal{alg}=[f/E]_\textnormal{alg}+f^*\eta^*_{G_Q}\beta. 
	\end{equation*}
\end{prop}

\begin{proof}
	The first part of the proposition is a consequence of Proposition~\ref{prop:twist_top_obs} so that it makes sense to compute the algebraic obstruction $[f/F]_\textnormal{alg}$. Remember that the algebraic obstruction $[f/E]_\textnormal{alg}$ is the class of the cocycle $\zeta$ defined in (\ref{eq:zeta}). We consider $\beta'$ a cocycle such that $\beta'(a;1)$ and $\beta'(1;a)$ both equal $1$ for all $a$ in $\pi_0(G)$, and that represents $\beta$. Therefore, the extension $F$ is isomorphic to $E'$ the extension with underlying manifold $E$ endowed with the modified group law (\ref{eq:grp_law_twist}) and projection $p_E$. We denote by $\zeta'$ the cocycle (\ref{eq:zeta}) representing $[f/E']_\textnormal{alg}$ and thus equivalently $[f/F]_\textnormal{alg}$. Let $(h_1;h_2)$ be a point in $H\times_Q H$, we denote by $a_1$ (resp. $a_2$) the image of $f(h_1)$ (resp. $f(h_2)$) in $\pi_0(G)_Q$. Since $p_E$ induces an isomorphism between the groups of connected components of $E$ and $G$, $a_1$ (resp. $a_2$) also represents the connected component of $E_G$ to which $g(h_1)$ (resp. $g(h_2)$) belongs. This also implies that if $k$ belongs to $(\pi_1(G)/\theta)_Q$ and $g$ belongs to $E_Q$ and that they both lie over the same point of $Q$, then $e_Q^kg$, $ge_Q^k$, and $g$ belong to the same connected component. Therefore, we have:
	\begin{align*}
		e_Q^{\zeta'(h_1;h_2)}&=\big(g(h_1)\underset{\beta}{\cdot}g(h_2)\big)\underset{\beta}{\cdot}g(h_1h_2)^{(-1)_\beta} \quad (\textnormal{ \footnotesize$\underset{\beta}{\cdot}$ is the modified group law and $(-)^{(-1)_\beta}$ the modified inverse}) \\
		&=\big(e_Q^{\beta'(a_1;a_2)}g(h_1)g(h_2)\big)\underset{\beta}{\cdot}\big(g(h_1h_2)^{-1}e_Q^{-\beta'(a_1a_2;(a_1a_2)^{-1})}\big) \quad (\textnormal{\footnotesize cf. (\ref{eq:grp_law_twist}) and Remark \ref{rem:modified_inverse}})\\
		&=e_Q^{\beta'(a_1a_2;(a_1a_2)^{-1})+\beta'(a_1;a_2)}g(h_1)g(h_2)g(h_1h_2)^{-1}e_Q^{-\beta'(a_1a_2;(a_1a_2)^{-1})}\\[5pt]
		&=e_Q^{\beta'(a_1a_2;(a_1a_2)^{-1})+\beta'(a_1;a_2)+\zeta(h_1;h_2)-\beta'(a_1a_2;(a_1a_2)^{-1})}\\[5pt]
		&=e_Q^{\beta'(a_1;a_2)+\zeta(h_1;h_2)}.
	\end{align*}
	And the proposition follows. 
\end{proof}

\subsection{Effective Extensions}
	Let us consider $(G;X)$ a regular polar space. Recall that, in Definition~\ref{dfn:funct_Br}, we introduced the subcategory $\E(G;X)$ of effective extensions of $G$ relatively to $X$. Moreover, we characterised its objects in Corollary~\ref{cor:effective_ubi}.
	
\begin{prop}\label{prop:full_subcat_effective}
	Let $E$ and $F$ be two discrete extensions of $G$ and $f:E\rightarrow F$ be a morphism of extensions. If $E$ is effective relatively to $X$, so is $F$. In particular, $\E(G;X)$ is a full subcategory of $\E(G)$.
\end{prop}

\begin{proof}
	Let $Q$ be the quotient space of $(G;X)$ and $u:H\rightarrow G_Q$ be the isotropy of a regular section. By Corollary~\ref{cor:effective_ubi}, we can find a ubiquitous lift $\rev{u}$ of $u$ along $(p_E)_Q$. Thus, $f_Q\circ \rev{u}$ is a lift of $u$ along $(p_F)_Q$. It is ubiquitous for $f$ is surjective by assumption. Therefore, by Corollary~\ref{cor:effective_ubi}, $F$ is effective relatively to $X$. 
\end{proof}

\begin{dfn}
	Let $(G;X)$ be a polar space and $X'$ be a polar covering of $X$. We will refer to the magnitude and phase of the discrete extension $G'$ of $G$ induced by the functor $\Br$ as the \emph{magnitude} and the \emph{phase} of the covering respectively. Moreover, we say that a magnitude $\mu$ (resp. a phase $\theta$) is \emph{effective relatively to} $X$ or (\emph{effective} if the context is clear) if it is the magnitude (resp. phase) of a polar covering of $X$. We denote by $\M(G;X)$ (resp. $\Theta(G;X)$) the subcategory of $\M(G)$ (resp. $\Theta(G)$) made of effective magnitudes (resp. phases) and morphisms arising from polar coverings.
\end{dfn}

\paragraph{Effective Magnitudes}
Let $(G;X)$ be a regular polar space with quotient space $Q$ and $H$ be the isotropy of a regular section $x$. Since $X$ is connected, Proposition~\ref{prop:ubi_surj} yields a distinguished magnitude of $G$. Namely, the extension $\hat{\pi}_0(H)\rightarrow \pi_0(G)$ induced by the inclusion of $H$ in $G_Q$.

\begin{prop}\label{prop:effective_magnitudes}
	A magnitude $\mu$ is effective relatively to $X$ if and only if it is dominated by $\hat{\pi}_0(H)$, i.e. it admits a morphism of magnitudes $\hat{\pi}_0(H)\rightarrow \mu$. Moreover, $\M(G;X)$ is a full subcategory of $\M(G)$.
\end{prop}

\begin{proof}
	The necessity follows Corollary~\ref{cor:effective_ubi}. Indeed, if $p_\mu:\mu\rightarrow \pi_0(G)$ is an effective magnitude one can find an extension $E$ of $G$ with the given magnitude that admits an ubiquitous lift $u:H\rightarrow E_Q$ of the inclusion of $H$ in $G_Q$ along $p_Q$. Hence, we have the following commutative diagram:
	\begin{equation*}
		\begin{tikzcd}
			\pi_0(E) \ar[r,"(r_E)_*"] & \mu \ar[d,"p_\mu"]\\
			\hat{\pi}_0(H) \ar[u, "u_*"] \ar[r] & \pi_0(G) 
		\end{tikzcd}
	\end{equation*}
	Since $u_*$ is surjective, cf. Proposition~\ref{prop:ubi_surj}, $(r_E)_*\circ u_*$ is a morphism of magnitudes. The sufficiency follows the next construction. Assume $\mu$ is endowed with a morphism of magnitudes $u_*:\hat{\pi}_0(H)\rightarrow \mu$. We can form the following extension of $G$:
	\begin{equation*}
		p:\underset{\pi_0(G)}{G\times\mu} \longrightarrow G.
	\end{equation*} 
	Following Proposition~\ref{prop:functor_iso_ext}, the magnitude of this extension is $\mu$ and its phase is $\pi_1(G)$. The $Q$-group $(G\times_{\pi_0(G)}\mu)_Q$ is naturally isomorphic to the fibre product of $G_Q$ and $\mu_Q$ over $\pi_0(G)_Q$. The inclusion of $H$ in $G_Q$ and the morphism $(u_*)_Q\circ \eta_H : H\rightarrow \mu_Q$ induce the same morphism from $H$ to $\pi_0(G)_Q$. Therefore, they yield a lift of the inclusion of $H$ in $G_Q$ along $p$. This lift is ubiquitous for it induces $u_*:\hat{\pi}_0(H)\rightarrow \mu$ which is surjective by assumption. The same construction proves the remaining of the proposition. If $f_*:\mu\rightarrow \mu'$ is a morphism of magnitudes, then $\mu'$ is also dominated by $\hat{\pi}_0(H)$ by simply considering $f_*\circ u_*$. Moreover, $f_*$ gives rise to a morphism of extensions $f:(G\times_{\pi_0(G)}\mu) \rightarrow (G\times_{\pi_0(G)}\mu')$ by universal property of the fibre product. It sends the ubiquitous lift induced by $u_*$ on the ubiquitous lift yielded by $f_*\circ u_*$. Hence, this morphism is induced by a morphism of polar coverings, cf. Theorem~\ref{thm:equiv_pol_cov}. Since, it induces $f_*$ between the magnitudes, we find that $\M(G;X)$ is a full subcategory of $\M(G)$.
\end{proof}

\begin{cor}\label{cor:exact_seq_pi1}
	Let $(G;X)$ be a regular polar space with quotient space $Q$, and $H$ be the isotropy of a regular section. The fundamental group of $X$ is a central extension of the kernel of the morphism $\hat{\pi}_0(H)\rightarrow \pi_0(G)$ by a quotient of $\pi_1(G)$ by a $\pi_0(G)$-submodule $\theta_\textnormal{uni}$. In particular, we have an exact sequence:
	\begin{equation*}
		0 \longrightarrow \theta_\textnormal{uni} \longrightarrow \pi_1(G) \longrightarrow \pi_1(X) \longrightarrow \hat{\pi}_0(H)\rightarrow \pi_0(G) \longrightarrow 1.
	\end{equation*}
\end{cor}

\begin{proof}
	The magnitude of the universal covering of $X$ is necessarily an initial element of $\M(G;X)$. By proposition~\ref{prop:effective_magnitudes} it has to be $u_*:\hp(H)\rightarrow \pi_0(G)$ where $u$ denotes the inclusion of $H$ in $G_Q$. The remaining of the corollary is then a simple consequence of Propositions~\ref{prop:exact_seq_lift_group}~and~\ref{prop:kernel_discrete_ext}. 
\end{proof}

\begin{ex}
	Let $X$ be a real toric variety under the action of $\G{\R}^n$. We assume its real locus $X(\R)$ is connected. Hence, $(\{\pm1\}^n;X(\R))$ is a polar space, cf. Example~\ref{ex:tor_var1}. Since the acting group is discrete Corollary~\ref{cor:exact_seq_pi1} breaks down to the following exact sequence:
	\begin{equation*}
		1 \longrightarrow \pi_1(X) \longrightarrow \underset{e\in \Cell\, X_+}{\colim} \{\pm1\}^n_e \rightarrow \{\pm1\}^n \longrightarrow 1.
	\end{equation*}
	Where we abuse our conventions and consider the stratification of $X_+$ as a cellular structure (this is an honest cellular structure when $X$ is complete), and denote by $\{\pm1\}^n_e$ the isotropy group of any point of the strata $e$. This is precisely the statement of \cite[Corollary~4.5]{davis_convex_1991}. We can also recall here that the colimit is a right-angled Coxeter group and that it is, in this case, the extension of the universal covering of $X(\R)$. Let us illustrate this with two concrete examples: \begin{enumerate}
	\item Let us consider the real projective line. We depicted its polar coordinate in Figure~\ref{fig:pol_coor_line}. Recall that its quotient is the line segment $[0;\infty]$. In this case, the colimit is the free product of $\{\pm1\}$ with itself, i.e. the group of presentation $\langle a;b\;|\;a^2;b^2\rangle$. In the polar coordinates of the universal covering of $\Proj^1(\R)$, the isotropy of $0$ is $\langle a\rangle$ while the isotropy of $\infty$ is $\langle b\rangle$. Figure~\ref{fig:real_torline} depicts the gluings induced by the polar coordinates. The fundamental group is spanned by $ba$. In the picture, the translation by $ba$ corresponds to a jump of two steps;
	\item If we consider now the case of the real projective plane, an elementary computation shows that the colimit is $\{\pm1\}^3$ and that a basis of $\F_2$ vector space is provided by the images of the generators of the isotropy of the three edges of the triangle, cf. Figure~\ref{fig:isoP2}. Figure~\ref{fig:real_torplane} depicts the gluings performed by the polar coordinates of the universal covering of $\Proj^2(\R)$. One can recover the lifts of the isotropy groups of the cells of the triangle from the picture.
	\end{enumerate}
\end{ex}

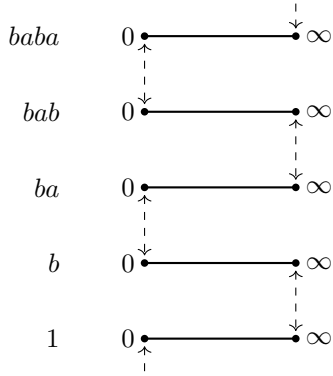
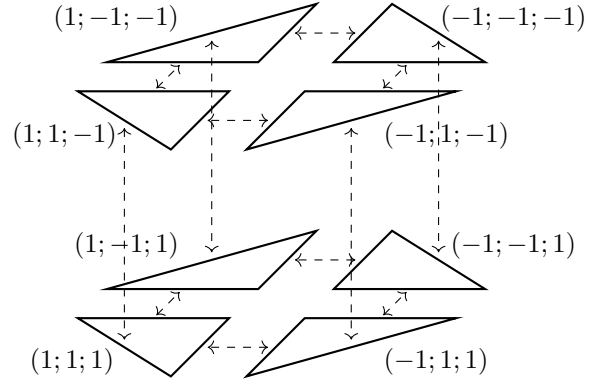
\begin{figure}[H]
	\centering
	\begin{subfigure}[t]{.45\textwidth}
		\centering
		\begin{tikzpicture}
			\draw[thick] (0,5) node[left]{$0$}-- (2,5) node[right]{$\infty$};
			\draw[thick] (0,4) node[left]{$0$}-- (2,4) node[right]{$\infty$};
			\draw[thick] (0,3) node[left]{$0$}-- (2,3) node[right]{$\infty$};
			\draw[thick] (0,2) node[left]{$0$}-- (2,2) node[right]{$\infty$};
			\draw[thick] (0,1) node[left]{$0$}-- (2,1) node[right]{$\infty$};
			\draw[<-,dashed] (2,5.1) -- (2,5.5);
			\draw[<->,dashed] (0,4.9) -- (0,4.1);
			\draw[<->,dashed] (2,3.9) -- (2,3.1);
			\draw[<->,dashed] (0,2.9) -- (0,2.1);
			\draw[<->,dashed] (2,1.9) -- (2,1.1);
			\draw[<-,dashed] (0,.9) -- (0,.5);
			\draw (-1,1) node[left]{$1$};
			\draw (-1,2) node[left]{$b$};
			\draw (-1,3) node[left]{$ba$};
			\draw (-1,4) node[left]{$bab$};
			\draw (-1,5) node[left]{$baba$};
			\fill (0,5) circle (.05) (2,5) circle (.05);
			\fill (0,4) circle (.05) (2,4) circle (.05);
			\fill (0,3) circle (.05) (2,3) circle (.05);
			\fill (0,2) circle (.05) (2,2) circle (.05);
			\fill (0,1) circle (.05) (2,1) circle (.05);
		\end{tikzpicture}
		\caption{The Polar Coordinates of the Universal Covering of the Real Projective Line.}
		\label{fig:real_torline}
	\end{subfigure}
	\hfill
	\begin{subfigure}[t]{.5\textwidth}
		\centering
		\begin{tikzpicture}
			\draw[thick] (.5,0,.5) -- (2.5,0,.5) -- (.5,0,2.5) -- cycle;
			\draw[thick] (.5,0,-.5) -- (2.5,0,-.5) -- (.5,0,-2.5) -- cycle;
			\draw[thick] (-.5,0,-.5) -- (-2.5,0,-.5) -- (-.5,0,-2.5) -- cycle;
			\draw[thick] (-.5,0,.5) -- (-2.5,0,.5) -- (-.5,0,2.5) -- cycle;
			\draw[thick] (.5,3,.5) -- (2.5,3,.5) -- (.5,3,2.5) -- cycle;
			\draw[thick] (.5,3,-.5) -- (2.5,3,-.5) -- (.5,3,-2.5) -- cycle;
			\draw[thick] (-.5,3,-.5) -- (-2.5,3,-.5) -- (-.5,3,-2.5) -- cycle;
			\draw[thick] (-.5,3,.5) -- (-2.5,3,.5) -- (-.5,3,2.5) -- cycle;
			\draw[<->, dashed] (1.5,0,.4) -- (1.5,0,-.4);
			\draw[<->, dashed] (1.5,3,.4) -- (1.5,3,-.4);
			\draw[<->, dashed] (-1.5,0,.4) -- (-1.5,0,-.4);
			\draw[<->, dashed] (-1.5,3,.4) -- (-1.5,3,-.4);
			\draw[<->, dashed] (.4,0,1.5) -- (-.4,0,1.5);
			\draw[<->, dashed] (.4,3,1.5) -- (-.4,3,1.5);
			\draw[<->, dashed] (.4,0,-1.5) -- (-.4,0,-1.5);
			\draw[<->, dashed] (.4,3,-1.5) -- (-.4,3,-1.5);
			\draw[<->, dashed] (1.5,2.9,1.5) -- (1.5,0.1,1.5);
			\draw[<->, dashed] (-1.5,2.9,1.5) -- (-1.5,.1,1.5);
			\draw[<->, dashed] (-1.5,2.9,-1.5) -- (-1.5,.1,-1.5);
			\draw[<->, dashed] (1.5,2.9,-1.5) -- (1.5,.1,-1.5);
			\draw (2,0,2) node[right]{$(-1;1;1)$};
			\draw (2,3,2) node[right]{$(-1;1;-1)$};
			\draw (-2,0,2) node{$(1;1;1)$};
			\draw (-2.1,3,2) node{$(1;1;-1)$};
			\draw (2.3,0,-2) node{$(-1;-1;1)$};
			\draw (2.3,3,-2) node{$(-1;-1;-1)$};
			\draw (-2,0,-2) node[left]{$(1;-1;1)$};
			\draw (-2,3,-2) node[left]{$(1;-1;-1)$};
		\end{tikzpicture}
		\caption{The Polar Coordinates of the Universal Covering of the Real Projective Plane.}
		\label{fig:real_torplane}
	\end{subfigure}
	\caption{The Polar Coordinates of the Universal Coverings of Two Real Toric Varieties.}
	\label{fig:real_tor}
\end{figure}

\paragraph{Effective Phases} In this last paragraph, we use our lift obstruction classes to characterise effective phases. As in the previous paragraph $(G;X)$ denotes a regular polar space with quotient space $Q$ and $u:H\rightarrow G_Q$ denotes the inclusion of the isotropy of a regular section of the quotient map. 

\begin{prop}
	Let $\theta$ and $\theta'$ be two phases of $G$. If $\theta'$ contains $\theta$ and $\theta$ is effective relatively to $X$, then $\theta'$ is also effective. In particular, $\Theta(G;X)$ is the full subcategory of $\Theta(G)$ made of phases containing the phase of the universal covering of $X$. 
\end{prop}

\begin{proof}
	Let $E$ be a discrete extension of $G$ of phase $\theta$. If $\theta'$ is a phase that contains $\theta'$, then $\theta'/\theta$ is contained in the kernel of $p_E$ and $E/(\theta'/\theta)$ endowed with the morphism induced by $p_E$ is a discrete extension of $G$ of phase $\theta'$. Furthermore, the quotient morphism $E\rightarrow E/(\theta'/\theta)$ is a morphism of discrete extensions so, by Proposition~\ref{prop:full_subcat_effective}, if $E$ is effective so is $E/(\theta'/\theta)$. Hence, the proposition follows. 
\end{proof}

\begin{prop}\label{prop:topological_minoration}
	If a phase $\theta$ is effective relatively to $X$, it must contain $\theta_\textnormal{top}$, the $\pi_0(G)$-submodule spanned by the image of the morphism:
	\begin{equation*}
		H_1(H;\Z)\overset{u_*}{\longrightarrow} H_1(G;\Z) \overset{\delta_G}{\longrightarrow}\pi_1(G). 
	\end{equation*}
\end{prop}

\begin{proof}
	Let $E$ be an effective discrete extension of $G$ of phase $\theta$ and $\rev{u}:H\rightarrow E_Q$ be a ubiquitous lift of $u$ along $(p_E)_Q$. We can use $u_*:\hp(H)\rightarrow \pi_0(G)$ and $\rev{u}_*:\hp(H)\rightarrow \pi_0(E)$ to form the fibre products $G\times_{\pi_0(G)}\hp(H)$ and $E\times_{\pi_0(E)}\hp(H)$ so that we have the following commutative diagram:
	\begin{equation*}
		\begin{tikzcd}
			&& E_Q\times_{\pi_0(G)_Q}\hp(H)_Q \ar[r,equal]& \left( E\times_{\pi_0(E)}\hp(H)\right)_Q   \ar[d,"(p_E\times\id_{\hp(H)})_Q"] \\
			H \ar[rr,"(u;\eta_H)" swap] \ar[rru,"(\rev{u};\eta_H)"] && G_Q\times_{\pi_0(G)_Q}\hp(H)_Q \ar[r,equal] & \left( G\times_{\pi_0(G)}\hp(H)\right)_Q
		\end{tikzcd}
	\end{equation*}
	Since the vertical arrow is a relatively connected discrete extension, the topological obstruction to the lift of $(u;\eta_H)$ along $E\times_{\pi_0(G)}\hp(H)$ vanishes and the morphism:
	\begin{equation*}\begin{tikzcd}
		H_1(H;\Z)\ar[rr,"(u;\eta_H)_*"] && H_1\big(G\times_{\pi_0(G)}\hp(H);\Z\big) \ar[rr,"\delta_{G\times_{\pi_0(G)}\hp(H)}"] &&\pi_1\big(G\times_{\pi_0(G)}\hp(H)\big),
	\end{tikzcd}\end{equation*}
	takes its values in $\theta$. But we have the following commutative diagram:
	\begin{equation*}\begin{tikzcd}
		&& H_1\big(G\times_{\pi_0(G)}\hp(H);\Z\big) \ar[rr,"\delta_{G\times_{\pi_0(G)}\hp(H)}"] \ar[dd,"(\pr_G)_*"] &&\pi_1\big(G\times_{\pi_0(G)}\hp(H)\big) \ar[dd,"(\pr_G)_*", "\cong" swap]\\
		H_1(H;\Z)\ar[rru,"(u;\eta_H)_*"] \ar[rrd,"u_*" swap] \\
		&& H_1(G;\Z) \ar[rr,"\delta_{G}" swap] &&\pi_1(G)
	\end{tikzcd}\end{equation*}
	Thus $\theta$ must contain $\theta_\textnormal{top}$ for it is a $\pi_0(G)$-submodule of $\pi_1(G)$.
\end{proof}

\begin{dfn}
	Let us denote by $\Theta^*(G;X)$ the full subcategory of phases $\theta$ containing $\theta_\textnormal{top}$ and whose obstruction $\chi(\hat{\pi}_0(H);\theta)$ vanishes. It contains $\Theta(G;X)$ by direct application of Theorem~\ref{thm:taylor} and Proposition~\ref{prop:topological_minoration}.  
\end{dfn}

\begin{prop}\label{prop:effective_phases}
	For all $\theta$ in $\Theta^*(G;X)$, there is a class:
	\begin{equation*}
		\xi(\theta)\in\Hom_{\Z[\hp(H)]}(\Delta_2(H);\pi_1(G)/\theta),
	\end{equation*}
	that vanishes if and only if $\theta$ is effective relatively to $X$. Moreover if $\theta_1\leq \theta_2$ are two phases in $\Theta^*(G;X)$, then $\xi(\theta_2)$ is the reduction of $\xi(\theta_1)$ modulo $\theta_2/\theta_1$.
\end{prop}

\begin{proof}
	We start by considering the following commutative diagram:
	\begin{equation*}
		\begin{tikzcd}
			& (G\times_{\pi_0(G)}\hp(H))_Q \ar[d,"(\pr_G)_Q"] \\
			H \ar[r,"u" swap] \ar[ru,"(u;\eta_H)"] & G_Q
		\end{tikzcd}
	\end{equation*}
	where $(u;\eta_H)$ is, by construction, a ubiquitous lift of $u$ along $(\pr_G)_Q$. Let now $\theta$ be a phase of $\Theta^*(G.X)$. By hypothesis, there exists a relatively connected discrete extension $E$ of $G\times_{\pi_0(G)}\hp(H)$ of phase $\theta$. Since $\theta$ contains $\theta_\textnormal{top}$, Proposition~\ref{prop:topological_minoration} and its proof ensure that $[(u;\eta_H)/E]_\textnormal{top}$ vanishes. We define $\xi(\theta)$ to be the image of $[(u;\eta_H)/E]_\textnormal{alg}$ in $\Hom_{\Z[\hp(H)]}(\Delta_2(H);\pi_1(G)/\theta)$ through the projection of the natural exact sequence of Corollary~\ref{cor:exact_seq_h2}: 
	\begin{equation}\label{seq:Cohomo-Q-Cohomo}
		0 \longrightarrow H^2(\hp(H);\pi_1(G)/\theta) \overset{\eta_H^*}{\longrightarrow} H^2_Q(H;(\pi_1(G)/\theta)_Q) \longrightarrow \Hom_{\Z[\hp(H)]}(\Delta_2(H);\pi_1(G)/\theta) \longrightarrow 0
	\end{equation}
	The class $\xi(\theta)$ does not depend on the particular choice of relatively connected discrete extension $E$ for (\ref{seq:Cohomo-Q-Cohomo}) is exact. Indeed, a different choice would necessarily be a shift of $E$ by a class $\beta$ in $H^2(\hp(H);\pi_1(G)/\theta)$, hence resulting in a obstruction $[(u;\eta_H)/E]_\textnormal{alg}+\eta^*_H\beta$ for $\eta_{G\times_{\pi_0(G)}\hp(H)}\circ (u;\eta_H)$ equals $\eta_H$, cf. Proposition~\ref{prop:twist_alg_obs}. 
	
	\vspace{5pt}
	
	If $\xi(\theta)$ vanishes, there is a (unique) $\beta$ in $H^2(\hp(H);\pi_1(G)/H)$ such that $[(u;\eta_H)/E]_\textnormal{alg}$ equals $\eta_H^*\beta$. Thus any shift $F$ of $E$ by $-\beta$ will admit a lift $\rev{u}$ of $(u;\eta_H)$, cf. Propositions~\ref{prop:twist_alg_obs}~and~\ref{prop:lifts_vanish_obs}. This lift is necessarily ubiquitous. The composition $\pr_G\circ p_F:F\rightarrow G$ is discrete extension of $G$ of phase $\theta$ that has a ubiquitous lift $\rev{u}$ of $u$. Therefore, $\theta$ is effective. 
	
	\vspace{5pt}
	
	Reciprocally, let us assume that $\theta$ is effective. We can find a discrete extension $E$ of $G$ of phase $\theta$ that admits a ubiquitous lift $\rev{u}$ of $u$. As in the proof of Proposition~\ref{prop:topological_minoration}, we use $u_*:\hp(H)\rightarrow \pi_0(G)$ and $\rev{u}_*:\hp(H)\rightarrow \pi_0(E)$ to form the fibre products $G\times_{\pi_0(G)}\hp(H)$ and $E\times_{\pi_0(E)}\hp(H)$. The projection $p_E\times\id_{\hp(H)}: E\times_{\pi_0(E)}\hp(H) \rightarrow G\times_{\pi_0(G)}\hp(H)$ is a relatively connected discrete extension of phase $\theta$ that admits a lift $(\rev{u};\eta_H)$ of $(u;\eta_H)$. Thus $\xi(\theta)$ vanishes by construction. 
\end{proof}

\begin{dfn}
	Let $\theta$ be a phase in $\Theta^*(G;X)$. The \emph{effective closure} $\overline{\theta}$ of $\theta$ is the inverse image by the quotient projection $\pi_1(G)\rightarrow \pi_1(G)/\theta$ of $\xi(\theta)(\Delta_2(H))$, the image of $\xi(\theta)$.
\end{dfn}

The naturality of (\ref{seq:Cohomo-Q-Cohomo}) directly imply the following property of the effective closure:

\begin{prop}\label{prop:idemp}
	If $\theta_1$ is contained in $\theta_2$, then $\overline{\theta}_2=\theta_2+\overline{\theta}_1$. In particular, the effective closure is an idempotent operation. 
\end{prop}

We can reformulate Proposition~\ref{prop:effective_phases} as follows:

\begin{prop}\label{prop:effective_closed}
	A phase $\theta$ is effective relatively to $X$ if and only if it is closed, i.e. $\theta$ equals $\overline{\theta}$.
\end{prop}

\begin{prop}
	$\overline{\theta}$ is the smallest effective phase containing $\theta$.
\end{prop}

\begin{proof}
	Let $\theta$ be a phase in $\Theta^*(G;X)$ and $\theta'$ be an effective phase that contains $\theta$. By Propositions~\ref{prop:idemp}~and~\ref{prop:effective_closed}, we have:
	\begin{equation*}
		\theta'= \bar{\theta}' = \theta'+\bar{\theta}.
	\end{equation*}
	Thus, $\theta'$ contains $\bar{\theta}$ which is effective by Proposition~\ref{prop:effective_closed}. 
\end{proof}

\begin{cor}\label{cor:phase}
	The phase $\theta_\textnormal{uni}$ of the universal covering of $X$ is given by the following formula:
	\begin{equation*}
		\theta_\textnormal{uni}=\bigcap_{\theta\in \Theta^*(G;X)}\overline{\theta}.
	\end{equation*}
	In particular, if $\chi(\hat{\pi}_0(H);\theta_\textnormal{top})$ vanishes, then $\theta_\textnormal{uni}$ is the effective closure of $\theta_\textnormal{top}$.
\end{cor}

\begin{proof}
	The phase $\theta_\textnormal{uni}$ must be an initial object in $\Theta(G;X)$, i.e. its minimum. The set $\Theta(G;X)$ can be parametrised by $\Theta^*(G;X)$ using the closure operator. Hence:
	\begin{equation*}
		\theta_\textnormal{uni}=\bigcap_{\theta\in \Theta(G;X)}\theta=\bigcap_{\theta\in \Theta^*(G;X)}\overline{\theta}.
	\end{equation*} 
	Moreover, if $\chi(\hat{\pi}_0(H);\theta_\textnormal{top})$ vanishes, then $\theta_\textnormal{top}$ is, by definition, the minimum of $\Theta^*(G;X)$. Since the closure operator is non-decreasing, $\theta_\textnormal{uni}$ is in this case the closure of $\theta_\textnormal{top}$.
\end{proof}

\begin{ex}
	Let $X$ be a toric variety. We consider the polar space associated to its complex locus $((\Sph^1)^n;X(\C))$, cf. Example~\ref{ex:tor_var1}. Since the group and the isotropy of the orbits are connected the magnitude plays no role here. Moreover, if $(\Sph^1)^k$ is included in $(\Sph^1)^l$ and $\pi_1((\Sph^1)^k)$ is contained in a subgroup $\theta$ of $\pi_1((\Sph^1)^l)$ then there is a unique group morphism $(\Sph^1)^k\rightarrow (\pi_1((\Sph^1)^l)\otimes_\Z \R)/\theta$ that lifts the inclusion of $(\Sph^1)^k$ in $(\Sph^1)^l$ along the covering $(\pi_1((\Sph^1)^l)\otimes_\Z \R)/\theta\rightarrow (\Sph^1)^l$. Thus, the algebraic obstruction will not play any role either. Only the topological obstruction matters here. We find that $\pi_1(X(\C))$ is isomorphic to $\pi_1((\Sph^1)^n)/\theta_\textnormal{top}$. In addition, $\theta_\textnormal{top}$ is nothing but the sum the images of the fundamental groups of the isotropies of the points of $X(\C)$. If we use the natural isomorphism between the fundamental group of a complex torus (which retracts onto its maximal compact subgroup) and its group of cocharacter we recover \cite[Theorem~9.1 and Proposition~9.3]{danilov_geometry_1978}.
\end{ex}

\paragraph{Extension of the Universal Covering and Fundamental Group}Once the phase $\theta_\textnormal{uni}$ is determined, one can determine the extension $G_\textnormal{uni}$ of the universal covering as follows:\begin{enumerate}
	\item We choose any relatively connected extension $E$ of $G\times_{\pi_0(G)}\hp(H)$ of phase $\theta_\textnormal{uni}$;
	\item We compute the class $[(u;\eta_H)/E]_\textnormal{alg}$ in $H^2_Q(H;(\pi_1(G)/\theta_\textnormal{uni})_Q)$;
	\item We find the unique $\beta$ in $H^2(\hp(H);\pi_1(G)/\theta_\textnormal{uni})$ whose pullback $\eta_H^*\beta$ equals $[(u;\eta_H)/E]_\textnormal{alg}$;
	\item We consider a shift $F$ of $E$ by $-\beta$;
	\item $G_\textnormal{uni}$ is the isomorphic to $\textnormal{pr}_G\circ p_F:F\rightarrow G\times_{\pi_0(G)}\hp(H) \rightarrow G$.
\end{enumerate}
Similarly, the class of the central extension:
	\begin{equation*}
		0\rightarrow \pi_1(G)/\theta_\textnormal{uni} \rightarrow \pi_1(X) \rightarrow \ker(\hp(H)\rightarrow G)\rightarrow 1,
	\end{equation*}
	in $H^2(\ker(\hp(H)\rightarrow G);\pi_1(G)/\theta_\textnormal{uni})$ is given by $\ker[\pr_G\circ p_E: E\rightarrow G]-\beta|_{\ker(\hp(H)\rightarrow G)}$, cf. Proposition~\ref{prop:kernel_discrete_ext}.

\begin{ex}[Klein Bottles]
	Let $G$ be a connected Lie group and $\sigma_0,\sigma_1$ two elements of order two. We can consider the cellular subgroup $H(\sigma_0;\sigma_1)$ of $G_{[0;1]}$ given by $(\langle\sigma_0\rangle\times\{0\}) \cup (0\times[0;1]) \cup(\langle\sigma_1\rangle\times\{1\})$ and the polar space $B(G;\sigma_0;\sigma_1)$ obtained as the quotient of $G_{[0;1]}$ by $H(\sigma_0;\sigma_1)$. We easily see that $B(G;\sigma_0;\sigma_1)$ has a structure of differentiable manifold. If $G$ is the circle, there is only one element of order two and $B(\Sph^1;-1;-1)$ is the usual Klein bottle. One can compute that $H(\sigma_0;\sigma_1)$, which as a $[0;1]$-group does not depend on $G$ or $\sigma_0,\sigma_1$, has trivial first singular homology group and trivial defect. Hence, $\theta_\textnormal{uni}$ is $0$ and the fundamental group of $B(G;\sigma_0;\sigma_1)$ is a central extension of $\pi_1(G)$ by $\Z/2*\Z/2$. Following the recipe described in the last paragraph, we can determine $G_\textnormal{uni}$. An extension of trivial phase of $G\times(\Z/2*\Z/2)$ is simply given by $G'\times(\Z/2*\Z/2)$ where $G'$ denotes the universal covering of $G$. To compute $[(u;\eta_H)/G'\times(\Z/2*\Z/2)]_\textnormal{alg}$, we may use (\ref{eq:zeta}). Let $\sigma'_0$ (resp. $\sigma'_1$) be a lift of $\sigma_0$ (resp. $\sigma_1$) in $G'$. Then, we denote by $\lambda_0$ (resp. $\lambda_1$) the unique element of $\pi_1(G)$ that satisfies:
	\begin{equation*}
		(\sigma_0')^2=e^{\lambda_0} \quad (\textnormal{resp.}\; (\sigma_1')^2=e^{\lambda_1}).
	\end{equation*}
	Now let $r:\Z/2\rightarrow \Z$ denote the function sending $0$ to $0$ and $1$ to $1$. If $k$ belongs to $\Z/2*\Z/2$, we denote by $|k|_0$ (resp. $|k|_1$) the image of $k$ by the unique projection $\Z/2*\Z/2\rightarrow \Z/2$ that sends the second (resp. first) generator to $0$. Then some group cohomology computations\footnote{We use the fact that if $M$ is a trivial $(\Z/2*\Z/2)$-module the sum of the pullbacks $H^2(\Z/2;M)\oplus H^2(\Z/2;M)\rightarrow H^2(\Z/2*\Z/2;M)$ is an isomorphism. The inverse morphism is constructed with the two projections $\Z/2*\Z/2\rightarrow \Z/2$ each sending a generator to 0.} show that the $\eta_G$-pullback of the class $\beta$ of the cocycle:
	\begin{equation*}
		(k;l)\mapsto r(|k|_0|l|_0)\lambda_0 + r(|k|_1|l|_1)\lambda_1,
	\end{equation*}
	equals $[(u;\eta_H)/G'\times(\Z/2*\Z/2)]_\textnormal{alg}$ (here the product in $r$ means the actual product of the ring $\Z/2$). Hence, the group law of $G_\textnormal{uni}$ on $G'\times(\Z/2*\Z/2)$ is given by the following formula:
	\begin{equation*}
		(g;k)\cdot(h;l)\coloneqq (ghe^{r(|k|_0|l|_0)\lambda_0 + r(|k|_1|l|_1)\lambda_1};kl).
	\end{equation*}
	The one of the fundamental group $\pi_1(B(G;\sigma_0;\sigma_1))$ on $\pi_1(G)\times(\Z/2*\Z/2)$ is:
	\begin{equation*}
		(v;k)\cdot(w;l)\coloneqq (v+w+{r(|k|_0|l|_0)\lambda_0 + r(|k|_1|l|_1)\lambda_1};kl).
	\end{equation*}
\end{ex}

\bibliographystyle{apalike}
\bibliography{PolarSpaces}
\end{document}